\documentclass[11pt]{amsart}

\usepackage[utf8]{inputenc}
\usepackage[english]{babel}
\usepackage[a4paper,twoside,top=1.2in, bottom=1.1in, left=0.8in, right=0.8in]{geometry}

\usepackage{graphicx}
\usepackage{caption} % Better handling of empty figure captions.
\usepackage{amsmath}
\usepackage{amsthm}
\usepackage{amssymb}
\usepackage{esint} % For the average integral
\usepackage{mathrsfs} % Just for cool sigma-algebra
\usepackage{dsfont} % For identity
\usepackage{xcolor}
\usepackage{array}
\usepackage{hhline}
\usepackage{enumitem} % Customize enumerate and itemize
\usepackage{comment} % For comment environment
\usepackage[toc,page]{appendix} % Appendices
\usepackage{xparse} % In order to use ExplSyntax (in older latex versions)
\usepackage{mathtools}
\usepackage{fancyhdr} % Header and footer of pages
\usepackage{ifthen} % Ifthen construct
\usepackage{forloop} % forloop construct
\usepackage{xstring}
\usepackage{emptypage} % Leave completely empty the empty pages
\usepackage{setspace}
\usepackage[initials]{amsrefs} % For alphanumeric labels: alphabetic

\usepackage{tikz}
\usetikzlibrary{arrows,shapes,patterns,calc,fadings,decorations.pathreplacing,decorations.markings}

\usepackage{hyperref} % References become hyperlinks.
\hypersetup{
	colorlinks = true,
	linkcolor = {blue},
	urlcolor = {red},
	citecolor = {blue}
}

\usepackage[nameinlink,capitalise]{cleveref} % References show what kind of thing is being referenced (theorem, lemma, proposition...)

\newcounter{results}[section] % Uniform counters for lemmas, theorems, propositions etc

\theoremstyle{plain}
\newtheorem{theorem}[results]{Theorem}
\newtheorem{lemma}[results]{Lemma}
\newtheorem{proposition}[results]{Proposition}
\newtheorem{corollary}[results]{Corollary}

\newtheorem*{theorem*}{Theorem}
\newtheorem*{lemma*}{Lemma}
\newtheorem*{proposition*}{Proposition}
\newtheorem*{corollary*}{Corollary}
\newtheorem*{exercise*}{Exercise}
\newtheorem*{fact*}{Fact}

\theoremstyle{remark}
\newtheorem{remark}[results]{Remark}

\newtheorem*{remark*}{Remark}
\newtheorem*{question*}{Question}

\theoremstyle{definition}
\newtheorem{definition}[results]{Definition}
\newtheorem{example}[results]{Example}

\newtheorem*{definition*}{Definition}
\newtheorem*{example*}{Example}

\numberwithin{equation}{section}

\crefname{figure}{Figure}{Figures}

\let\crefnoname\cref
\makeatletter
\ExplSyntaxOn
\newcommand{\myref}[1]{
	\IfSubStr{#1}{,}{
		\crefnoname{#1}%
	}
	{
		\ifcsname r@#1\endcsname% La reference è definita?
			\edef\@RefInfo{\csname r@#1\endcsname}% Contiene tutte le informazioni sulla reference
			
\edef\@CompleteCounter{\expandafter\@fourthoffive\@RefInfo}% Contiene il nome espanso del contatore
			
			\expandafter\IfBeginWith{\@CompleteCounter}{equation}{% La reference è un'equazione?
				\crefnoname{#1}%
			}
			{
			\expandafter\IfBeginWith{\@CompleteCounter}{figure}{% La reference è una figura?
				\crefnoname{#1}%
			}
			{
				
\edef\@RefName{\expandafter\@thirdoffive\@RefInfo}% Contiene, se definito, il nome della reference
				\StrLen{\@RefName}[\@RefNameLen]
				
				\ifthenelse{\@RefNameLen<2}{% Il nome è vuoto o di un solo carattere?
					\crefnoname{#1}%
				}
				{
					\crefnoname{#1}~(\nameref*{#1})%
				}
			}
			}
		\else
			\crefnoname{#1}
		\fi
	}
}
\renewcommand{\cref}{\myref}
\ExplSyntaxOff
\makeatother

\ifdefined\comma 
        \renewcommand{\comma}{\ensuremath{\, \text{, }}}
\else 
        \newcommand{\comma}{\ensuremath{\, \text{, }}}
\fi
\newcommand{\semicolon}{\ensuremath{\, \text{; }}}
\newcommand{\point}{\ensuremath{\, \text{. }}}
\newcommand{\pheq}{\ensuremath{\hphantom{{}={}}}}

\newcommand{\N}{\ensuremath{\mathbb N}}%Natural numbers
\newcommand{\R}{\ensuremath{\mathbb R}}%Real numbers
\DeclarePairedDelimiter\ClOp{[}{)}
\DeclarePairedDelimiter\OpOp{(}{)}
\DeclarePairedDelimiter\ClCl{[}{]}

\newcommand{\co}[2]{\ensuremath{\ClOp{#1, #2}}}
\newcommand{\oo}[2]{\ensuremath{\OpOp{#1, #2}}}
\newcommand{\cc}[2]{\ensuremath{\ClCl{#1, #2}}}

\DeclarePairedDelimiter\abs{\lvert}{\rvert} % Absolute value
\DeclarePairedDelimiter\norm{\lVert}{\rVert} % Norm
\newcommand{\scal}[2]{\ensuremath{\langle #1 , #2 \rangle}} % Scalar product 

\newcommand \eps{\ensuremath{\varepsilon}}
\newcommand{\st}{\ensuremath{\ :\ }} % Such that in formulae.
\newcommand{\eqdef}{\ensuremath{\coloneqq}} % Equal in a definition.
\DeclareMathOperator{\supp}{supp}

\renewcommand{\d}{\ensuremath{d}} % differential
\newcommand{\de}{\ensuremath{\, d}} % differential in integrals

\let\div\undefined
\newcommand{\div}{\ensuremath{\mathrm{div}}} % Divergence
\newcommand{\grad}{\ensuremath{\nabla}} % Gradient
\newcommand{\lapl}{\ensuremath{\Delta}} % Laplacian
\newcommand{\DerParz}[2]{\ensuremath{\frac{\partial #1}{\partial #2}}} % Partial derivative
\newcommand{\vf}{\ensuremath{\mathfrak X}} % Vector fields space
\newcommand{\cov}{\ensuremath{\nabla}} %Covariant derivative
\DeclareMathOperator{\II}{I\!I} % Second fundamental form
\newcommand{\Scal}{\ensuremath{R}} % Scalar curvature
\DeclareMathOperator{\Ric}{Ric} % Ricci tensor
\newcommand{\Haus}{\ensuremath{\mathscr H}} % Hausdorff measure
\newcommand{\ms}{\ensuremath{\Sigma}} % Minimal submanifold
\newcommand{\amb}{\ensuremath{M}} % Ambient manifold
\DeclareMathOperator{\ind}{ind} % Index

\newcommand{\hypP}{\hyperref[HypP]{$(\mathfrak{P})$}} % No bad fbms
\newcommand{\hypC}{\hyperref[HypC]{$(\mathfrak{C})$}} % Mean convexity of the boundary
\newcommand{\hypN}{\hyperref[HypN]{$(\mathfrak{N})$}} % Local picture

\newcommand{\lam}{\ensuremath{\mathcal L}} % Lamination
\newcommand{\set}{\ensuremath{\mathcal S}} % Set of points
\newcommand{\area}{\ensuremath{\Haus^2}} % Area
\newcommand{\length}{\ensuremath{\Haus^1}} % Length
\DeclareMathOperator{\genus}{genus} % Genus
\DeclareMathOperator{\bdry}{boundaries} % Number of boundary components
\newcommand{\Diff}{\ensuremath{D}} % Hessian
\newcommand{\jac}{\ensuremath{J}} % Jacobian operator
\newcommand{\A}{\ensuremath{A}} % Second fundamental form

\newcommand{\bu}[1]{\ensuremath{\bar{#1}}}

\newcommand{\notg}{\ensuremath{a}}
\newcommand{\notb}{\ensuremath{b}}

\newcommand{\euno}{\ensuremath{\scriptscriptstyle{(1)}}}
\newcommand{\edue}{\ensuremath{\scriptscriptstyle{(2)}}}
\newcommand{\ebi}{\ensuremath{\scriptscriptstyle{(\notb)}}}

\colorlet{myGray}{gray}
\colorlet{myBlue}{blue}
\colorlet{myBlack}{black}
\colorlet{myBackground}{gray!10}

\allowdisplaybreaks

\begin{document}

\title[Inequivalent complexity criteria for free boundary minimal surfaces]{Inequivalent complexity criteria\\ for free boundary minimal surfaces}

\author{Alessandro Carlotto and Giada Franz}
     \address{ \noindent Alessandro Carlotto: 
     	\newline ETH D-Math, R\"amistrasse 101, 8092 Z\"urich, Switzerland 
     	\newline IAS, 1 Einstein drive, 08540 Princeton, United States of America
     	\newline
     	 \textit{E-mail address: alessandro.carlotto@math.ethz.ch, alessandro.carlotto@ias.edu} 
     	 \newline \newline \indent Giada Franz: 
     	\newline ETH D-Math, R\"amistrasse 101, 8092 Z\"urich, Switzerland 
     	 	\newline
     	 \textit{E-mail address: giada.franz@math.ethz.ch} 
}

\begin{abstract}
We obtain a series of results in the global theory of free boundary minimal surfaces, which in particular provide a rather complete picture for the way different \emph{complexity criteria}, such as area, topology and Morse index compare, beyond the regime where effective estimates are at disposal. 
\end{abstract}

\maketitle

\thispagestyle{empty}

%%%%%%%%%%%%%%%%%%%%%%%%%%%%%%%%%%%%%%%%%%%%%%%%%%%%%%%%%%%%%%%%
%%% Introduction
%%%%%%%%%%%%%%%%%%%%%%%%%%%%%%%%%%%%%%%%%%%%%%%%%%%%%%%%%%%%%%%%

\begin{spacing}{1.09}

\section{Introduction}

Free boundary minimal surfaces naturally arise, in Riemannian Geometry, as critical points of the area functional in the category of relative cycles. More precisely, if $(M^n,g)$ is a Riemannian manifold with boundary, and one considers the class of deformations induced by proper diffeomorphisms (that is to say: compactly supported diffeomorphisms that map the boundary $\partial M$ onto itself) the first variation of the area functional at $\Sigma^k$ vanishes if and only if
the submanifold in question has zero mean curvature and meets the boundary of the ambient manifold orthogonally: if that is the case $\Sigma$ shall be called a \emph{free boundary minimal submanifold} (specified to \emph{surface} when $k=2$).

In recent years, the series of works by A. Fraser and R. Schoen \cite{FraSch11}, \cite{FraSch13}, \cite{FraSch16} on the relation between free boundary minimal surfaces and the Steklov eigenvalues has breathed new life into the study of these objects, whose investigation goes back almost one century. The theory turns out to be extremely rich already in the simplest case of surfaces in the three-dimensional Euclidean unit ball, where many different examples have been discovered. We refer the reader to the introduction of \cite{AmbBuzCarSha18} and to the survey \cite{Li19} for a gallery of recent existence theorems, but we shall mention here the significant undergoing project, by Q. Guang, M. Li, Z. Wang and X. Zhou (see in particular \cite{LiZho16} and \cite{GuaLiWanZho19}) to \emph{transfer} the min-max theory by Almgren-Pitts to this setting, with the perspective of transposing the impressive results that have been achieved in the closed case (through the efforts of  F. Marques, A. Neves,  Y. Liokumovich, D. Ketover, K. Irie and A. Song among others).

These striking developments pose a number of challenges. Among those, it is natural to ask how different pieces of information that one can associate to a free boundary minimal surface relate to each other. In this article, we will primarily focus on three sources of data: the \emph{Euler characteristic} (as a topological descriptor, cf. \crefnoname{sec:EulChar}), the \emph{area} (as a measure of geometric size) and the \emph{Morse index} (that plays the role of the most basic analytic invariant one can associate to the surface in question). See \crefnoname{subsec:FBMS} for a precise definition and \crefnoname{sec:Properness} for a discussion on the role of properness in that respect.
 
 The general scope of the present work is to investigate how these different \emph{complexity criteria} can be compared to each other, under mild \emph{positive curvature} assumptions, i.e. in a regime where effective/quantitative estimates are typically not available. More precisely, the context to have in mind is that of a compact 3-manifold satisfying either of the following two pairs of curvature conditions:
\begin{enumerate}[label={\normalfont(\roman*)}]
\item the scalar curvature of $\amb$ is positive and $\partial \amb$ is mean convex with no minimal components; \label{ch:PosScal}
\item the scalar curvature of $\amb$ is non-negative and $\partial\amb$ is strictly mean convex. \label{ch:PosMean}
\end{enumerate}
First of all, we consider questions of this sort: `Given a Riemannian 3-manifold $(M,g)$ as above, is it possible to construct a monotone function $f$ such that any free boundary minimal surface whose index is bounded by $C$ has area bounded by $f(C)$?'
We can combine two of the main results we present here with other recent advances in the field to fully determine whether each of the six natural implications that one can associate to the three pieces of data above hold true, thereby obtaining a rather complete description of the scenario in front of us.

\begin{figure}[htpb]
\centering
%%%%%%%%%%%%%%%%%%%%%%%%%%%%%%%%%%%%%%%%%%%%%%%%%%%%%
%%% complexity_diagram.tikz
%%%%%%%%%%%%%%%%%%%%%%%%%%%%%%%%%%%%%%%%%%%%%%%%%%%%%
\begin{tikzpicture}[scale=1.6]
\tikzset{bharrownegated/.style={
        decoration={markings,
            mark= at position 0.5 with {
                \node[transform shape, scale=1.7] (tempnode) {$\times$};
            },
            mark=at position 1 with {\arrow[scale=2]{>}}
        },
        postaction={decorate},
    }
}
\tikzset{bharrow/.style={
    decoration={markings,mark=at position 1 with {\arrow[scale=2]{>}}},
    postaction={decorate},
    }
}
\begin{scope}[every node/.style = {
    shape = rectangle,
    minimum width= 3.8cm,
    minimum height=1.5cm,
    align=center}]
\node (top) at (0,1) [draw=gray] {topological bounds\\ \footnotesize $\chi(\ms)$};
\node (ind) at (-4,-2.5) [draw=gray] {index bounds\\ \footnotesize $\ind(\ms)$};
\node (are) at (4,-2.5) [draw=gray] {area bounds\\ \footnotesize $\area(\ms)$};
\end{scope}
\draw[bharrow] ([xshift=1mm,yshift=-1.5mm]ind.east) -- node[sloped,midway,below=2mm] {\cref{thm:AreaBound}} ([xshift=-1mm,yshift=-1.5mm]are.west);
\draw[bharrownegated] ([xshift=-1mm,yshift=1.5mm]are.west) -- node[sloped,midway,above=2mm] {\cref{rmk:AreaNoControl} (cf. \cite{FraSch16}, \cite{KapLi17}, \cite{KapWiy17})} ([xshift=1mm,yshift=1.5mm]ind.east);
\draw[bharrow] ([xshift=-2mm,yshift=1mm]ind.north) -- node[sloped,midway,above=2mm] {\cref{thm:AreaBound} $\oplus$ \cite[Corollary 4]{AmbBuzCarSha18}} ([xshift=-2mm,yshift=-1mm]top.south west);
\draw[bharrownegated] ([xshift=3mm,yshift=-1mm]top.south west) -- node[sloped,midway,below=2mm] {\cref{thm:Counterexample}} ([xshift=3mm,yshift=1mm]ind.north);
\draw[bharrownegated] ([xshift=-3mm,yshift=-1mm]top.south east) -- node[sloped,midway,below=2mm] {\cref{thm:Counterexample}} ([xshift=-3mm,yshift=1mm]are.north);
\draw[bharrownegated]([xshift=2mm,yshift=1mm]are.north) -- node [sloped,midway,above=2mm] {\cref{rmk:AreaNoControl} (cf. \cite{FraSch16}, \cite{KapLi17}, \cite{KapWiy17})} ([xshift=2mm,yshift=-1mm]top.south east);
\end{tikzpicture}
%%%%%%%%%%%%%%%%%%%%%%%%%%%%%%%%%%%%%%%%%%%%%%%%%%%%%
%%%%%%%%%%%%%%%%%%%%%%%%%%%%%%%%%%%%%%%%%%%%%%%%%%%%%
\caption{A diagram comparing complexity criteria for compact Riemannian 3-manifolds $(M,g)$ satisfying \emph{either} $\Scal_g>0$, $H^{\partial\amb}\ge 0$ \emph{or} $\Scal_g\ge0$, $H^{\partial\amb}> 0$.}\label{fig:Complexity}
\end{figure}
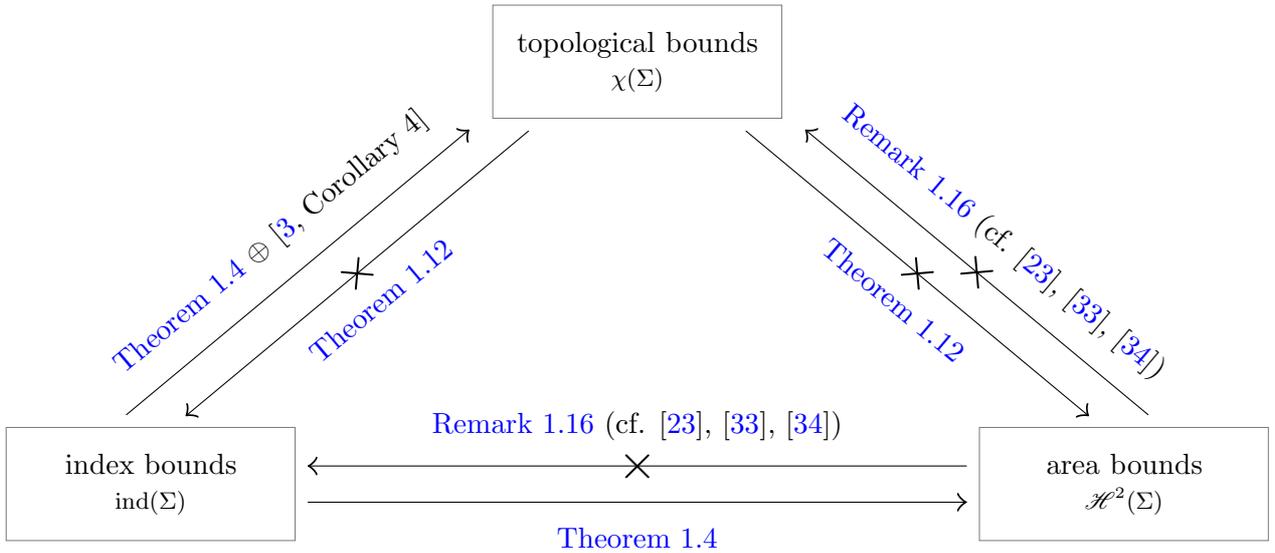

In particular, we develop a detailed analysis of the topological degenerations that may occur, in the limit, to sequences of free boundary minimal surfaces solely subject to a uniform Morse index bound to ultimately prove that \emph{a bound on the index implies a bound on the area, the topology and the total curvature} (see \cref{thm:AreaBound} for a precise statement). In turn, this theorem implies novel unconditional compactness (\cref{cor:CptBoundIndPosRic}) and generic finiteness (\cref{cor:GenericFin}) results. We note that, prior to this work, no (effective or ineffective) counterpart of \cref{thm:AreaBound} was known for free boundary minimal surfaces, not even under the stronger curvature assumptions that the Ricci curvature of the ambient manifold be positive, and its boundary strictly convex (see also \cref{rmk:FraserLicomment} and \cref{rmk:Hersch}).

To show that a bound on the topology cannot possibly imply a bound on the area, nor on the index, we build a large class of pathological counterexamples: in \cref{thm:Counterexample} we construct, for any smooth 3-manifold $M$ supporting Riemannian metrics of positive scalar curvature and mean convex boundary and any $a\geq 0, b>0$, one such Riemannian metric $g=g(a,b)$ in a way that $(M,g)$ contains a sequence of connected, embedded, free boundary minimal surfaces of genus $a$ and exactly $b$ boundary components, and whose area and Morse index attain arbitrarily large values. In fact, we have some freedom on the geometric boundary conditions we impose, so that a few variants of the construction are actually possible. 

Concerning the last two (possible) implications above, we note how a bound on the area cannot possibly imply a bound on the index, nor on the topology (no matter how strong curvature conditions are imposed). To that scope, we examine in \cref{rmk:AreaNoControl} three different classes of existing examples (due to Fraser-Schoen \cite{FraSch16}, Kapouleas-Li \cite{KapLi17} and Kapouleas-Wiygul \cite{KapWiy17}): in each case one has a sequence of free boundary minimal surfaces in the unit ball of $\R^3$ satisfying a uniform bound on the area, but arbitrarily large genus and hence, by \crefnoname{thm:TopFromArea}, arbitrarily large index as well. Roughly speaking, a bound on the area only implies weaker forms of convergence, typically of measure-theoretic character (e.g. in the sense of varifolds, flat chains, currents \ldots), but does not capture finer geometric properties.

This diagram then implicitly defines a hierarchy of conditions, based on the implications that hold or do not hold true. Of course, it is then natural to ask whether \emph{pairs} of `weak conditions' imply a `strong' one, e.g. prototypically whether a bound on the area and the topology implies a bound on the index. This interesting question has recently been answered, in the affirmative, by V. Lima \cite[Theorem B]{Lim17} who adapted to the free boundary setting some remarkable estimates by Ejiri-Micallef \cite{EjiMic08}: one can bound the Morse index from above by a linear function of the area and the Euler characteristic, with a multiplicative constant only depending on the ambient manifold. Thereby, the picture we obtain is quite exhaustive and final.

We will now present the contents of the article in more detail, point out the technical challenges we faced and relate them to the pre-existing results in the literature, some of which played a fundamental role with respect to this project.

\subsection{Topological degeneration analysis}

A good starting point for our discussion is the following result of Fraser-Li, which is the free boundary analogue of the classical compactness theorem \cite[Theorem 1]{ChoSch85} by Choi-Schoen.
\begin{theorem}[{\cite[Theorem 1.2]{FraLi14}}]
\label{thm:FraLiCpt}
Let $(\amb^3,g)$ be a compact Riemannian manifold with non-empty boundary. Suppose that $\amb$ has non-negative Ricci curvature and strictly convex boundary. Then the space of compact, properly embedded, free boundary minimal surfaces of fixed topological type in $\amb$ is compact in the $C^k$ topology for any $k\ge 2$.
\end{theorem}

Afterwards, a general investigation of the spaces of free boundary minimal hypersurfaces with bounded index and volume has been carried through in \cite{AmbCarSha18-Compactness} and \cite{AmbBuzCarSha18}. In absence of any curvature assumptions, it is not possible to obtain a strong compactness result as in \crefnoname{thm:FraLiCpt}, since curvature concentration can occur at certain (isolated) points. However, one can still prove a milder form of subsequential convergence (smooth, graphical convergence with multiplicity $m\geq 1$ away from finitely many points), and develop an accurate blow-up analysis near the points of bad convergence.
This analysis leads to several compactness and finiteness results, among which we want to recall the following theorem.
\begin{theorem}[{\cite[Corollary 4]{AmbBuzCarSha18}}] \label{thm:TopFromArea}
Let $(\amb^3,g)$ be a compact Riemannian manifold with boundary and consider $I\in \N$, $\Lambda\ge 0$. Then there exist constants $a_0=a_0(\amb,g,I,\Lambda)$ and $b_0 = b_0(\amb,g,I,\Lambda)$ such that every compact, properly embedded, free boundary minimal surface with index bounded by $I$ and area bounded by $\Lambda$ has genus bounded by $a_0$ and number of boundary components bounded by $b_0$. Furthermore, there exists a constant $\tau_0=\tau_0(\amb,g,I,\Lambda)$ so that the total curvature (i.e. the integral of the square length of the second fundamental form) of any such surface is bounded from above by $\tau_0$. 
\end{theorem}

\begin{remark}
The statement in \cite{AmbBuzCarSha18} is more general as it only requires a bound on some eigenvalue of the Jacobi operator instead of an index bound, and it applies to free boundary minimal hypersurfaces in ambient manifolds of dimension $3\leq n+1\leq 7$.
\end{remark}

Here we shall be concerned with the space of free boundary minimal surfaces (inside a three-dimensional Riemannian manifold) with bounded index but without any a priori bound on the area. The analogous task has been carried out, for the case of \emph{closed} minimal surfaces, in the remarkable article \cite{ChoKetMax17} by Chodosh-Ketover-Maximo. Some of the methods developed there are essential for our analysis, and we often rely on the results presented in that article for interior points, although serious technical work (and specific tools) are needed to properly handle the possible degenerations occurring near the boundary of the ambient manifold. 

Let us now describe the key steps in this approach. Hence, let us consider a compact Riemannian manifold with non-empty boundary $(\amb^3,g)$. For the sake of simplicity, let us assume (in the context of this introduction) the following additional property:
\begin{description}
\item [$(\mathfrak{P})$]\label{HypP} \ \ If $\ms^2\subset \amb$ is a smooth, connected, complete (possibly non-compact), embedded surface with zero mean curvature which meets the boundary of the ambient manifold orthogonally along its own boundary, then $\partial\ms = \ms\cap\partial\amb$. 
\end{description}
Such a condition prevents the existence of minimal surfaces that touch $\partial\amb$ in their interior and it is implied by simple geometric assumptions (for example property \hypC{} in \crefnoname{sec:NotDef}, namely that the boundary of the ambient manifold be mean convex with no minimal component).
We postpone the discussion of this property, its relevance and the issues that arise when one drops it to \crefnoname{sec:NotDef} and \crefnoname{sec:Properness}.
In particular, in that appendix we point out how far from trivial the `properness issues' are: if one allows for an interior contact set, then there are at least four natural definitions of Morse index one can adopt, and the value one computes depends not only on the adopted definition but also, even for a given definition, on the size of the contact set.

Fixed $I\in\N$ and given a sequence of compact, properly embedded, free boundary minimal surfaces $\ms_j^2\subset\amb$ with $\ind(\ms_j)\le I$, we develop our analysis in two steps:
\begin{description}
\item [Macroscopic behavior]  First we prove that, up to subsequence, the surfaces $\ms_j$ converge locally smoothly away from a finite set of points $\set_\infty$ to a smooth free boundary minimal lamination $\lam\subset\amb$, that is a suitable disjoint union of free boundary minimal surfaces (see \cref{def:fbmlam}). 
Note that we cannot expect anything better than a lamination without imposing uniform bounds on the area.
Moreover, it holds that the curvature of $\ms_j$ is locally uniformly bounded away from $\set_\infty$.
This tells us that the surfaces are well-controlled away from a finite set of points.
Let us remark that the points in $\set_\infty$ can belong to the boundary.

The main ingredient here is an extension of the curvature estimate for stable minimal surfaces (cf. \cite{Sch83_estimates}, \cite{Whi87} and \cite{SchSim81} for higher dimensions) to free boundary minimal surfaces with bounded index.

\item [Microscopic behavior] The second step consists in carefully studying the local behavior of the surfaces $\ms_j$ near the `bad' points in $\set_\infty$. In particular, we prove that, for $\eps>0$ sufficiently small, $\ms_j\cap B_{\eps}(\set_\infty)$ contains only a finite number of components where the curvature is not bounded and these components have controlled topology and area.

The way to proceed in the proof is to blow-up the components of $\ms_j\cap B_{\eps}(\set_\infty)$ with unbounded curvature at the `scale of the curvature'. In this way we obtain a (free boundary) minimal surface in $\R^3$ or in a half-space of $\R^3$ (if initially we were near the boundary of $\amb$) with index less or equal than $I$. 
Thanks to \cite{ChoMax16} (cf. also \cite{ChoMax18}), possibly combined with the theory developed in Section 2 of \cite{AmbBuzCarSha18}, we are able to conclude the description of $\ms_j$ at small scale, paying attention to prevent data loss in the blow-up.
\end{description}

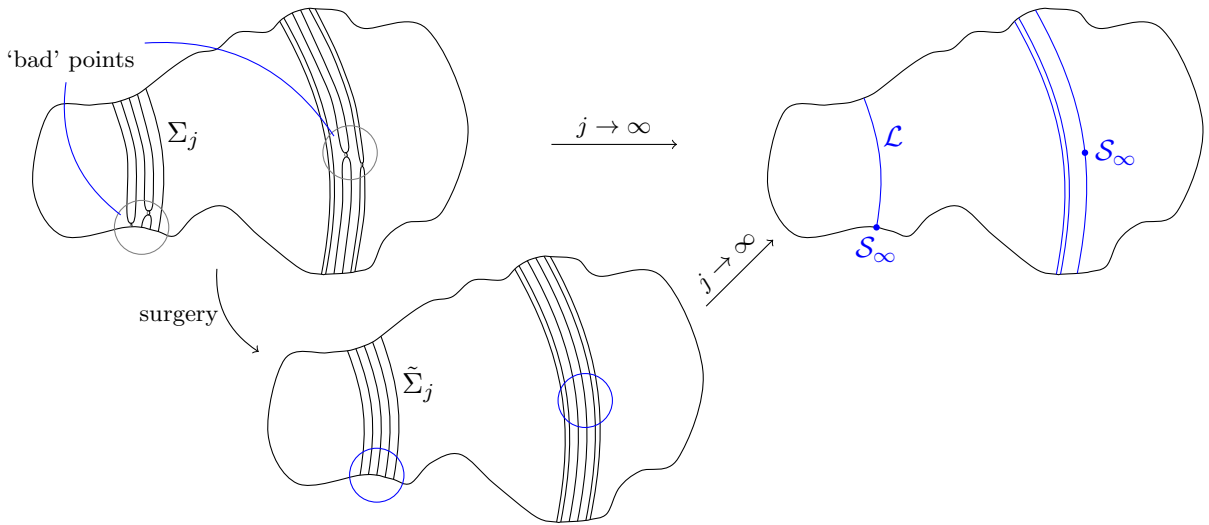
\begin{figure*}[htpb]
\centering
%%%%%%%%%%%%%%%%%%%%%%%%%%%%%%%%%%%%%%%%%%%%%%%%%%%%%
%%% intro.tikz
%%%%%%%%%%%%%%%%%%%%%%%%%%%%%%%%%%%%%%%%%%%%%%%%%%%%%
\begin{tikzpicture}[scale =0.55]
\coordinate (A1) at (0,0); 
\coordinate (B1) at (0.2,0.05); 
\coordinate (C1) at (0.4,0.12);
\coordinate (D1) at (0.6,0.21);
\coordinate (E1) at (0.8,0.35);

\coordinate (F1) at (4,1.97);
\coordinate (G1) at (4.1,2.05);
\coordinate (H1) at (4.3,2.17);
\coordinate (I1) at (4.5,2.24);
\coordinate (J1) at (4.7,2.28);
\coordinate (K1) at (4.85,2.27);

\coordinate (F2) at (5,-4.13);
\coordinate (G2) at (5.1,-4.13);
\coordinate (H2) at (5.3,-4.12);
\coordinate (I2) at (5.5,-4.11);
\coordinate (J2) at (5.7,-4.09);
\coordinate (K2) at (5.8,-4.07);

\coordinate (C2) at (0.7,-3);
\coordinate (D2) at (0.9,-3.05);
\coordinate (E2) at (1.1,-3.1);

\coordinate (X) at (-1.5,-0.2);
\coordinate (Y) at (-1.2, -3.2);
\coordinate (U) at (7.5, 1.8);
\coordinate (V) at (7.9, -2.6);

\draw plot [smooth cycle, tension=0.6] coordinates {(A1) (B1) (C1) (D1) (E1)
			      (1.5,0.9) (2,1.2) (2.5,1.3) (3,1.7) (3.7,1.71)
			      (F1) (G1) (H1) (I1) (J1) (K1)
			      (5.5,2.2) (6,1.8)
			      (U) (8.5,-0.5) (V)
			      (7,-2.9) (6.5, -3.3) (6.2, -3.93) 
			      (K2) (J2) (I2) (H2) (G2) (F2)
			      (4.5,-4) (3.5,-3.2) (2.7, -2.5) (2,-2.7) (1.5, -3.2) 
			      (E2) (D2) (C2)
			      (0.3, -3)
			      (Y) (-1.9,-2) (X) (-0.5,-0.05)};

\draw (F1) to [bend left=20] (F2);
\draw (G1) to [bend left=20] (G2);

\draw plot [smooth, tension = 0.5] coordinates {(H1) (5.2,0.2) (5.45,-1) (5.57, -1.2) (5.65, -1) (5.4,0.2) (I1)};
\draw plot [smooth, tension = 0.5] coordinates {(H2) (5.5,-2.5) (5.5,-1.5) (5.6, -1.3) (5.7, -1.5) (5.7,-2.5) (I2)};
\pgfmathsetmacro{\t}{atan(0.03/0.1)}
\begin{scope}[rotate around={\t:(5.57,-1.2)}]
\draw[gray] (5.57, -1.2) arc (90:270:0.03 and {sqrt(0.03*0.03+0.1*0.1)/2});
\draw (5.57, -1.2) arc (90:-90:0.03 and {sqrt(0.03*0.03+0.1*0.1)/2});
\end{scope}

\draw plot [smooth, tension = 0.5] coordinates {(J1) (5.25,1.2) (5.68,0) (5.9,-1.29) (5.97, -1.45) (6, -1.27) (5.8,0) (5.37,1.2) (K1)};
\draw plot [smooth, tension = 0.5] coordinates {(J2)  (5.9,-2.7) (5.93, -1.7) (5.97, -1.53) (6.03,-1.7) (6.02, -2.7) (K2)};
\draw[gray] (5.97, -1.45) arc (90:270:0.02 and {0.04});
\draw (5.97, -1.45) arc (90:-90:0.02 and {0.04});

\draw[gray] (0.45, -2.9) arc (90:182:0.02 and {0.08});
\draw (0.45, -2.9) arc (90:0:0.02 and {0.08});
\draw plot [smooth, tension = 0.55] coordinates {(A1) (0.35,-1.2) (0.35,-2.65) (0.45, -2.9) (0.55, -2.65) (0.55,-1.2) (B1)};

\pgfmathsetmacro{\t}{atan(0.01/0.1)}
\begin{scope}[rotate around={-\t:(0.87, -2.6)}]
\draw[gray] (0.87,-2.6) arc (90:270:0.03 and {sqrt(0.01*0.01+0.1*0.1)/2});
\draw (0.87,-2.6) arc (90:-90:0.03 and {sqrt(0.01*0.01+0.1*0.1)/2});
\end{scope}
\draw plot [smooth, tension = 0.55] coordinates {(C1) (0.75,-1.2) (0.77,-2.4) (0.87, -2.6) (0.97, -2.4) (0.95,-1.2) (D1)};
\draw plot [smooth, tension = 0.9] coordinates {(C2) (0.75, -2.8) (0.86, -2.7) (0.93, -2.8) (D2)};

\draw (E1) to [bend left = 15] (E2);

\pgfmathsetmacro\R{0.65}
\draw[gray] (C2) circle (\R);
\draw[gray] (5.685, -1.2) circle (\R);

\node(bp) at (-1,1) {\footnotesize `bad' points};
\draw[blue] (bp) to [bend right = 30] (0.2,-2.7);
\draw[blue] (bp) to [bend left=30] (5.3,-0.8);

\node at (1.7,-0.8) {$\ms_j$};

\draw[->] (10.5,-1) -- (13.5,-1);
\node at (12,-0.6) {\footnotesize $j\to\infty$};

\begin{scope}[xshift=500]
\coordinate (A1) at (0,0); 
\coordinate (B1) at (0.2,0.05); 
\coordinate (C1) at (0.4,0.12);
\coordinate (D1) at (0.6,0.21);
\coordinate (E1) at (0.8,0.35);

\coordinate (F1) at (4,1.97);
\coordinate (G1) at (4.1,2.05);
\coordinate (H1) at (4.3,2.17);
\coordinate (I1) at (4.5,2.24);
\coordinate (J1) at (4.7,2.28);
\coordinate (K1) at (4.85,2.27);

\coordinate (F2) at (5,-4.13);
\coordinate (G2) at (5.1,-4.13);
\coordinate (H2) at (5.3,-4.12);
\coordinate (I2) at (5.5,-4.11);
\coordinate (J2) at (5.7,-4.09);
\coordinate (K2) at (5.8,-4.07);

\coordinate (C2) at (0.7,-3);
\coordinate (D2) at (0.9,-3.05);
\coordinate (E2) at (1.1,-3.1);

\coordinate (X) at (-1.5,-0.2);
\coordinate (Y) at (-1.2, -3.2);
\coordinate (U) at (7.5, 1.8);
\coordinate (V) at (7.9, -2.6);

\draw plot [smooth cycle, tension=0.6] coordinates {(A1) (B1) (C1) (D1) (E1)
			      (1.5,0.9) (2,1.2) (2.5,1.3) (3,1.7) (3.7,1.71)
			      (F1) (G1) (H1) (I1) (J1) (K1)
			      (5.5,2.2) (6,1.8)
			      (U) (8.5,-0.5) (V)
			      (7,-2.9) (6.5, -3.3) (6.2, -3.93) 
			      (K2) (J2) (I2) (H2) (G2) (F2)
			      (4.5,-4) (3.5,-3.2) (2.7, -2.5) (2,-2.7) (1.5, -3.2) 
			      (E2) (D2) (C2)
			      (0.3, -3)
			      (Y) (-1.9,-2) (X) (-0.5,-0.05)};
			      
\draw[blue] (C1) to [bend left = 15] (C2);
\draw[blue] (F1) to [bend left = 20] (F2);
\draw[blue] (G1) to [bend left = 20] (G2);
\draw[blue] (I1) to [bend left = 20] (I2);

\fill[blue] (C2) circle [radius=0.2em] node[below] {$\set_\infty$};
\fill[blue] (5.685, -1.2) circle [radius=0.2em] node[right] {$\set_\infty$};
\node[blue] at (1.1,-0.8) {$\lam$};
\end{scope}

\draw[->] (2.5, -4) to [bend right = 30] (3.5,-6);
\node at (1.6,-5.2) {\footnotesize surgery};

\draw[->] (14.2, -4.9) to (15.8,-3.3);
\node[rotate=45] at (14.7,-3.8) {\footnotesize $j\to\infty$};

\begin{scope}[yshift=-170,xshift=160]

\coordinate (A1) at (0,0); 
\coordinate (B1) at (0.2,0.05); 
\coordinate (C1) at (0.4,0.12);
\coordinate (D1) at (0.6,0.21);
\coordinate (E1) at (0.8,0.35);

\coordinate (F1) at (4,1.97);
\coordinate (G1) at (4.1,2.05);
\coordinate (H1) at (4.3,2.17);
\coordinate (I1) at (4.5,2.24);
\coordinate (J1) at (4.7,2.28);
\coordinate (K1) at (4.85,2.27);

\coordinate (F2) at (5,-4.13);
\coordinate (G2) at (5.1,-4.13);
\coordinate (H2) at (5.3,-4.12);
\coordinate (I2) at (5.5,-4.11);
\coordinate (J2) at (5.7,-4.09);
\coordinate (K2) at (5.8,-4.07);

\coordinate (C2) at (0.7,-3);
\coordinate (D2) at (0.9,-3.05);
\coordinate (E2) at (1.1,-3.1);

\coordinate (X) at (-1.5,-0.2);
\coordinate (Y) at (-1.2, -3.2);
\coordinate (U) at (7.5, 1.8);
\coordinate (V) at (7.9, -2.6);

\draw plot [smooth cycle, tension=0.6] coordinates {(A1) (B1) (C1) (D1) (E1)
			      (1.5,0.9) (2,1.2) (2.5,1.3) (3,1.7) (3.7,1.71)
			      (F1) (G1) (H1) (I1) (J1) (K1)
			      (5.5,2.2) (6,1.8)
			      (U) (8.5,-0.5) (V)
			      (7,-2.9) (6.5, -3.3) (6.2, -3.93) 
			      (K2) (J2) (I2) (H2) (G2) (F2)
			      (4.5,-4) (3.5,-3.2) (2.7, -2.5) (2,-2.7) (1.5, -3.2) 
			      (E2) (D2) (C2)
			      (0.3, -3)
			      (Y) (-1.9,-2) (X) (-0.5,-0.05)};
			      
\draw (A1) to [bend left = 15] (0.3,-3);
\draw (B1) to [bend left = 15] (0.5,-2.99);
\draw (C1) to [bend left = 15] (C2);
\draw (D1) to [bend left = 15] (D2);
\draw (E1) to [bend left = 15] (E2);
\draw (F1) to [bend left = 20] (F2);
\draw (G1) to [bend left = 20] (G2);
\draw (H1) to [bend left = 20] (H2);
\draw (I1) to [bend left = 20] (I2);
\draw (J1) to [bend left = 20] (J2);
\draw (K1) to [bend left = 20] (K2);

\pgfmathsetmacro\R{0.65}
\draw[blue] (C2) circle (\R);
\draw[blue] (5.685, -1.2) circle (\R);

\node at (1.7,-0.8) {$\tilde\ms_j$};
\end{scope}
\end{tikzpicture}
%%%%%%%%%%%%%%%%%%%%%%%%%%%%%%%%%%%%%%%%%%%%%%%%%%%%%
%%%%%%%%%%%%%%%%%%%%%%%%%%%%%%%%%%%%%%%%%%%%%%%%%%%%%
\end{figure*}

At this point, due to this precise description of the degeneration, we are able to perform a `simplification surgery' on $\ms_j$.
Namely, fixing $\eps>0$ sufficiently small, we modify the surfaces $\ms_j$ inside $B_\eps(\set_\infty)$ to obtain new surfaces $\tilde\ms_j$ with the following properties:
\begin{itemize}
\item $\tilde\ms_j$ coincides with $\ms_j$ outside $B_\eps(\set_\infty)$ (the surgery is performed only near the `bad points');
\item the surfaces $\tilde\ms_j$ have uniformly bounded curvature;
\item the topology and the area of $\tilde\ms_j$ are comparable to those of $\ms_j$;
\item the surfaces $\tilde\ms_j$ converge locally smoothly to the lamination $\lam$ introduced above.
\end{itemize}

We refer the reader to \cref{cor:ExistenceBlowUpSetAndCurvatureEstimate}, \cref{thm:GlobalDeg} and \cref{cor:Surgery} for precise statements concerning the description of the topological degeneration and the surgery procedure, respectively. We note here that the \emph{refined} Morse-theoretic arguments we present in \crefnoname{sec:MorseTheory} (where we need to separately count the number of curves $\Sigma_j$ traces, locally, along the two parts of the boundary of small geodesic balls near $\partial M$) are essential to make the whole machinery work. 
That being said, by means of this analysis we obtain the following result, which shows that it is possible to remove the assumption on the area bound from \crefnoname{thm:TopFromArea} in the case of ambient manifolds satisfying the mild curvature assumptions mentioned above. 

\begin{theorem} \label{thm:AreaBound}
Let $(\amb^3,g)$ be a compact Riemannian manifold with boundary. 
Moreover assume that
\begin{enumerate}[label={\normalfont(\roman*)}]
\item \emph{either} the scalar curvature of $\amb$ is positive and $\partial \amb$ is mean convex with no minimal components;
\item \emph{or} the scalar curvature of $\amb$ is non-negative and $\partial\amb$ is strictly mean convex.
\end{enumerate}
Given $I\in\N$, there exist constants $\Lambda_0 = \Lambda_0(\amb, g, I)$, $\tau_0=\tau_0(\amb, g, I)$,  $a_0 = a_0(\amb, g, I)$ and $b_0 = b_0(\amb, g, I)$ such that for every compact, connected, embedded, free boundary minimal surface $\ms^2\subset\amb$ with non-empty boundary and with index at most $I$ we have that its area is bounded by $\Lambda_0$, its total curvature is bounded by $\tau_0$, its genus by $a_0$ and the number of its boundary components by $b_0$.
\end{theorem}

\begin{remark}
Note that the requirement that either of the curvature assumptions hold \emph{strictly} is actually necessary, for the manifold $S^1\times S^1\times I$, endowed with a flat metric, contains stable minimal annuli of arbitrarily large area.
\end{remark}

\begin{remark}\label{rmk:FraserLicomment}
An area bound for free boundary minimal surfaces is required also in the proof of \crefnoname{thm:FraLiCpt} (see \cite[Proposition 3.4]{FraLi14}).
However, the assumption that is made there is a bound on the topology and, also, the method employed in that case is essentially analytic, relying on the aforementioned connection between free boundary minimal surfaces and the first Steklov eigenvalue.
\end{remark}

\begin{remark}\label{rmk:Hersch}
So far \cref{thm:AreaBound} was known only for surfaces with index $I=0$ or $I=1$ (see \cite[Appendix A]{AmbBuzCarSha18}). Indeed, for stable free boundary minimal surfaces the area bound follows from the stability inequality and a similar argument can be applied to surfaces with index $1$ based on the well-known \emph{Hersch trick}. 
\end{remark}

In order to prove the previous theorem, it turns out that one needs to gain some control on the size of \emph{stable subdomains} of free boundary minimal surfaces. Prior to this work, this result was known only in the closed case (see \cite{Car15}, based on ideas going back to the work by Schoen-Yau \cite{SchYau83}).

\begin{proposition} \label{prop:StableImpliesCptness}
Let $(\amb^3,g)$ be a three-dimensional Riemannian manifold with boundary. Denote by $\varrho_0\eqdef \inf_M \Scal_g$ the infimum of the scalar curvature of $\amb$ and by $\sigma_0\eqdef \inf_{\partial\amb}H^{\partial\amb}$ the infimum of the mean curvature of $\partial\amb$. Assume that 
\begin{enumerate}[label={\normalfont(\roman*)}]
\item \label{sic:posscal} \emph{either} the scalar curvature of $\amb$ is uniformly positive ($\varrho_0>0$) and $\partial \amb$ is mean convex ($\sigma_0\ge 0$) with no minimal components;
\item \label{sic:strictmc} \emph{or} the scalar curvature of $\amb$ is non-negative ($\varrho_0\ge 0$) and $\partial\amb$ is uniformly strictly mean convex ($\sigma_0 > 0$).
\end{enumerate}
Then every complete, connected, embedded, stable free boundary minimal surface $\ms^2\subset\amb$ that is two-sided and has non-empty boundary is compact, and its intrinsic diameter satisfies the bound
\[
\operatorname{diam}(\ms)\eqdef \sup_{x,y\in\ms} d_\ms(x,y) \le \min\left\{ \frac{2\sqrt 2 \pi}{\sqrt{3\varrho_0}}, \frac{\pi + 8/3}{\sigma_0}\right\}\point
\]
Moreover, one has that
\[
0< \frac{\varrho_0}2 \area(\ms) + \sigma_0 \length(\partial\ms) \le 2\pi\chi(\ms) \semicolon
\]
in particular, $\ms$ is diffeomorphic to a disc.
\end{proposition}

It is interesting to note that the variational argument that allows to prove the corresponding estimate in the closed case is not sufficient in the free boundary context and, indeed, this result turns out to be much more delicate (see \crefnoname{sec:CptnessStable}). Yet, the key idea remains similar: stable (free boundary) minimal hypersurfaces inherit the `positivity' curvature properties of the ambient manifold (cf. \cite{SchYau17} for a striking application of this principle to the proof of the \emph{positive mass theorem} in all dimensions).

We shall now present two significant consequences of \cref{thm:AreaBound}. The first one descends by combining the area bound with the geometric compactness result of \crefnoname{thm:FraLiCpt}.
\begin{corollary} \label{cor:CptBoundIndPosRic}
Let $(\amb^3,g)$ be a compact Riemannian manifold with boundary. Suppose that $\amb$ has non-negative Ricci curvature and that the boundary $\partial\amb$ is strictly convex.
Then any set of compact embedded free boundary minimal surfaces with uniformly bounded index is compact in the $C^k$ topology for every $k\ge 2$.
\end{corollary}

In addition, one can employ the Baire-type result given in \cite[Theorem 9]{AmbCarSha18-Compactness} to obtain the following \emph{generic finiteness} result.

\begin{corollary}\label{cor:GenericFin}
Let $\amb^3$ be a compact manifold with boundary. For a generic choice of $g$ in the class of Riemannian metrics such that $M$ has positive scalar curvature and $\partial M$ has strictly mean convex boundary, the space of compact, embedded, free boundary minimal surfaces with index bounded by $I$ is finite (for any $I\in\mathbb{N}$). Hence, the set of all such surfaces (regardless of their Morse index) is countable. Analogous conclusions hold true for $g$ chosen in a dense subclass of metrics satisfying the curvature conditions \ref{ch:PosScal} or \ref{ch:PosMean} given above.
\end{corollary}

\begin{remark}\label{rmk: GenericYau}
We note that the fact that for a generic choice of the background metric the set of closed minimal surfaces in a compact manifold $(M,g)$ without boundary is \emph{countable} played an essential role in the proof of the generic case of the Yau conjecture by Irie-Marques-Neves \cite{IriMarNev18}. The conclusion of \cref{cor:GenericFin} should be a key building block for the corresponding free boundary result.
\end{remark}

\subsection{Pathological families of free boundary minimal surfaces}

Explicit examples, based on equivariant constructions due to W.-Y. Hsiang \cite{Hsi83} and E. Calabi \cite{Cal67} show that the conclusion of the geometric compactness theorem by Choi-Schoen cannot possibly hold in higher dimension or codimension, respectively, no matter how restrictive curvature assumptions one considers on the ambient manifold. A related question, explicitly posed by B. White in 1984, see \cite{Bro86}, is whether strong compactness still holds under weaker curvature conditions but in ambient dimension three, specifically in the class of compact 3-manifolds of positive scalar curvature. For the closed case, this question was fully answered by T. Colding and C. De Lellis in \cite{ColDeL05} (after earlier, significant contributions by Hass-Norbury-Rubinstein \cite{HasNorRub03}): given any compact 3-manifold $M$ supporting metrics of positive scalar curvature, and a non-negative integer $a$, one can construct a metric of positive scalar curvature $g=g(a)$ so that the ambient manifold $(M,g)$ contains a sequence of pairwise distinct, closed minimal surfaces of genus $a$ which is not compact in the sense above. In fact, one can analyze the limit behavior of such a sequence in detail: one witnesses convergence to a singular minimal lamination (inside the given ambient manifold) with precisely two singular points, located at two antipodes on a stable minimal sphere. Furthermore, the value of the area and of the Morse index of these surfaces diverge. The reader is referred to the monograph \cite{ColMin11} for basic background and as a reference for the terminology we employ. Here we discuss, and solve, the corresponding problem in the setting of \emph{free boundary} minimal surfaces inside smooth compact 3-manifolds with boundary. 

\begin{theorem}\label{thm:Counterexample}
Let $M$ be a compact, orientable, 3-manifold supporting Riemannian metrics of positive scalar curvature and mean convex boundary and let $a\geq 0$ and $b>0$ be integers. Then there exists a Riemannian metric $g=g(a,b)$ of positive scalar curvature and totally geodesic boundary such that $(M,g)$ contains a sequence of connected, embedded, free boundary minimal surfaces of genus $a$ and exactly $b$ boundary components, and whose area and Morse index attain arbitrarily large values. Analogous conclusions hold true requiring the ambient manifold to have positive scalar curvature and strictly mean convex boundary.
\end{theorem}

This shows that such curvature conditions, i.e. $R_g> 0$ in $M$ and $H^{\partial M}\geq 0$ on its boundary $\partial M$ are in general too weak to ensure any form of geometric compactness. In fact, in each of the above examples we witness not only the lack of smooth single-sheeted subsequential convergence, but even of a milder form of subsequential convergence (namely: smooth, possibly with integer multiplicity $m\geq 1$, away from finitely many points) to a smooth free boundary minimal surface. The limit object is, as above, more pathologic than one would hope for. On the other hand, the reader may want to compare these `negative' results with some of the geometric applications in \cite{AmbBuzCarSha18}, based on a bubbling analysis: it is shown that (in 3-manifolds satisfying the aforementioned curvature conditions) sequences of free boundary minimal surfaces of fixed topological type \emph{and correspondingly low index} shall indeed sub-converge in the strongest geometric sense.

\begin{remark}\label{rmk:ConvexBC}
In certain cases, for instance when $M$ is a ball and $b=1$, it is possible to deform the metric $g$ near the boundary $\partial M$ so that $(M,g)$ has positive scalar curvature, \emph{strictly convex} boundary and contains a sequence of free boundary minimal surfaces whose limit behavior is as above.
\end{remark}

\begin{remark}
A simpler variant of the very same construction we shall present in the proof of \cref{thm:Counterexample} allows to prove that, when dropping any curvature requirement, these non-compactness phenomena can be made to occur inside \emph{any} pre-assigned topological 3-manifold.
\end{remark}

\begin{remark}
A full topological characterization of those compact 3-manifolds that support metrics of positive scalar curvature and mean convex boundary has been obtained by the first author and C. Li in \cite[Theorem A]{CarLi19}. Roughly speaking, one can assert that those curvature conditions are only `mildly' restrictive from the topological perspective, as in particular the boundary $\partial M$ can be the disjoint union of closed surfaces whose genera correspond to any pre-assigned string of non-negative integers.  
\end{remark}

The examples we construct partly rely on earlier work by Colding-De Lellis \cite{ColDeL05}. Their construction is, in some sense, \emph{modular}: they build some simple blocks and develop tools to glue such blocks together by means of a suitable \emph{wire matching argument}. The novel ingredients one needs to prove our results is a new family of building blocks: for any $b>0$ we construct a (conveniently simple) Riemannian metric of positive scalar curvature on the 3-ball so that the resulting 3-manifold contains a sequence of free boundary surfaces of genus 0 and exactly $b$ boundary components. The construction we present is rather different depending on whether $b\geq 2$ or instead $b=1$, the latter relying on a smoothing lemma by P. Miao employed to the scope of desingularizing a preliminary edgy model. A detailed proof of \cref{thm:Counterexample} is provided in \crefnoname{sec:Counterexample}.

\vspace{5mm}

To conclude this introduction, we consider three classes of examples that have appeared in recent years in the literature and briefly discuss how they naturally provide counterexamples to the last two implications in \crefnoname{fig:Complexity} we still need to discuss.

\begin{remark}\label{rmk:AreaNoControl}
We remind the reader of the following existence results:
\begin{itemize}
\item{In \cite{FraSch16} A. Fraser and R. Schoen have constructed a sequence $\Sigma^{1}_{k}$ of free boundary minimal surfaces in the unit ball of $\mathbb{R}^3$ having genus 0 and any number $b\geq 2$ of boundary components; it follows from their construction that as one lets $k\to \infty$ the surfaces in question converge (e.g. in the sense of varifolds) with multiplicity two to a flat equatorial disc.} 
\item{In \cite{KapLi17} N. Kapouleas and M. Li have constructed a sequence $\Sigma^{2}_{k}$ of free boundary minimal surfaces in the unit ball of $\mathbb{R}^3$ having genus $a\geq a_0$ (sufficiently large) and $b=3$ boundary components; it follows from their construction that as one lets $k\to \infty$ the surfaces in question converge with multiplicity one to the union of an equatorial disc and a critical catenoid.}
\item{In \cite{KapWiy17} N. Kapouleas and D. Wiygul have constructed a sequence $\Sigma^{3}_{k}$ of free boundary minimal surfaces in the unit ball of $\mathbb{R}^3$ having genus $a\geq a_0$ (sufficiently large) and exactly one boundary component; it follows from their construction that as one lets $k\to \infty$ the surfaces in question converge with multiplicity three to an equatorial disc (more generally, their methods allow a stacking construction as mentioned in Remark 1.2 therein).}
\end{itemize}
Hence, in each of the three cases we have a sequence of minimal surfaces with uniformly bounded area, unbounded topology and (by \crefnoname{thm:TopFromArea}) unbounded index. Note that the second and third family actually arise via a desingularization procedure. The reader might want to compare these existence results with those obtained by D. Ketover \cite{Ket16fb} by means of equivariant min-max techniques \cite{Ket16equiv}.
\end{remark}

\subsection{Structure of the article}

We provide an outline of the contents of this article by means of the following diagram.

\vspace{4ex}
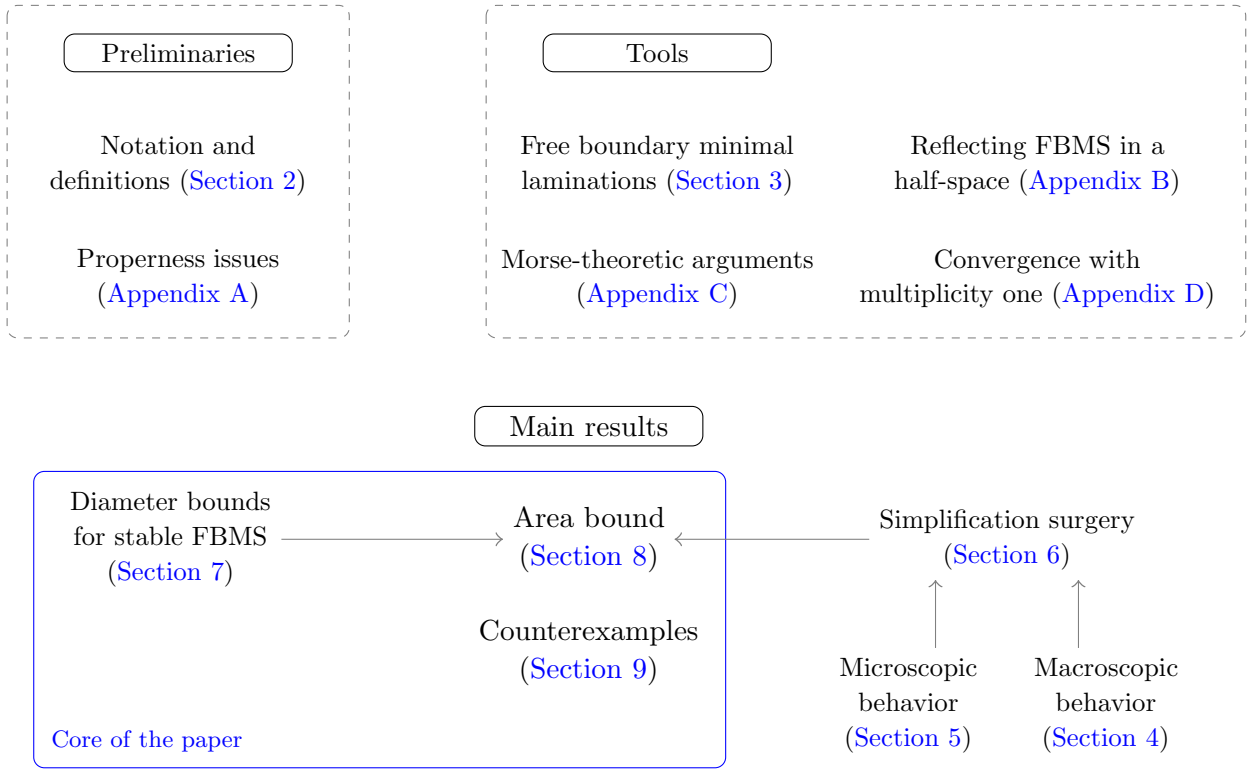
\begin{center}
%%%%%%%%%%%%%%%%%%%%%%%%%%%%%%%%%%%%%%%%%%%%%%%%%%%%%
%%% article_scheme.tikz
%%%%%%%%%%%%%%%%%%%%%%%%%%%%%%%%%%%%%%%%%%%%%%%%%%%%%
\usetikzlibrary{trees}
\tikzstyle{every node}=[anchor=west]
\tikzstyle{tit}=[shape=rectangle, rounded corners,
    draw, minimum width = 3cm]
\tikzstyle{sec}=[shape=rectangle, rounded corners]
\tikzstyle{optional}=[dashed,fill=gray!50]
\begin{tikzpicture}[scale=0.9]

\tikzstyle{mybox} = [draw, rectangle, rounded corners, inner sep=10pt, inner ysep=20pt]
\tikzstyle{fancytitle} =[fill=white, text=black, draw]

\begin{scope}[align=center, font=\small]
 \node [tit] (pre) {Preliminaries};
 \node [sec] [below of = pre, yshift=-0.5cm] (not) {Notation and \\ definitions (\crefnoname{sec:NotDef})};
 \node [sec] [below of = not, yshift=-0.5cm] (pro) {Properness issues\\ (\crefnoname{sec:Properness})};
\end{scope}

\begin{scope}[align=center, xshift=7cm, font=\small]
 \node [tit] (too) {Tools};
 \node [sec] [below of = too, yshift=-0.5cm] (fbml) {Free boundary minimal\\ laminations (\crefnoname{sec:FBMLam})};
 \node [sec] [right of = fbml, xshift=4cm] (ref) {Reflecting FBMS in a \\half-space (\crefnoname{sec:ReflPrinc})};
 \node [sec] [below of = ref, yshift=-0.5cm] (mul) {Convergence with \\ multiplicity one (\crefnoname{sec:Mult1Conv})};
 \node [sec] [below of = fbml, yshift=-0.5cm] (mor) {Morse-theoretic arguments\\ (\crefnoname{sec:MorseTheory})};
\end{scope}

\draw[gray,dashed,rounded corners] ($(pre)+(-2.5,0.7)$)  rectangle ($(pre)+(2.5,-4.2)$);
\draw[gray,dashed, rounded corners] ($(too)+(-2.5,0.7)$)  rectangle ($(too)+(8.6,-4.2)$);

\begin{scope}[align=center, xshift=6cm,yshift=-5.5cm]
\node [tit] (res) {Main results};
\node [sec] [below of = res, yshift=-0.5cm] (are) {Area bound\\ (\crefnoname{sec:AreaBound})};
\node [sec] [below of = are, yshift=-0.5cm] (cou) {Counterexamples\\ (\crefnoname{sec:Counterexample})};
\begin{scope}[font=\small]
\node [sec] [right of = are, xshift=4.5cm] (sur) {Simplification surgery\\ (\crefnoname{sec:Surgery})};
\node [sec] [left of = are, xshift=-4.5cm] (cpt) {Diameter bounds \\ for stable FBMS \\ (\crefnoname{sec:CptnessStable})};
\node [sec] [below of = sur, xshift=-1.3cm, yshift=-1.2cm] (mic) {Microscopic\\ behavior \\ (\crefnoname{sec:LocalDegeneration})} ;
\node [sec] [below of = sur, xshift=1.3cm, yshift=-1.2cm] (mac) {Macroscopic\\ behavior \\ (\crefnoname{sec:MacroDescr})} ;

\draw[blue,rounded corners] ($(cou)+(2,-1.7)$)  rectangle ($(cpt)+(-2,1)$);
\node[blue] at ($(cou)+(-8,-1.3)$) {\footnotesize Core of the paper};

\end{scope}

\tikzset{bigarr/.style={
decoration={markings,mark=at position 1 with {\arrow[scale=1.5]{>}}},
postaction={decorate}}}
 
\begin{scope}[gray]
\draw[bigarr] ([xshift=0.4cm]mic.north) to ([xshift=0.4cm,yshift=1cm]mic.north);
\draw[bigarr] ([xshift=-0.4cm]mac.north) to ([xshift=-0.4cm,yshift=1cm]mac.north);
\draw[bigarr] (sur) to (are);
\draw[bigarr] (cpt) to (are);
\end{scope}
\end{scope}
\end{tikzpicture}
%%%%%%%%%%%%%%%%%%%%%%%%%%%%%%%%%%%%%%%%%%%%%%%%%%%%%
%%%%%%%%%%%%%%%%%%%%%%%%%%%%%%%%%%%%%%%%%%%%%%%%%%%%%
\end{center}

\noindent \textit{Acknowledgements:} The first author would like to thank Lucas Ambrozio and Pengzi Miao for helpful discussions.
A. C. is partly supported by the National Science Foundation (through grant DMS 1638352) and by the Giorgio and Elena Petronio Fellowship Fund; G. F. is partly supported by a fellowship of the Zurich Graduate School in Mathematics. This project was partly developed at the Institute for Advanced Study during the special year \emph{Variational Methods in Geometry}.

%%%%%%%%%%%%%%%%%%%%%%%%%%%%%%%%%%%%%%%%%%%%%%%%%%%%%%%%%%%%%%%%
%%% Notations and definitions 
%%%%%%%%%%%%%%%%%%%%%%%%%%%%%%%%%%%%%%%%%%%%%%%%%%%%%%%%%%%%%%%%

\section{Notation and definitions} \label{sec:NotDef}

\subsection{Basic notation}

Let $(\amb^3,g)$ be a Riemannian manifold with boundary and let $\ms^2\subset\amb$ be an embedded surface. Then we denote by:
\begin{itemize}
\item $\Diff$ the connection on $\amb$ and $\cov$ the induced connection on $\ms$.
\item $\nu$ a choice of a global unit normal vector field on $\ms$ when $\ms$ is two-sided.
\item $\eta$ the outward unit co-normal vector field to $\partial \ms$.
\item $\hat\eta$ the outward unit co-normal vector field to $\partial\amb$ (which coincides with $\eta$ along $\partial\ms$ when $\ms$ satisfies the free boundary property). 
\item $\A(X,Y) = g(D_X Y,\nu)$ the second fundamental form of $\ms\subset\amb$ and $\II^{\partial\amb}(X,Y) = g(D_X Y,\hat{\eta})$ the second fundamental form of $\partial\amb\subset\amb$. Observe that $\II^{\partial\amb} < 0$ if for example $\amb$ is the unit ball in $\R^3$ (and thus $\partial\amb$ is the unit sphere).
\item $H^{\partial\amb}$ the mean curvature of $\partial\amb$, that is $H^{\partial\amb} = - \II^{\partial\amb}(E_1,E_1) - \II^{\partial\amb}(E_2,E_2)$ for every choice of a local orthonormal frame $\{E_1,E_2\}$ of $\partial\amb$. 
We say that $\partial\amb$ is mean convex if $H^{\partial\amb} \ge 0$, and strictly mean convex if $H^{\partial\amb} > 0$ (that is, for example, the case of the unit ball in $\R^3$).
\item $\chi(\ms)$, $\genus(\ms)$, $\bdry(\ms)$ and $\area(\ms)$ respectively the Euler characteristic, the genus, the number of boundary components and the area of $\ms$.
\end{itemize}
All these notions are easily extended to the case when the surface in question is only immersed, rather than embedded.

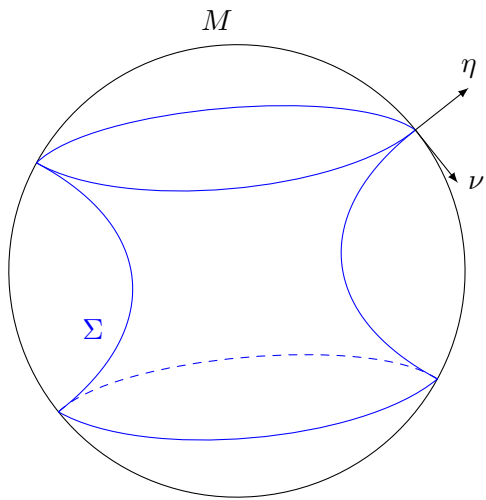
\begin{figure}[htpb]
\centering
%%%%%%%%%%%%%%%%%%%%%%%%%%%%%%%%%%%%%%%%%%%%%%%%%%%%%
%%% crit_cat.tikz
%%%%%%%%%%%%%%%%%%%%%%%%%%%%%%%%%%%%%%%%%%%%%%%%%%%%%
\begin{tikzpicture}[line cap=round,line join=round,rotate=5]

\pgfmathsetmacro{\R}{3}
\pgfmathsetmacro{\T}{1.19968}
\pgfmathsetmacro{\a}{1/(\T*cosh(\T))}
\pgfmathsetmacro{\e}{0.403*\R}
\coordinate(Pa)at({\R*\a*cosh(\T)},{\R*\a*\T});
\coordinate(Pav)at({1.3*\R*\a*cosh(\T)},{1.3*\R*\a*\T});
\coordinate(Pao)at({\R*\a*cosh(\T)+0.3*\R*\a*\T},{\R*\a*\T-0.3*\R*\a*cosh(\T)});
\coordinate(Pb)at({\R*\a*cosh(-\T)},{-\R*\a*\T});
\coordinate(Qa)at({-\R*\a*cosh(\T)},{\R*\a*\T});
\coordinate(Qb)at({-\R*\a*cosh(-\T)},{-\R*\a*\T});

\pgfmathsetmacro{\wa}{215}
\pgfmathsetmacro{\wb}{140}
\pgfmathsetmacro{\ua}{0.4375*\R}
\pgfmathsetmacro{\ub}{0.35*\R}

\draw[myBlack] (0,0) circle (\R);
\node[above=2pt] at(0,\R) {$\amb$};

\begin{scope}[myBlue]
\node at(-0.65*\R,-0.2*\R) {$\ms$};
\draw[scale=1,domain=-\T:\T,smooth,variable=\s]  plot ({\R*\a*cosh(\s)},{\R*\a*\s});
\draw[scale=1,domain=-\T:\T,smooth,variable=\s]  plot ({-\R*\a*cosh(\s)},{\R*\a*\s});

\draw(Pa)..controls+(\wa:\ua)and +(180-\wa:\ua)..(Qa);
\draw(Pb)..controls+(\wa:\ua)and +(180-\wa:\ua)..(Qb);
\draw(Pa)..controls+(\wb:\ub)and +(180-\wb:\ub)..(Qa);
\draw[dashed](Pb)..controls+(\wb:\ub)and +(180-\wb:\ub)..(Qb);
\end{scope}

\draw[-latex] (Pa) -- (Pav) node[above]{$\eta$};
\draw[-latex] (Pa) -- (Pao) node[right]{$\nu$};
\end{tikzpicture}
%%%%%%%%%%%%%%%%%%%%%%%%%%%%%%%%%%%%%%%%%%%%%%%%%%%%%
%%%%%%%%%%%%%%%%%%%%%%%%%%%%%%%%%%%%%%%%%%%%%%%%%%%%%
\caption{An example of a free boundary minimal surface in the unit ball $B^3\subset \R^3$ with some notation included.} \label{fig:NotEx}
\end{figure}

In certain circumstances, for instance as a result of a blow-up procedure, we will have to work in a half-space of $\R^3$.
In that respect, for $0 \ge a \ge -\infty$ we shall set $\Xi(a) \eqdef\{ x^1\ge a \}$ and $\Pi(a)\eqdef\{ x^1= a \}$. When $a=-\infty$, we agree that $\Xi(a)$ coincides with $\R^3$ and $\Pi(a)$ is empty.
We omit the dependence on $a$ when we are interested in a generic half-space, without caring about the distance of the origin from the boundary.

\subsection{Free boundary minimal surfaces} \label{subsec:FBMS}

Given a Riemannian manifold $(\amb^3,g)$, we denote by $\vf_\partial(\amb)$ the linear space of smooth ambient vector fields $X$ such that 
\begin{enumerate}[label={\normalfont(\roman*)}]
\item $X(x) \in T_x\amb$ for all $x\in\amb$,
\item $X(x)\in T_x\partial\amb$ for all $x\in\partial\amb$.
\end{enumerate}
The first variation of the area functional at a smooth, embedded surface $\ms^2\subset\amb$ with respect to the flow $\psi_t$ generated by a compactly supported vector field $X\in \vf_\partial(\amb)$ is given by
\begin{equation} \label{eq:FirstVar}
\frac{\d}{\d t}\Big|_{t=0} \Haus^n(\psi_t(\ms))  = \int_\ms \div_\ms (X) \de \Haus^2 = -\int_\ms \scal HX \de \Haus^2 + \int_{\partial\ms} \scal{X}\eta \de \Haus^1 \point
\end{equation}
Once again, the same formula holds true, modulo straightforward notational changes, in the case of immersed surfaces.

When $M$ and $\ms$ are compact, it is customary to say that a surface $\ms^2\subset \amb$ is \emph{properly embedded} if $\partial\ms = \ms\cap \partial\amb$. If that is the case, the equation above implies that $\ms$ is a stationary point of the area functional if and only if it has zero mean curvature and meets the ambient boundary orthogonally. In this case $\ms$ is called \emph{free boundary minimal surface}.

\begin{remark}  \label{subsec:ImpRem}
It is possible to extend this definition also to non-properly embedded surfaces, convening that $\ms^2\subset\amb$ is a free boundary minimal surface if it has zero mean curvature and meets the boundary of the ambient manifold orthogonally along its own boundary. 
\end{remark}

In this article we focus our attention on compact Riemannian manifolds $(\amb^3,g)$ that satisfy property \hypP{}, that has been stated in the introduction. However, note that local limits of properly embedded free boundary minimal surfaces can be non-properly embedded even assuming property \hypP{}.
In particular this is the case of \crefnoname{thm:LamCptness}.

\begin{remark} \label{rem:ConfusionProper}
We stress that, when $\ms$ or $\amb$ are non-compact, there is a possible confusion between \emph{proper} in the sense stated above, namely $\partial \ms = \ms\cap \partial\amb$, and \emph{proper} in the sense that the inclusion map $\ms\hookrightarrow\amb$ is proper as in General Topology.
Unfortunately during the article we have to use both meanings, since we need to consider non-compact free boundary minimal surfaces in our ambient manifold and even free boundary minimal surfaces in non-compact ambient spaces.
Therefore, when we talk about \emph{properly embedded} surfaces with boundary we mean that \emph{both} properties are satisfied.
In all cases of possible ambiguity, we try to make the interpretation as explicit and clear as possible.
\end{remark}

Thanks to the maximum principle, property \hypP{} is implied by the following geometric condition on the boundary.
\begin{description}
\item [$(\mathfrak{C})$]\ \ \  The boundary $\partial\amb$ is mean convex and has no minimal component. \label{HypC}
\end{description}
\begin{remark}
Note that the assumptions of \crefnoname{thm:FraLiCpt}, \cref{thm:AreaBound}, \cref{prop:StableImpliesCptness} and \cref{cor:CptBoundIndPosRic} all imply \hypC{} and thus \hypP{}.
\end{remark}

\begin{remark}\label{rem:edged}
In certain circumstances, we shall deal with free boundary minimal surfaces $\ms$ whose boundary is not entirely contained in $\partial\amb$, that is when $\partial\ms\setminus\partial\amb \not=\emptyset$ (cf. also \cite[Remark 13]{AmbCarSha18-Compactness}).
For example, we look at the intersection of free boundary minimal surfaces with geodesic balls in the ambient manifold centered at boundary points: in that case we talk about \emph{edged} free boundary minimal surface and we mean that they satisfy the free boundary condition only with respect to the boundary of the ambient manifold that is contained in the ball, and not necessarily with respect to the relative boundary of the ball in question.
\end{remark}

Then, given a properly embedded free boundary minimal surface $\ms$, one can consider the second variation of the area functional, which can be written as
\begin{equation} \label{eq:SecondVar}
\begin{split}
\frac{\d^2}{\d t^2}\Big|_{t=0} \Haus^n(\psi_t(\ms)) &= \int_\ms (\abs{\cov^{\perp}X^{\perp}}^2 - (\Ric_{\amb}(X^\perp,X^\perp) + \abs A^2 \abs{X^\perp}^2)) \de \Haus^2 +{}\\
&\pheq + \int_{\partial\ms} \II^{\partial\amb}(X^\perp, X^\perp) \de \Haus^1\comma
\end{split}
\end{equation}
where $X^\perp$ is the normal component of $X$ and $\cov^\perp$ is the induced connection on the normal bundle $N\ms$ of $\ms\subset\amb$.
Thus, given a section $V\in\Gamma(N\ms)$ of the normal bundle, the second variation along the flow generated by $V$ is equal to the quadratic form
\[
 Q^\ms(V,V) \eqdef  \int_\ms (\abs{\cov^{\perp}V}^2 - (\Ric_{\amb}(V,V) + \abs A^2 \abs{V}^2)) \de \Haus^2 + \int_{\partial\ms} \II^{\partial\amb}(V, V) \de \Haus^1 \point
\]
The \emph{(Morse) index} of $\ms$ is defined as the maximal dimension of a linear subspace of $\Gamma(N\ms)$ where $Q^\ms$ is negative definite.
Under the above assumptions, this number equals the number of negative eigenvalues of the elliptic problem
\[
\begin{cases}
\lapl_\ms^\perp V + \Ric^\perp_\amb(V,\cdot) +\abs A^2 V + \lambda V = 0 & \text{on $\ms$}\comma\\
\cov^\perp_\eta V = - (\II^{\partial\amb}(V,\cdot))^\sharp & \text{on $\partial\ms$}\point
\end{cases}
\]

Observe that, if $\ms^2\subset\amb^3$ is two-sided, then every vector field $V\in\Gamma(N\ms)$ can be written as $V = f\nu$ and $Q^\ms(V,V)$ coincides with the quadratic form $Q_\ms(f,f)$, defined as
\[
\begin{split}
Q_\ms(f,f) &\eqdef \int_\ms (\abs{\cov f}^2 - (\Ric_{\amb}(\nu,\nu) + \abs A^2 )f^2) \de \Haus^2 + \int_{\partial\ms} \II^{\partial\amb}(\nu, \nu)f^2 \de \Haus^1\\
&= -\int_\ms f\jac_\ms(f) \de\Haus^2 + \int_{\partial\ms} \left( f\DerParz f \eta + \II^{\partial\amb}(\nu,\nu) f^2 \right) \de\Haus^1 \comma
\end{split}
\]
where $\jac_\ms \eqdef \lapl_\ms + (\Ric_\amb(\nu,\nu) + \abs A^2)$ is the (scalar) Jacobi operator of $\ms$.

\subsection{Euler characteristic} \label{sec:EulChar}
Recall that the Euler characteristic of a compact surface $\ms$ is equal to
\[
\chi(\ms) = \begin{cases}
            2 - 2\genus(\ms) - \bdry(\ms) & \text{if $\ms$ is orientable} \comma\\
            1- \genus(\ms) - \bdry(\ms) & \text{if $\ms$ is not orientable} \point
            \end{cases}
\]
Now let $\ms_1,\ms_2$ be two compact oriented surfaces with boundary and consider $c_1$, $c_2$ be two boundary components of $\ms_1$ and $\ms_2$ respectively. Note that $c_1$ and $c_2$ are both homeomorphic to $S^1$.
As shown in \crefnoname{fig:GlSurf}, we can glue $\ms_1$ and $\ms_2$ along $c_1$ and $c_2$ in two ways:
\begin{enumerate} [label={\normalfont(\roman*)}]
\item We can attach all $c_1$ to all $c_2$. \label{at:circle}
\item We can attach an arc of $c_1$ to an arc of $c_2$. \label{at:segment}
\end{enumerate}

\begin{figure}[htpb]
\centering
%%%%%%%%%%%%%%%%%%%%%%%%%%%%%%%%%%%%%%%%%%%%%%%%%%%%%
%%% gluing_surf.tikz
%%%%%%%%%%%%%%%%%%%%%%%%%%%%%%%%%%%%%%%%%%%%%%%%%%%%%
\begin{tikzpicture}[scale=0.8]
\pgfmathsetmacro{\radius}{4}
\pgfmathsetmacro{\a}{4/sqrt(17)}
\draw (-\radius,0) arc (180:360:{\radius} and {\radius/4});
\draw[dashed] (\radius,0) arc (0:180:{\radius} and {\radius/4});
\draw (\radius,0) arc (0:180:{\radius});
\draw[dashed,blue] (\a,\a) arc ({atan(1/4)}:90:{\radius/4} and {\radius});
\draw[blue] (0,\radius) arc (90:{180+atan(1/4)}:{\radius/4} and {\radius});

\node at ({-\radius+0.9}, {\radius-0.9}) {$\ms_1$};
\node at ({\radius-0.9}, {\radius-0.9}) {$\ms_2$};

\pgfmathsetmacro{\xs}{-5/2*\radius}
\pgfmathsetmacro{\ys}{\radius/2-0.5}
\pgfmathsetmacro{\radiusa}{(\radius/2)+0.5}
\pgfmathsetmacro{\radiusb}{\radius}

\draw (\xs,\ys) ellipse ({\radiusb} and {\radiusa});
\draw[blue] (\xs, {\ys+\radiusa}) arc (90:270:{\radius/4} and {\radiusa});
\draw[dashed,blue] (\xs, {\ys+\radiusa}) arc (90:-90:{\radius/4} and {\radiusa});
\node at ({-\radius+0.3+\xs}, {\radius-0.9}) {$\ms_1$};
\node at ({\radius-0.3+\xs}, {\radius-0.9}) {$\ms_2$};
\end{tikzpicture}
%%%%%%%%%%%%%%%%%%%%%%%%%%%%%%%%%%%%%%%%%%%%%%%%%%%%%
%%%%%%%%%%%%%%%%%%%%%%%%%%%%%%%%%%%%%%%%%%%%%%%%%%%%%
\caption{Gluing of two discs via \ref{at:circle} (on the left) and \ref{at:segment} (on the right).} \label{fig:GlSurf}
\end{figure}
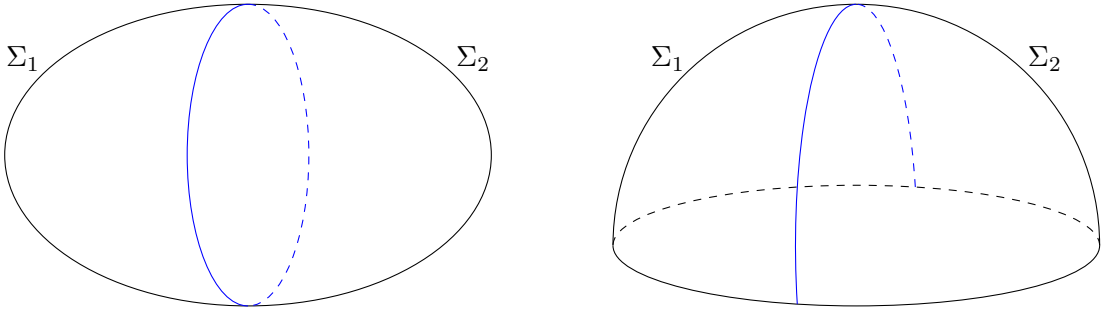

In general, we can construct an oriented surface $\ms$ by gluing $b$ boundary components of $\ms_1,\ms_2$ as in \ref{at:circle} and $b'$ boundary components as in \ref{at:segment}. Then the Euler characteristic of $\ms$ is given by
\[
\chi(\ms) = \chi(\ms_1)+\chi(\ms_2) - b' \point
\]
Therefore $\ms$ has genus equal to $\genus(\ms_1)+\genus(\ms_2) + b + b' - 1$ and number of boundary components equal to the sum of the boundary components of $\ms_1$ and $\ms_2$ minus $2b+b'$.

%%%%%%%%%%%%%%%%%%%%%%%%%%%%%%%%%%%%%%%%%%%%%%%%%%%%%%%%%%%%%%%%
%%% Free boundary minimal laminations
%%%%%%%%%%%%%%%%%%%%%%%%%%%%%%%%%%%%%%%%%%%%%%%%%%%%%%%%%%%%%%%%

\section{Free boundary minimal laminations}\label{sec:FBMLam}

In this section we want to generalize the definitions and the results about minimal laminations (see for example \cite[Definition 2.1]{Car15} or \cite[Definition 2.2]{MeePerRos16}) to the case of manifolds with boundary (see also \cite[Section 5]{GuaLiZho17}). 
Indeed, laminations naturally arise as limits of (free boundary) minimal surfaces that are assumed to have uniformly bounded index, but not necessarily uniformly bounded area.

\subsection{Definitions and compactness}

\begin{definition} \label{def:fbmlam}
A \emph{free boundary minimal lamination} $\lam$ in a three-dimensional Riemannian manifold $(\amb^3,g)$ with boundary $\partial\amb$ is the union of a collection of pairwise disjoint, connected, embedded free boundary minimal surfaces of $\amb$. Moreover we require that $\cup_{L\in \lam} L$ is a closed subset of $\amb$ and that, for each $x\in\amb$ one of the following assertions holds:
\begin{enumerate}  [label={\normalfont(\roman*)}]
\item 
\label{fbml:interior}
$x\in \amb\setminus \partial\amb$ and there exists an open neighborhood $U$ of $x$ and a local coordinate chart $\varphi: B_1^2(0)\times \oo01\subset \R^3 \to U$ such that $\varphi^{-1}((\cup_{L\in \lam} L) \cap U) = B_1^2(0) \times C$ for a closed subset $C\subset \oo 01$; 
\item \label{fbml:goodbdry}
$x\in \partial\amb$ and there exists a relatively open neighborhood $U$ of $x$ and a local coordinate chart $\varphi:(B_1^2(0) \cap \{x^1\ge 0\})\times \oo01 \subset \Xi\to U$ such that $\varphi^{-1}((\cup_{L\in \lam} L) \cap U) = (B_1^2(0)\cap\{x^1\ge 0\})\times C$ for a closed subset $C\subset \oo 01$; 
\item \label{fbml:badbdry}
$x\in\partial\amb$ and there exists an extension $\check\amb$ without boundary and an open neighborhood $U$ of $x$ in $\check\amb$ such that property \ref{fbml:interior} is satisfied for the neighborhood $U\ni x$.
\end{enumerate}
\end{definition}
A schematic representation of the three different situations in \cref{def:fbmlam} can be found in \cref{fig:DefLam}.

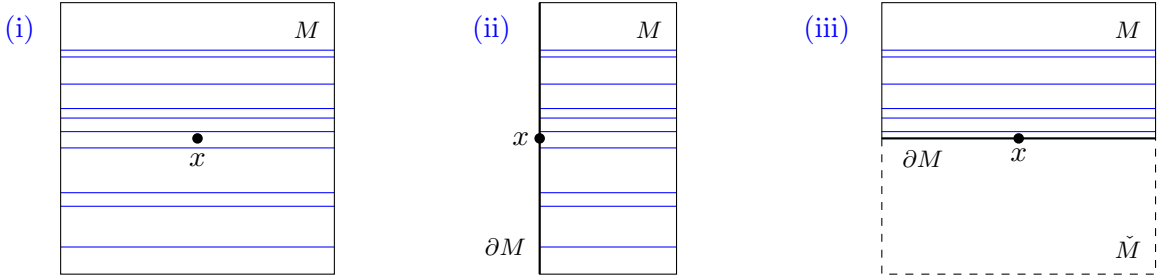
\begin{figure}[htpb]
\centering
%%%%%%%%%%%%%%%%%%%%%%%%%%%%%%%%%%%%%%%%%%%%%%%%%%%%%
%%% def_lam.tikz
%%%%%%%%%%%%%%%%%%%%%%%%%%%%%%%%%%%%%%%%%%%%%%%%%%%%%
\begin{tikzpicture}[scale=1.8]
\pgfmathsetmacro{\xa}{-3}
\pgfmathsetmacro{\xb}{-0.5}
\pgfmathsetmacro{\xc}{3}
\pgfmathsetmacro{\r}{1}
\pgfmathsetmacro{\s}{0.8}
\pgfmathsetmacro{\t}{0.3}

\pgfmathsetmacro{\hp}{5}
\def\heightspos{{0.05, 0.15, 0.22, 0.4, 0.6, 0.65}}
\pgfmathsetmacro{\hn}{3}
\def\heightsneg{{-0.07, -0.4, -0.5, -0.8}}

% Lamination i
\foreach \i in {0,...,\hp}
  \draw[myBlue] ({\xa-\r},\heightspos[\i]) -- ({\xa+\r},\heightspos[\i]);
\foreach \i in {0,...,\hn}
  \draw[myBlue] ({\xa-\r},\heightsneg[\i]) -- ({\xa+\r},\heightsneg[\i]);
  
\fill (\xa, 0) circle [radius=0.1em] node[below=2pt] {$x$};
\draw  ({\xa+\r},\r) -- ({\xa-\r},\r) -- ({\xa-\r},-\r) -- ({\xa+\r},-\r) -- cycle;
\node at ({\xa+\s}, \s) {\footnotesize $\amb$};
\node at ({\xa-\r-\t}, \s) {\ref{fbml:interior}};

% Lamination ii
\foreach \i in {0,...,\hp}
  \draw[myBlue] (\xb,\heightspos[\i]) -- ({\xb+\r},\heightspos[\i]);
\foreach \i in {0,...,\hn}
  \draw[myBlue] (\xb,\heightsneg[\i]) -- ({\xb+\r},\heightsneg[\i]);
  
\fill (\xb, 0) circle [radius=0.1em] node[left] {$x$};
\draw  ({\xb+\r},\r) -- ({\xb},\r) -- ({\xb},-\r) -- ({\xb+\r},-\r) -- cycle;
\draw[thick] ({\xb},\r) -- ({\xb},-\r);
\node at ({\xb+\s}, \s) {\footnotesize $\amb$};
\node at ({\xb-0.25}, -\s) {\footnotesize $\partial \amb$};
\node at ({\xb-\t-0.05}, \s) {\ref{fbml:goodbdry}};

% Lamination iii
\foreach \i in {0,...,\hp}
  \draw[myBlue] ({\xc-\r},\heightspos[\i]) -- ({\xc+\r},\heightspos[\i]);

\fill (\xc, 0) circle [radius=0.1em] node[below] {$x$};
\draw ({\xc+\r},\r) -- ({\xc-\r},\r) -- ({\xc-\r},0) -- ({\xc+\r},0) -- cycle;
\draw[thick] ({\xc-\r},0) -- ({\xc+\r},0);
\node at ({\xc+\s}, \s) {\footnotesize $\amb$};
\node at ({\xc-0.7}, -0.15) {\footnotesize $\partial\amb$};
\draw[dashed]  ({\xc-\r},0) -- ({\xc-\r},-\r) -- ({\xc+\r},-\r) -- ({\xc+\r},0);
\node at ({\xc+\s}, -\s) {\footnotesize $\check{\amb}$};
\node at ({\xc-\r-\t-0.1}, \s) {\ref{fbml:badbdry}};
\end{tikzpicture}
%%%%%%%%%%%%%%%%%%%%%%%%%%%%%%%%%%%%%%%%%%%%%%%%%%%%%
%%%%%%%%%%%%%%%%%%%%%%%%%%%%%%%%%%%%%%%%%%%%%%%%%%%%%
\caption{Definition of lamination in chart.} \label{fig:DefLam}
\end{figure}
\begin{remark}
Note that, if we require property \hypP{} on $\amb$, then case \ref{fbml:badbdry} cannot occur. We have included it in the definition because we need to consider free boundary minimal laminations in half-spaces of $\R^3$ (which do not fulfill \hypP{}) and local limits of free boundary minimal surfaces (or laminations) which a priori can be non-properly embedded (see \crefnoname{subsec:ImpRem} and \crefnoname{thm:LamCptness} below).
\end{remark}

\begin{definition}
We say that a point $p$ of a minimal lamination $\lam$ is a \emph{limit point} if there exists a coordinate chart $(U,\varphi)$ with $p\in U$ as in the previous definition such that $\varphi^{-1}(p) = (t,x)$ and $t$ is an accumulation point for $C$.
\end{definition}

\begin{remark} \label{rem:LimLeaf}
Thanks to the Harnack inequality, if $p$ is a limit point of a lamination $\lam$, then the entire leaf through $p$ consists of limit points of $\amb$. In this case, we shall call it a \emph{limit leaf}.
\end{remark}

As anticipated, we introduce the concept of lamination to gain compactness. Indeed the following theorem holds.
\begin{theorem}[{\cite[Proposition B.1]{ColMin04} and \cite[Theorem 5.5]{GuaLiZho17}}] \label{thm:LamCptness}
Let $(\amb^3, g)$ be a complete three-dimensional Riemannian manifold with boundary. Given $x\in\amb$ and $r>0$, let $\lam_j$ be a sequence of free boundary minimal laminations with uniformly bounded curvatures in $\subset B_{2r}(x) \subset \amb$. Then there exists a subsequence which converges in $B_r(x)$ in the $C^{0,\alpha}$ topology for any $\alpha<1$ to a Lipschitz minimal lamination $\lam$ whose (possibly non-properly embedded) leaves have free boundary with respect to $\partial\amb$. Moreover the leaves of $\lam$ are smooth free boundary minimal surfaces and the leafwise convergence is $C^\infty$. 
\end{theorem}

\begin{remark} \label{rem:ConvOfLam} In the scenario described in the previous statement, we will say that $\lam_j$ locally converges to $\lam$ \emph{in the sense of laminations}.
In that respect, we recall that the convergence of a sequence of laminations $\lam_j$ to a lamination $\lam$ is leafwise $C^\infty$ if, for every sequence of points $x_j\in \lam_j$ that converges to a point $x\in \amb$, there exists a leaf $L\subset\lam$ such that $x\in L$ and a neighborhood $U$ of $x$ such that the connected component $L_j'$ of $L_j\cap U$ that contains $x_j$ converges to the connected component of $L\cap U$ that contains $x$ smoothly with multiplicity one (in the sense of graphs).
\end{remark}

It is well-known that a two-sided minimal surface having a positive Jacobi field is stable (cf. e.g. \cite[Lemma 1.36]{ColMin11}). In turn, the existence of a positive Jacobi field can be deduced whenever multi-sheeted convergence occurs. We shall state here a helpful variation on this theme, whose proof is, by now, rather standard (cf. e.g. Theorem 22 in \cite{Whi16}, and Proposition 2.1 in \cite{CarChoEic16}).
\begin{lemma} \label{lem:StableUnivOrMultOne}
Let $(\amb^3,g)$ be a complete three-dimensional Riemannian manifold with boundary and let $\ms_j^2\subset\amb$ be a sequence of (connected) properly embedded free boundary minimal surfaces that locally converge to a free boundary minimal lamination $\lam\subset \amb\setminus\set_\infty$ away from a finite set of points $\set_\infty$. Consider a leaf $L\in\lam$, then one of the following assertions holds:
\begin{enumerate} [label={\normalfont(\arabic*)}]
\item $L$ has stable universal cover; \label{suomo:stable}
\item the convergence of $\ms_j$ to $L$ is locally smooth with multiplicity one (in the sense of graphs); namely, for every $x\in L$ there exists a neighborhood $U$ of $x$ in $\amb$ such that $U\cap (\bigcup_{L'\in\lam} L') = U\cap L$ and $U\cap \ms_j$ converges smoothly to $U\cap L$ with multiplicity one (as graphs). \label{suomo:multone}
\end{enumerate}
\end{lemma}

\subsection{Removable singularities}

Thanks to \crefnoname{thm:LamCptness}, in \crefnoname{sec:MacroDescr} we prove that free boundary minimal surfaces with uniformly bounded index converge, possibly after extracting a subsequence, to a free boundary minimal lamination which is smooth away from a finite number of points (so that, in particular, we are in the scenario described in the statement of \cref{lem:StableUnivOrMultOne}). The aim of the following propositions is to provide tools to show that these singularities are actually removable.

\begin{corollary} \label{cor:StableLam}
Let $\varphi:\ms\to \Xi(0)\setminus\{0\}\subset\R^3$ be a stable, two-sided minimal immersion that has free boundary with respect to $\Pi(0)$ and is complete away from $\{0\}$. Then the closure of $\varphi(\ms)$ is a plane or a half-plane.
\end{corollary}
\begin{proof} If $\partial\ms=\emptyset$ then the result is a direct consequence of the Bernstein-type theorem by Gulliver-Lawson \cite{GulLaw86}. Otherwise,
consider the double $\check{\ms}$ of $\ms$ (that is to say: the boundaryless surface that is obtained by reflecting $\ms$ across $\partial\ms$) and let $\tau:\check{\ms}\to \check{\ms}$ be the associated involution. 
Denoting with $\varrho:\R^3\to\R^3$ the reflection with respect to $\Pi(0)$, we define the map $\check\varphi:\check{\ms} \to \R^3\setminus\{0\}$ as follows:
\[
\check\varphi(x) \eqdef \begin{cases}
                   \varphi(x) & \text{if $x\in \ms\subset \check{\ms}$}\comma\\
                   \varrho(\varphi(\tau(x))) & \text{if $x\in \check{\ms}\setminus\ms$} \point
                  \end{cases}
\]
Observe that $\check\varphi$ is a two-sided minimal immersion that is complete away from $\{0\}$. Moreover, it follows from the discussion presented in \cite[Section 2]{AmbBuzCarSha18} that $\check\varphi$ is stable, so we can conclude as above.
\end{proof}

The following result mirrors the one obtained for interior points in Proposition D.3 of \cite{ChoKetMax17}.

\begin{proposition} \label{prop:RemovableSingStab}
Let $(\amb^3,g)$ be a complete Riemannian manifold with boundary. Fix $p\in\partial\amb$ and $\eps_0>0$ and consider an embedded minimal surface $\hat\ms\subset B_{\eps_0}(p)\setminus\{p\}$ having free boundary with respect to $\partial\amb$, and stable universal cover.
Then $\hat\ms$ smoothly extends across $p$, i.e. there exists a free boundary minimal surface $\ms\subset B_{\eps_0}(p)$ such that $\hat\ms = \ms\setminus\{p\}$.
\end{proposition}
\begin{remark}
Note that we are not requiring that $\hat\ms$ is properly embedded in $B_{\eps_0}(p)\setminus \{p\}$ (in particular it could be non-properly embedded in the sense of maps).
\end{remark}

\begin{proof}
First observe that we can assume that $p$ belongs to the topological closure of $\hat \ms$, otherwise the result would be obvious.
Taking $\eps_0>0$ possibly smaller and using \cite[Theorem 1.2]{GuaLiZho17}, we can assume that in $B_{\eps_0}(p)$ it holds $\abs{\A_{\hat\ms}}(x)d_g(p,x) \le C$ for all $x\in{\hat\ms}$, for some $C>0$. Therefore, for any $r_j\to 0$, the surfaces $r_j^{-1}({\hat\ms} - p)$ locally converge (in the sense of laminations), up to subsequence, to a free boundary minimal lamination ${\lam}_\infty$ of $\Xi(0)\setminus\left\{0\right\}$ thanks to \crefnoname{thm:LamCptness}.

Observe that each leaf of $\lam_\infty$ is complete away from $\{0\}$. We now argue that each leaf also has stable universal cover. Consider a leaf $L\in \lam_\infty$; then, by \cref{lem:StableUnivOrMultOne}, $L$ has stable universal cover or the convergence to $L$ is locally smooth with multiplicity one. However, if the second case occurs, the stability of the universal cover of the surfaces $r_j^{-1}({\hat\ms} - p)$ is inherited by the universal cover of $L$.

Hence, we can apply \cref{cor:StableLam} to obtain that ${\lam}_\infty$ consists of parallel planes or half-planes and thus, possibly further restricting the ball we are considering, we can improve the curvature estimate to (say)
\begin{equation}\label{eq:14estLam}
\abs{\A_{\hat\ms}}(x) d_g(x,p) \le \frac 14
\end{equation}
for all $x\in {\hat\ms} \subset B_{\eps_0}(p)\setminus\{p\}$.

We now want to prove that $\lam_\infty$ is either a single half-plane $\Delta$ passing through the origin and orthogonal to $\Pi(0)$ or $\Pi(0)$ itself.
If not the case, then there would exist another plane or half-plane not passing through the origin which appears in the (lamination) limit of the rescalings $r_j^{-1}(\hat\ms-p)$; hence one could define $\delta\in \oo0{\eps_0}$ sufficiently small such that $\hat\ms\cap (B_\delta(p)\setminus \{p\})$ contains a properly embedded component diffeomorphic to a disc or a half-disc, which does not contain $p$. As a result, we can choose $\delta$ (and possibly taking $\eps_0$ smaller) in such a way that $\hat\ms$ satisfies the assumptions of \cref{cor:14PuncturedBall} between radii $\delta$ and $\eps_0$ in a suitable Fermi chart.
Applying the corollary, we conclude that $\hat\ms$ itself contains a disc or half-disc (obtained, roughly speaking, by gluing the previous disc or half-disc with its corresponding `trivial' component of $\hat\ms\setminus B_\delta(p)$), but this is a contradiction since $\hat\ms$ is connected and its closure contains $p$.
Therefore $\lam_\infty$ consists of only one leaf\footnote{Note that $\lam_\infty$ cannot contain two intersecting (half-)planes, because this would contradict the embeddedness of ${\hat\ms}$.}, which is either $\Delta$ or $\Pi(0)$.

 Let us consider $\delta\in \oo0{\eps_0/3}$ such that $\hat\ms\cap (B_{3\delta}(p)\setminus B_\delta (p))$ intersects $\partial B_{2\delta}(p)$ transversely, is sufficiently (smoothly) close to the limit half-plane or plane and is a multigraph over it.
In particular $\hat\ms\cap \partial B_{2\delta}(p)$ is the union of injectively immersed curves.

If $\hat\ms\cap \partial B_{2\delta}(p)$ contains a compact curve, which can be either an $S^1$ or an arc, then $\hat\ms\cap B_{3\delta}(p)$ is a properly embedded topological punctured disc or half-disc in $B_{3\delta}(p)\setminus \{p\}$ thanks to estimate \eqref{eq:14estLam}, since we can apply \cref{cor:14PuncturedBall} (in Fermi chart) in the variant described in \cref{rem:14PunctBallVariant}.
Thus, using Proposition D.1 in \cite{ChoKetMax17} or its free boundary analogue (based on \cite[Theorem 4.1]{FraLi14}), respectively, we obtain that $\hat\ms$ extends smoothly across $\{p\}$.

The only other situation that can happen is that $\hat\ms\cap \partial B_{2\delta}(p)$ consists of one or more spiraling curves. Observe that this case can only occur if $\lam_\infty$ has $\Pi(0)$ as its only leaf since, otherwise,  $\Delta\cap (B_{3\delta}(p)\setminus B_\delta (p))$ is simply connected and thus it is not possible to see a spiraling behavior.
However, the method to deal with the spiraling situation in the case when $\lam_\infty=\left\{\Pi(0)\right\}$ is completely analogous to the one in the spiraling case in the proof for interior points in \cite[Proposition D.3]{ChoKetMax17}, therefore we can employ the same argument to conclude.
\end{proof}

We are now able to discuss the general case, when one needs to deal with isolated singularities arising when taking limits of free boundary minimal surfaces with bounded index.

\begin{theorem} \label{thm:RemSingLimLam}
Let $(\amb^3,g)$ be a compact three-dimensional Riemannian manifold with boundary.
Let $\ms_j^2\subset\amb$ be a sequence of properly embedded free boundary minimal surfaces with uniformly bounded index, which locally converge to a free boundary minimal lamination ${\hat\lam}$ in $\amb\setminus\set_\infty$ away from a finite set of points $\set_\infty$.
Then ${\hat\lam}$ extends smoothly through $\set_\infty$ to a smooth lamination $\lam$ of $\amb$.
Moreover, given any leaf $L\in\lam$, one of the following assertions holds:
\begin{enumerate} [label={\normalfont(\arabic*)}]
\item $L$ has stable universal cover; \label{rsll:stable}
\item the convergence of $\ms_j$ to $L$ is locally smooth with multiplicity one (in the sense of graphs); namely, for every $x\in L$ there exists a neighborhood $U$ of $x$ in $\amb$ such that $U\cap (\bigcup_{L'\in\lam} L') = U\cap L$ and $U\cap \ms_j$ converges smoothly to $U\cap L$ with multiplicity one (as graphs). \label{rsll:multone}
\end{enumerate}
\end{theorem}
\begin{proof}
We split the proof in two steps: in the first step we show that $\hat\lam$ extends through $\set_\infty$, while in the second step we prove properties \ref{rsll:stable} and \ref{rsll:multone} of the leaves of $\lam$.

\vspace{2mm}\textbf{Step 1.} Since the result is local, we can assume to work around a single point $p\in \set_\infty$.
First, let us prove that there exists $\eps_0>0$ sufficiently small such that each leaf of ${\hat\lam}$ in $B_{\eps_0}(p)$ has stable universal cover.

Thanks to \cref{lem:StableUnivOrMultOne}, either $\hat L\in \hat \lam$ has stable universal cover or the convergence of $\ms_j$ to $\hat L$ is locally smooth and graphical with multiplicity one. Observe that there can be only a finite number of unstable leaves of the second type, otherwise the uniform bound on the index of $\ms_j$ would be violated.
If we focus on \emph{one} such leaf, we can argue as follows.
Taking $\eps_0$ possibly smaller, we can assume that $B_{\eps_0}(p)$ is simply connected and thus $\Sigma_j\cap B_{\eps_0}(p)$ is two-sided\footnote{This is a general topological fact, whose proof can be found for example in \cite[Lemma C.1]{ChoKetMax17} (where the result is stated in a particular case but the proof is completely analogous).}. Hence, it is easily argued that (by virtue of the multiplicity one convergence) ${\hat L}\cap B_{\eps_0}(p)$ has to be two-sided as well. Now, a straightforward variation of the same argument as in \cite[Proposition 1]{Fis85} proves that we can pick $\eps_0$ even smaller, in such a way that $\hat L\cap (B_{\eps_0}(p)\setminus \{p\})$ is stable and, in addition, the Jacobi operator has a positive solution in the same domain. Hence, using such a function, we can derive that $\hat L\cap (B_{\eps_0}(p)\setminus \{p\})$ has, in fact, stable universal cover.

We have thus proved that there exists $\eps_0>0$ such that every leaf of $\hat\lam$ in $B_{\eps_0}(p)\setminus \{p\}$ has stable universal cover. Then we can apply \cref{prop:RemovableSingStab} to obtain that each leaf extends smoothly across $p$. Note that the extended leaves cannot meet at $p$ (since $\hat\lam$ is assumed to be a lamination in $M\setminus \mathcal{S}_{\infty}$).

It is only left to prove that $\lam$ obtained as union of the extended leaves has the structure of lamination around the point $p$. This fact follows from the proof of \cref{prop:RemovableSingStab}, where it was shown that that for any sequence of positive numbers $r_j\to 0$ we have that $r_j^{-1}(\lam-p)$ converge to a lamination consisting of parallel planes or half-planes. 

\vspace{2mm}\textbf{Step 2.} If $L$ is a limit leaf of $\lam$ (see \cref{rem:LimLeaf}), then $L$ has stable universal cover by standard arguments (cf. also \cref{lem:StableUnivOrMultOne}).
On the other hand, if the convergence of $\ms_j$ to $L$ is locally smooth with multiplicity one (in the sense of graphs) away from $\set_\infty$, then the convergence extends across $\set_\infty$ thanks to \cref{lem:MultOneConvExtends} and we end up in case \ref{rsll:multone}.

Therefore, let us assume that $L$ is not a limit leaf and that $\ms_j$ does not converge to $L$ with multiplicity one.
Possibly passing to the double cover, we can assume that $L$ is two-sided (cf. \cite[Section 6]{AmbCarSha18-Compactness}). We then show that $L$ admits a positive Jacobi function, which proves that $L$ is stable as well as its universal cover.

Let us consider a regular domain $\Omega\Subset L\setminus \set_\infty$. Consider a vector field $X\in \vf_\partial(M)$ that has unit length and is normal to $L$ along $L$. Denote by $\Phi(x,t)$ the flow associated to $X$ and, for every $\eps>0$, define
\[
\Omega_\eps\eqdef \{ \Phi(x,t) \st x\in\Omega\comma \abs t< \eps \} \point
\]
Observe that we can fix $\eps>0$ such that the convergence of $\ms_j$ to $\lam$ is smooth in the sense of laminations in $\Omega_\eps$ and such that the only component of $\lam$ in $\Omega_\eps$ is $L\cap \Omega_\eps$.

By definition of convergence in the sense of laminations, $\Omega_\eps\setminus \Omega_{\eps/2}$ does not intersect $\ms_j$ for $j$ sufficiently large, since it does not intersect any leaf of $\lam$.
Then, for every $j$, define
\[
u_j^+(x) \eqdef \sup  \{t\in \oo{-\eps}\eps \st \Phi(x,t)\in \ms_j\} \quad \text{and} \quad u_j^-(x) \eqdef \inf \{t\in \oo{-\eps}\eps \st \Phi(x,t)\in \ms_j\} \point
\]
Note that $-\eps/2< u_j^-(x) < u_j^+(x) < \eps/2$, since $\ms_j\cap(\Omega_\eps\setminus\Omega_{\eps/2}) = \emptyset$ and $\ms_j$ does not converge to $L$ with multiplicity one.
Furthermore observe that, by \cref{rem:ConvOfLam} together with the compactness of $\Omega$, for every $j$ sufficiently large (depending on $\Omega$) the surfaces $\ms_j^\pm \eqdef \{ \Phi(x, u_j^\pm(x)) \st x\in\Omega \}\subset \ms_j$ are well-defined smooth surfaces (with boundary) that converge uniformly smoothly to $\Omega$.

At this point, we can go through the very same argument as in the proof of Theorem 5 in \cite[Section 6]{AmbCarSha18-Compactness}, so we just sketch it briefly.

Fixing $x_0\in\Omega\setminus\partial\amb$, define $\tilde h_j\eqdef u_j^+ - u_j^->0$ and $h_j(x)\eqdef \tilde h_j(x_0)^{-1} \tilde h_j(x)$. Then, following exactly the same proof of Claim 1 in \cite[Section 6]{AmbCarSha18-Compactness}, we have that $h_j$ is bounded in $C^l(\Omega)$ for all $l\in\N$ and converges smoothly to $h\in C^\infty(\Omega)$, solution of
\begin{equation} \label{eq:JacEqSub}
\begin{cases}
\jac_L (h) = \lapl_L h +(\frac 12 \Scal_g+\frac 12 \abs A^2 - K)h = 0 & \text{in $\Omega$}\comma\\
\frac{\partial h}{\partial \eta} = - \II(\nu,\nu)h & \text{on $\partial\amb\cap\Omega$}\point
\end{cases}
\end{equation}
Then, by taking an exhaustion of $L\setminus\set_\infty$ by domains $\Omega\Subset L\setminus\set_\infty$ containing $x_0$, we obtain a function $h\in C^\infty(L\setminus\set_\infty)$ solving \eqref{eq:JacEqSub} in $L\setminus\set_\infty$. Moreover $h(x_0)=1$ and $h\ge 0$.
Finally, thanks to the same proof of Claim 2 in \cite[Section 6]{AmbCarSha18-Compactness}, it holds that $h$ is uniformly bounded and thus extends to a smooth Jacobi function on all $L$, which is positive everywhere thanks to the maximum principle and the Hopf boundary point lemma. 
\end{proof}

In the setting of Euclidean half-spaces, \cref{thm:RemSingLimLam} reads as follows.
\begin{proposition}\label{prop:LimLamInR3}
Let $g_j$ be a sequence of Riemannian metrics in $\R^3$ that converges, locally smoothly, to the Euclidean metric.
Let $\ms_j^2\subset \Xi(a_j) \subset \R^3$ for some $0\ge a_j\ge -\infty$ be a sequence of properly embedded, \emph{edged}, free boundary minimal surfaces in $(\Xi(a_j),g_j)$, such that for every $j\in\mathbb{N}$ we have $\Sigma_j\subset\Delta_j$ for a sequence of compact domains $\Delta_j$ exhausting $\Xi(a)$ for $a=\lim_{j\to\infty} a_j$, each being the intersection of a smooth domain of $\R^3$ with $\Xi(a_j)$.
Assume that the surfaces $\ms_j$ have index bounded by $I\in\N$ and locally converge (in the sense of laminations) to a free boundary minimal lamination $\hat\lam\subset \Xi(a)\setminus\set_\infty$ away from a finite set of points $\set_\infty$.
Then $\hat\lam$ extends smoothly through $\set_\infty$ to a free boundary minimal lamination $\lam\subset\Xi(a)$ and the following dichotomy holds:
\begin{enumerate} [label={\normalfont(\arabic*)}]
\item $\lam$ consists of parallel planes or half-planes; \label{llir:stable}
\item $\lam$ is a complete, non-flat, connected, properly embedded, free boundary minimal surface in $\Xi(a)$ of (positive) index at most $I$ and (a subsequence of) $\ms_j$ converges to $\lam$ locally smoothly (in the sense of graphs) with multiplicity one. \label{llir:multone}
\end{enumerate}
\end{proposition}
\begin{proof}
The first part of the statement follows directly from \cref{thm:RemSingLimLam}, so let us prove properties \ref{llir:stable} and \ref{llir:multone}.
Let us consider a leaf $L\in\lam$. By \cref{thm:RemSingLimLam}, $L$ has stable universal cover or the convergence of $\ms_j$ to $L$ is locally smooth with multiplicity one (in the sense of graphs).
In the first case, $L$ must be a plane by Corollary 22 in \cite{AmbBuzCarSha18}. In the second case, $L$ has index bounded by $I$, otherwise the bound on the index of $\ms_j$ would be violated. In particular, all the leaves of $\lam$ have bounded index, thus the result is just the free boundary analogue of Corollary B.2 in \cite{ChoKetMax17}, from which our result follows using \cref{lem:ReflectionPrinciple} and \cref{cor:MinSurfInR3WithFiniteIndex}, which is needed (in particular) to ensure that the reflected minimal lamination still has finite index.
\end{proof}

%%%%%%%%%%%%%%%%%%%%%%%%%%%%%%%%%%%%%%%%%%%%%%%%%%%%%%%%%%%%%%%%
%%% Macroscopic behavior
%%%%%%%%%%%%%%%%%%%%%%%%%%%%%%%%%%%%%%%%%%%%%%%%%%%%%%%%%%%%%%%%

\section{Macroscopic behavior} \label{sec:MacroDescr}

In this section we study the behavior at macroscopic scale of a sequence of free boundary minimal surfaces with bounded index. 

\subsection{Smooth blow-up sets}

Similarly to \cite[Section 2.1.1]{ChoKetMax17}, we first need to define a finite set of `bad points', away from which a sequence of free boundary minimal surfaces with bounded index is well-controlled.

\begin{definition} \label{def:BlowUpSets}
Let $(\amb^3,g)$ be a compact three-dimensional Riemannian 
manifold and suppose that $\ms_j^2\subset\amb$ is a sequence 
of properly embedded \emph{edged} free boundary minimal surfaces.
A sequence of finite sets of points $\set_j \subset \ms_j$ is said to be a 
\emph{sequence of smooth blow-up sets} if
\begin{enumerate} [label={\normalfont(\roman*)}]
\item The second fundamental form blows up at points in $\set_j$, i.e. \[\liminf_{j\to\infty}\, \min_{p\in \set_j}\ \abs{\A_{\ms_j}}(p) = \infty\point\]
\item Chosen a sequence of points $p_j \in \set_j$, the 
rescaled surfaces $\bu \ms_j \eqdef \abs{\A_{\ms_j}}(p_j) (\ms_j - p_j)$ 
converge (up to subsequence) locally smoothly with multiplicity one (in the sense of graphs) to a complete, non-flat, properly embedded, free 
boundary minimal surface $\bu \ms_\infty \subset \Xi(a)$, for some $0\ge a \ge -\infty$, satisfying 
\[\abs{\A_{\bu \ms_\infty}}(x) \le \abs{\A_{\bu \ms_\infty}} (0)\] for all $x\in \bu 
\ms_\infty$. \label{bus:BlowUpLimit}
\item Points in $\set_j$ do not appear in the blow-up limit of other points in $\set_j$, i.e. \[\liminf_{j\to\infty}\, \min_{p\neq q\in \set_j}\ \abs{\A_{\ms_j}}(p) d_{g_j}(p,q) = 
\infty\point\] \label{bus:PointsAreFar}
\end{enumerate}
\end{definition}

\subsection{Curvature estimates}

The starting point for our study is a curvature estimate for stable minimal surfaces, here stated in the free boundary setting.
\begin{theorem}[{\cite[Theorem 1.2]{GuaLiZho17}}]\label{thm:CurvEstStable}
Let $(\amb^3, g)$ be a compact Riemannian manifold with boundary. Then there 
exists a constant $C =C(\amb,g)$ such that, if $\ms^2\subset \amb$ is a 
compact, properly embedded, stable, \emph{edged}, free boundary minimal surface, then
\[
\sup_{x\in \ms}\ \abs{\A_\ms}(x) \min\{ 1, d_g(x, \partial \ms \setminus \partial 
\amb) \} \le C \point
\]
\end{theorem}
\begin{remark}
Actually, in \cite{GuaLiZho17} the theorem is stated only for \emph{non-edged} free 
boundary minimal surfaces, but it is observed in Remark 1.3 therein
that the conclusion still holds in such more general setting.
\end{remark}

We will need the following extension of the previous result.

\begin{lemma} \label{lem:CurvEstBoundedInd}
Let $(\amb^3,g)$ be a compact Riemannian manifold with boundary and fix $I\in 
\N$. Suppose that $\ms_j^2\subset \amb$ is a sequence of compact, properly 
embedded, \emph{edged}, free boundary minimal surfaces with $\ind(\ms_j) \le I$. Then, up 
to subsequence, there exist a constant $C=C(\amb,g)$ and a sequence of smooth 
blow-up sets $\set_j\subset \ms_j$ with $\abs {\set_j} \le I$ and
\[
\sup_{x\in \ms_j} \ \abs{\A_{\ms_j}}(x) \min\{ 1, d_g(x, \set_j \cup (\partial \ms_j 
\setminus \partial 
\amb)) \} \le C \point
\]
\end{lemma}
\begin{remark}
Note that the reason to state the lemma with \emph{edged} free boundary minimal surfaces is to perform the inductive procedure in the proof.
\end{remark}
\begin{proof}
We proceed by induction on $I\in \N$. If $I=0$ the statement is exactly 
\crefnoname{thm:CurvEstStable}, thus assume $I>0$.
Note that we can suppose that \[\rho_j \eqdef \max_{x\in\ms_j}\ \abs{\A_{\ms_j}}(x) \min\{ 1, d_g(x, \partial 
\ms_j \setminus \partial \amb) \} \to \infty\comma \] for otherwise one can take $\set_j=\emptyset$ and there is nothing left to do.
Take $q_j \in \ms_j$, a point where $\rho_j$ is attained, and define $R_j 
\eqdef d_g(q_j, \partial 
\ms_j \setminus \partial \amb)/2$.

\begin{figure}[htpb]
\centering
%%%%%%%%%%%%%%%%%%%%%%%%%%%%%%%%%%%%%%%%%%%%%%%%%%%%%
%%% pick_pt.tikz
%%%%%%%%%%%%%%%%%%%%%%%%%%%%%%%%%%%%%%%%%%%%%%%%%%%%%
\begin{tikzpicture}
\pgfmathsetmacro{\R}{3}
\pgfmathsetmacro{\r}{\R*(1-sqrt(37/100))}
\pgfmathsetmacro{\k}{1.2}
\coordinate(Q)at(0,0);
\coordinate(P)at(-0.6*\R,0.1*\R);

\coordinate(C1)at(-1.3*\R*\k,0*\k);
\coordinate(C2)at(-1.1*\R,0.1*\R);
\coordinate(C3)at(-0.7*\R,-0.2*\R);
\coordinate(C4)at(-0.1*\R,-0.05*\R);
\coordinate(C5)at(0.3*\R,-0.25*\R);
\coordinate(C6)at(0.8*\R,0.2*\R);
\coordinate(C7)at(1.2*\R*\k, 0.1*\R*\k);

\coordinate(D1)at(1.25*\R*\k,0.6*\R*\k);
\coordinate(D2)at(1.1*\R*\k,1.05*\R*\k);
\coordinate(D3)at(0.8*\R*\k,1.3*\R*\k);
\coordinate(D4)at(0.1*\R*\k,1.2*\R*\k);
\coordinate(D5)at(-0.3*\R*\k,1.4*\R*\k);
\coordinate(D6)at(-0.7*\R*\k,1.1*\R*\k);
\coordinate(D7)at(-1.2*\R*\k, 0.8*\R*\k);

\pgfmathsetmacro{\ta}{12}
\pgfmathsetmacro{\tb}{179}
\draw[myGray] ({\R*cos(\ta)}, {\R*sin(\ta)}) arc (\ta:\tb:\R);
\draw[dashed, myGray] ({\R*cos(\tb)}, {\R*sin(\tb)}) arc (\tb:360+\ta:\R);

\pgfmathsetmacro{\sa}{8}
\pgfmathsetmacro{\sb}{160}
\draw[myGray] ({-0.6*\R+\r*cos(\sa)}, {0.1*\R+\r*sin(\sa)}) arc (\sa:\sb:\r);
\draw[dashed, myGray] ({-0.6*\R+\r*cos(\sb)}, {0.1*\R+\r*sin(\sb)}) arc (\sb:360+\sa:\r);

\fill (Q) circle[radius=0.1em];
\node[right=3pt] at (Q) {$q_j$};
\fill (P) circle[radius=0.1em];
\node[left] at (P) {$p_j$};

\draw[black!80] (Q)--({cos(60)*\R},{sin(60)*\R}) node [midway,right] {$R_j$};
\draw[black!80] (P)--({-0.6*\R+cos(80)*\r},{0.1*\R+sin(80)*\r}) node [midway,right] {$r_j$};

\draw[thick] plot [smooth, tension=0.7] coordinates {(C1) (C2) (C3) (C4) (C5) (C6) (C7)};
\draw[black!80] plot [smooth, tension=0.7] coordinates {(C7) (D1) (D2) (D3) (D4) (D5) (D6) (D7) (C1)};

\node[below=2pt] at (C5) {$\partial\ms_j\cap\partial \amb$};
\node[above left=-1pt] at (D6) {$\partial\ms_j\setminus\partial \amb$};

\node[below right=20pt] at (D4) {$\ms_j$};
\end{tikzpicture}
%%%%%%%%%%%%%%%%%%%%%%%%%%%%%%%%%%%%%%%%%%%%%%%%%%%%%
%%%%%%%%%%%%%%%%%%%%%%%%%%%%%%%%%%%%%%%%%%%%%%%%%%%%%
\caption{Point picking argument.} \label{fig:PickPt}
\end{figure}
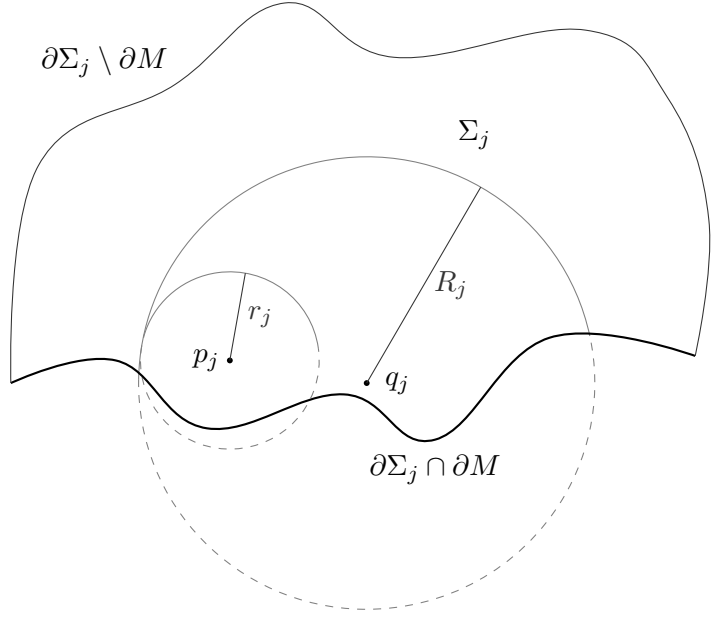

Now consider $p_j\in \ms_j \cap B_{R_j}(q_j)$ which 
realizes
\[
\max_{x\in \ms_j \cap B_{R_j}(q_j)}\, \abs{\A_{\ms_j}}(x) d_{g}(x, \partial B_{R_j}(q_j)\setminus \partial \amb)
\]
and define $r_j \eqdef d_{g}(p_j,  \partial B_{R_j}(q_j) \setminus\partial \amb)$ and 
$\lambda_j \eqdef \abs{\A_{\ms_j}}(p_j)$.
Note that 
\[
\begin{split}
\lambda_jr_j &= \abs{\A_{\ms_j}}(p_j) d_{g}(p_j,  \partial B_{R_j}(q_j)\setminus\partial\amb) \ge  \abs{\A_{\ms_j}}(q_j) d_{g}(q_j,  \partial B_{R_j}(q_j)\setminus \partial\amb) \\ &= \abs{\A_{\ms_j}}(q_j) R_j = \frac 12 \abs{\A_{\ms_j}}(q_j)d_g(q_j, 
\partial 
\ms_j \setminus \partial \amb) \ge \frac{\rho_j}2 \to\infty \point
\end{split}
\]
Thus in particular $\abs{\A_{\ms_j}}(p_j) \min\{ 1, d_g(p_j, \partial 
\ms_j \setminus \partial \amb) \}\to \infty$, where we are using that $d_g(p_j, 
\partial 
\ms_j \setminus \partial \amb)\ge r_j$, since there are no points of $\partial 
\ms_j\setminus \partial \amb$ in $B_{R_j}(q_j)$.

Now, we have that $d_{g}(x, \partial B_{r_j}(p_j) \setminus\partial\amb) \le 
d_{g}(x, \partial B_{R_j}(q_j)\setminus\partial\amb)$ for every $x\in B_{r_j}(p_j)$ with equality in $p_j$. Therefore $p_j$ also realizes 
\[
\max_{x\in \ms_j \cap B_{r_j}(p_j)}\, \abs{\A_{\ms_j}}(x) d_{
g}(x, \partial B_{r_j}(p_j)\setminus\partial\amb) \point
\]

Hence, we can now perform a blow-up argument around the points 
$p_j$. In particular, let us define 
$\bu \ms_j \eqdef \lambda_j(\ms_j - p_j)$, a surface in the manifold $\amb_j \eqdef 
\lambda_j (\amb - p_j)$ endowed with the rescaled\footnote{One can think of $M$ as isometrically embedded in some $\R^K$ and consider the blow-ups in this Euclidean space. Thus, in particular, we have that $g_j(x)=\lambda_j g(x+p_j)$.} metric $g_j$. Then we have that
\[
\abs{\A_{\bu \ms_j}}(x)\, d_{
g_j}(x, \partial B_{\lambda_jr_j}(0)\setminus \partial M_j) \le \lambda_jr_j \to\infty
\] 
for all $x \in \bu \ms_j \cap B_{\lambda_jr_j}(0)$. 
Note that here $B_{\lambda_jr_j}(0)$ is the ball in $\amb_j$ with respect to the metric $g_j$. We do not write explicitly the dependence on $j$ since it is always clear from the context.

Hence, fixing $R>0$, for all $x\in \bu \ms_j \cap B_{R}(0)$ it 
holds that
\begin{equation} \label{eq:BehaviourSecFundBlowUp}
\abs{\A_{\bu \ms_j}}(x) \le \frac{\lambda_jr_j}{\lambda_jr_j - R} \nearrow 1
\end{equation}
as $j \to \infty$. Moreover, observe that $\abs{\A_{\bu \ms_j}}(0) = 1$ and the domains $B_{\lambda_jr_j}(0) \subset \amb_j$ do not contain points of $\partial\bu\ms_j\setminus \partial\amb_j$, since $\ms_j$ has no points of 
$\partial \ms_j\setminus \partial \amb$ in $B_{r_j}(p_j)$.
By \crefnoname{thm:LamCptness} together with \cref{prop:LimLamInR3} (we are in case \ref{llir:multone} since the limit lamination cannot be flat), this implies that the surfaces $\bu\ms_j$ converge locally smoothly (in the sense of graphs) with multiplicity one to a properly embedded free boundary minimal surface $\bu\ms_\infty\subset \Xi(a)$ (for some $-\infty\leq a\leq 0$). Furthermore, the index of $\bu\ms_\infty$ is strictly positive and less or equal than $I$.

Hence, with a standard argument as in \cite[Proposition 1]{Fis85}, there exists $R_0>0$ such that $\bu \ms_\infty \cap B_{R_0}(0)$ has 
index greater than $0$ and $\bu \ms_\infty\setminus B_{R_0}(0)$ is stable. 
Moreover we can assume, without loss of generality, that $\bu \ms_\infty$ intersects $\partial B_{R_0}(0)$ 
transversely and, thanks to Proposition 28 in \cite{AmbBuzCarSha18}, also that 
\begin{equation}\label{eq:SmallSecFundSigmaInfty}
\abs{\A_{\bu \ms_\infty}}(x) \le 1/4
\end{equation}
on $\bu \ms_\infty\setminus B_{R_0}(0)$.

Now consider $\ms_j' \eqdef \ms_j\setminus B_{R_0/\lambda_j}(p_j)$. Choosing 
$j$ sufficiently large, we can assume that $\partial B_{R_0/\lambda_j}(p_j)$ intersects $\ms_j$ transversely and that the ball $B_{R_0/\lambda_j}(p_j)$ does not intersect the portion of the boundary of $\ms_j$ that is not contained in 
$\partial \amb$. Indeed we know that $\lambda_j d_{g}(p_j, \partial 
\ms_j\setminus \partial\amb)\to\infty$.
Therefore $\ms_j'$ is a sequence of manifolds which still fulfills the 
assumptions of the lemma, but with $\ind(\ms_j')\le I-1$ for $j$ 
sufficiently large.
Hence, by the inductive hypothesis, up to subsequence there exist a constant 
$C'>0$ and a sequence of smooth blow-up sets $\set_j'\subset \ms_j'$ 
with $\abs{\set_j'} \le I-1$ and
\begin{equation} \label{eq:EstimateInductive}
 \abs{\A_{\ms_j'}} (x) \min\{ 1, d_{ g}(x,  \set_j'\cup (\partial 
\ms_j'\setminus\partial\amb)) \} \le C' \point
\end{equation}

We now want to show that $\set_j\eqdef  \set_j'\cup \{ p_j \}$ is the desired 
sequence of blow-up sets. The only non-obvious point to check is 
point \ref{bus:PointsAreFar} of \cref{def:BlowUpSets}, for which it suffices to verify that
\[
\lim_{j\to\infty}\, \min_{q\in \set_j'}\ \abs{\A_{\ms_j}}(p_j) d_{g}(p_j,q) = 
\lim_{j\to\infty}\, \min_{q\in \set_j'}\ \abs{\A_{\ms_j}}(q) d_{g}(p_j,q) = 
\infty \point
\]
We first observe that indeed
\[
\liminf_{j\to\infty}\, \min_{q\in \set_j'}\ \abs{\A_{\ms_j}}(q) d_{g}(p_j,q)\geq  \liminf_{j\to\infty}\, \min_{q\in \set_j'}\ \abs{\A_{\ms_j}}(q) d_{g}(\partial \ms'_j\setminus\partial M,q)=
\infty 
\]
based on \ref{bus:BlowUpLimit} for $q\in\set'_j$ (the associated limit surface is \emph{not edged}). However, thanks to \eqref{eq:BehaviourSecFundBlowUp} and \eqref{eq:SmallSecFundSigmaInfty}, we derive
 $\abs{\A_{\ms_j}}(q) \le 
\abs{\A_{\ms_j}}(p_j)/2$ for all $q\in \set_j'$ provided one takes $j$ sufficiently large; hence $\min_{q\in \set_j'}\ \abs{\A_{\ms_j}}(p_j) d_{g}(p_j,q)$ could not be uniformly bounded either.

Therefore, it remains to check that there exists $C>0$ such that
\[
\abs{\A_{\ms_j}} (x) \min\{ 1, d_{g}(x, \set_j\cup (\partial \ms_j\setminus \partial\amb)) \} \le 
C
\]
for all $x\in \ms_j$.
This inequality easily holds on $B_{R_0/\lambda_j}(p_j)$, thus it is sufficient 
to check it for points $x\in \ms_j'$.
Assume by contradiction that there exists a sequence of points $z_j\in \ms_j'$ such that
\begin{equation} \label{eq:ContradictionAssumption}
\limsup_{j\to\infty}\ \abs{\A_{\ms_j}}(z_j) \min\{ 1, d_{g}(z_j, \set_j\cup 
(\partial \ms_j\setminus \partial \amb)) \} = \infty\point
\end{equation}
First observe that $\liminf_{j\to\infty}\lambda_jd_g(z_j,p_j) = \infty$, because otherwise $\lambda_j(z_j-p_j)$ would converge to some point $\bar z\in \bu\ms_\infty$ and we would obtain
\[
\begin{split}
\limsup_{j\to\infty}\ &\abs{\A_{\ms_j}}(z_j) \min\{ 1, d_{g}(z_j, \set_j\cup 
(\partial \ms_j\setminus \partial \amb)) \} \le \limsup_{j\to\infty}\ \abs{\A_{\ms_j}}(z_j) d_{g}(z_j,p_j) \\
& = \limsup_{j\to\infty}\ \abs{\A_{\bu\ms_j}}(\lambda_j(z_j-p_j))  d_{g_j}(\lambda_j(z_j-p_j),0) = \abs{\A_{\bu\ms_\infty}}(\bar z) d_{{\R^3}}(\bar z,0) < \infty \point
\end{split}
\]
Moreover, since both \eqref{eq:EstimateInductive} and \eqref{eq:ContradictionAssumption} hold and $[\set_j'\cup(\partial\ms_j'\setminus\partial\amb)]\setminus [\set_j\cup(\partial\ms_j\setminus\partial\amb)] = \partial\ms_j'\setminus\partial\ms_j$, we have that
\[
d_g(z_j,\set_j'\cup(\partial\ms_j'\setminus\partial\amb)) = d_g(z_j,\partial\ms_j'\setminus\partial\ms_j) = d_g(z_j,p_j) - \frac{R_0}{\lambda_j} \point
\]
Thus, we can conclude that
\[
\begin{split}
\limsup_{j\to\infty}\ &\abs{\A_{\ms_j}}(z_j) \min\{ 1, d_{g}(z_j, \set_j\cup 
(\partial \ms_j\setminus \partial \amb)) \} \le\limsup_{j\to\infty}\ \abs{\A_{\ms_j}}(z_j) d_{g}(z_j,p_j) \\
&\le\limsup_{j\to\infty}\ \frac{C'}{d_g(z_j,\set_j'\cup(\partial\ms_j'\setminus\partial\amb))}d_{g}(z_j,p_j) = \limsup_{j\to\infty}\ \frac{C'}{d_{g}(z_j,p_j) - R_0/\lambda_j}d_{g}(z_j,p_j) \\
& = \limsup_{j\to\infty}\ \frac{C'}{1 - R_0/(\lambda_jd_{g}(z_j,p_j) )} = C'\comma
\end{split}
\]
which is a contradiction and completes the proof.
\end{proof}

Given the previous lemma and the tools to handle free boundary minimal laminations presented in \crefnoname{sec:FBMLam}, we can conclude the description of the limit picture at macroscopic scale.
\begin{corollary} \label{cor:ExistenceBlowUpSetAndCurvatureEstimate}
Let $(\amb^3,g)$ be a compact Riemannian manifold with boundary and fix $I\in 
\N$. Suppose that $\ms_j^2\subset \amb$ is a sequence of compact, properly 
embedded, free boundary minimal surfaces with $\ind(\ms_j) \le I$. Then, up 
to subsequence, there exist a constant $C=C(\amb,g)$ and a sequence of smooth 
blow-up sets $\set_j\subset \ms_j$ with $\abs {\set_j} \le I$ and
\begin{equation}\label{eq:CurvEstSj}
\sup_{x\in \ms_j} \ \abs{\A_{\ms_j}}(x) \min\{ 1, d_g(x, \set_j ) \} \le C \point
\end{equation}
Moreover, the sets $\set_j$ converge to a set of points $\set_\infty$ and the surfaces $\ms_j$ converge locally smoothly away from $\set_\infty$ to some smooth free boundary minimal lamination $ \lam$ of $\amb$. Furthermore, if the ambient manifold satisfies \hypP{} then $\partial L = L \cap \partial\amb$ for all $L\in\lam$. 
\end{corollary}

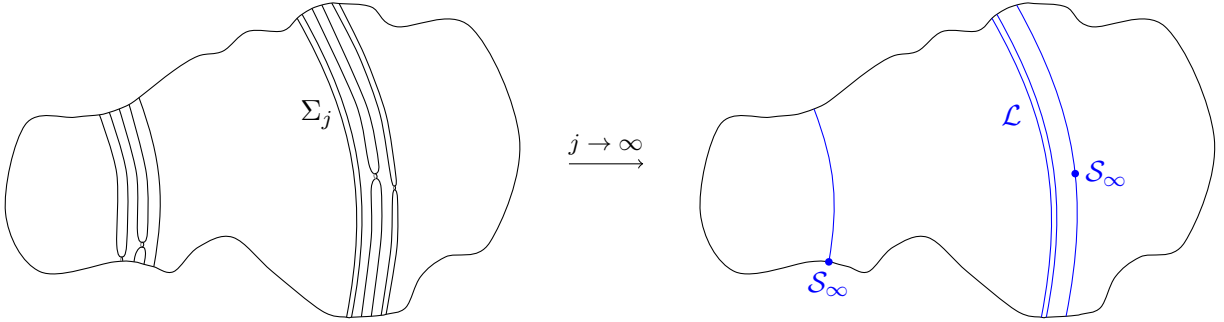
\begin{figure}[htpb]
\centering
%%%%%%%%%%%%%%%%%%%%%%%%%%%%%%%%%%%%%%%%%%%%%%%%%%%%%
%%% macro_descr.tikz
%%%%%%%%%%%%%%%%%%%%%%%%%%%%%%%%%%%%%%%%%%%%%%%%%%%%%
\begin{tikzpicture}[scale=0.65]
\coordinate (A1) at (0,0); 
\coordinate (B1) at (0.2,0.05); 
\coordinate (C1) at (0.4,0.12);
\coordinate (D1) at (0.6,0.21);
\coordinate (E1) at (0.8,0.35);

\coordinate (F1) at (4,1.97);
\coordinate (G1) at (4.1,2.05);
\coordinate (H1) at (4.3,2.17);
\coordinate (I1) at (4.5,2.24);
\coordinate (J1) at (4.7,2.28);
\coordinate (K1) at (4.85,2.27);

\coordinate (F2) at (5,-4.13);
\coordinate (G2) at (5.1,-4.13);
\coordinate (H2) at (5.3,-4.12);
\coordinate (I2) at (5.5,-4.11);
\coordinate (J2) at (5.7,-4.09);
\coordinate (K2) at (5.8,-4.07);

\coordinate (C2) at (0.7,-3);
\coordinate (D2) at (0.9,-3.05);
\coordinate (E2) at (1.1,-3.1);

\coordinate (X) at (-1.5,-0.2);
\coordinate (Y) at (-1.2, -3.2);
\coordinate (U) at (7.5, 1.8);
\coordinate (V) at (7.9, -2.6);

\draw plot [smooth cycle, tension=0.6] coordinates {(A1) (B1) (C1) (D1) (E1)
			      (1.5,0.9) (2,1.2) (2.5,1.3) (3,1.7) (3.7,1.71)
			      (F1) (G1) (H1) (I1) (J1) (K1)
			      (5.5,2.2) (6,1.8)
			      (U) (8.5,-0.5) (V)
			      (7,-2.9) (6.5, -3.3) (6.2, -3.93) 
			      (K2) (J2) (I2) (H2) (G2) (F2)
			      (4.5,-4) (3.5,-3.2) (2.7, -2.5) (2,-2.7) (1.5, -3.2) 
			      (E2) (D2) (C2)
			      (0.3, -3)
			      (Y) (-1.9,-2) (X) (-0.5,-0.05)};
			      
\draw (F1) to [bend left=20] (F2);
\draw (G1) to [bend left=20] (G2);

\draw plot [smooth, tension = 0.5] coordinates {(H1) (5.2,0.2) (5.45,-1) (5.57, -1.2) (5.65, -1) (5.4,0.2) (I1)};
\draw plot [smooth, tension = 0.5] coordinates {(H2) (5.5,-2.5) (5.5,-1.5) (5.6, -1.3) (5.7, -1.5) (5.7,-2.5) (I2)};
\pgfmathsetmacro{\t}{atan(0.03/0.1)}
\begin{scope}[rotate around={\t:(5.57,-1.2)}]
\draw[gray] (5.57, -1.2) arc (90:270:0.03 and {sqrt(0.03*0.03+0.1*0.1)/2});
\draw (5.57, -1.2) arc (90:-90:0.03 and {sqrt(0.03*0.03+0.1*0.1)/2});
\end{scope}

\draw plot [smooth, tension = 0.5] coordinates {(J1) (5.25,1.2) (5.68,0) (5.9,-1.29) (5.97, -1.45) (6, -1.27) (5.8,0) (5.37,1.2) (K1)};
\draw plot [smooth, tension = 0.5] coordinates {(J2)  (5.9,-2.7) (5.93, -1.7) (5.97, -1.53) (6.03,-1.7) (6.02, -2.7) (K2)};
\draw[gray] (5.97, -1.45) arc (90:270:0.02 and {0.04});
\draw (5.97, -1.45) arc (90:-90:0.02 and {0.04});

\draw[gray] (0.45, -2.9) arc (90:182:0.02 and {0.08});
\draw (0.45, -2.9) arc (90:0:0.02 and {0.08});
\draw plot [smooth, tension = 0.55] coordinates {(A1) (0.35,-1.2) (0.35,-2.65) (0.45, -2.9) (0.55, -2.65) (0.55,-1.2) (B1)};

\pgfmathsetmacro{\t}{atan(0.01/0.1)}
\begin{scope}[rotate around={-\t:(0.87, -2.6)}]
\draw[gray] (0.87,-2.6) arc (90:270:0.03 and {sqrt(0.01*0.01+0.1*0.1)/2});
\draw (0.87,-2.6) arc (90:-90:0.03 and {sqrt(0.01*0.01+0.1*0.1)/2});
\end{scope}
\draw plot [smooth, tension = 0.55] coordinates {(C1) (0.75,-1.2) (0.77,-2.4) (0.87, -2.6) (0.97, -2.4) (0.95,-1.2) (D1)};
\draw plot [smooth, tension = 0.9] coordinates {(C2) (0.75, -2.8) (0.86, -2.7) (0.93, -2.8) (D2)};

\draw (E1) to [bend left = 15] (E2);

\node at (4.4,0) {$\ms_j$};

\draw[->] (9.5,-1) -- (11,-1);
\node at (10.25,-0.6) {\footnotesize $j\to\infty$};

\begin{scope}[xshift=400]
\coordinate (A1) at (0,0); 
\coordinate (B1) at (0.2,0.05); 
\coordinate (C1) at (0.4,0.12);
\coordinate (D1) at (0.6,0.21);
\coordinate (E1) at (0.8,0.35);

\coordinate (F1) at (4,1.97);
\coordinate (G1) at (4.1,2.05);
\coordinate (H1) at (4.3,2.17);
\coordinate (I1) at (4.5,2.24);
\coordinate (J1) at (4.7,2.28);
\coordinate (K1) at (4.85,2.27);

\coordinate (F2) at (5,-4.13);
\coordinate (G2) at (5.1,-4.13);
\coordinate (H2) at (5.3,-4.12);
\coordinate (I2) at (5.5,-4.11);
\coordinate (J2) at (5.7,-4.09);
\coordinate (K2) at (5.8,-4.07);

\coordinate (C2) at (0.7,-3);
\coordinate (D2) at (0.9,-3.05);
\coordinate (E2) at (1.1,-3.1);

\coordinate (X) at (-1.5,-0.2);
\coordinate (Y) at (-1.2, -3.2);
\coordinate (U) at (7.5, 1.8);
\coordinate (V) at (7.9, -2.6);

\draw plot [smooth cycle, tension=0.6] coordinates {(A1) (B1) (C1) (D1) (E1)
			      (1.5,0.9) (2,1.2) (2.5,1.3) (3,1.7) (3.7,1.71)
			      (F1) (G1) (H1) (I1) (J1) (K1)
			      (5.5,2.2) (6,1.8)
			      (U) (8.5,-0.5) (V)
			      (7,-2.9) (6.5, -3.3) (6.2, -3.93) 
			      (K2) (J2) (I2) (H2) (G2) (F2)
			      (4.5,-4) (3.5,-3.2) (2.7, -2.5) (2,-2.7) (1.5, -3.2) 
			      (E2) (D2) (C2)
			      (0.3, -3)
			      (Y) (-1.9,-2) (X) (-0.5,-0.05)};
			      
\draw[blue] (C1) to [bend left = 15] (C2);
\draw[blue] (F1) to [bend left = 20] (F2);
\draw[blue] (G1) to [bend left = 20] (G2);
\draw[blue] (I1) to [bend left = 20] (I2);

\fill[blue] (C2) circle [radius=0.2em] node[below] {$\set_\infty$};
\fill[blue] (5.685, -1.2) circle [radius=0.2em] node[right] {$\set_\infty$};
\node[blue] at (4.4,0) {$\lam$};
\end{scope}
\end{tikzpicture}
%%%%%%%%%%%%%%%%%%%%%%%%%%%%%%%%%%%%%%%%%%%%%%%%%%%%%
%%%%%%%%%%%%%%%%%%%%%%%%%%%%%%%%%%%%%%%%%%%%%%%%%%%%%
\caption{Macroscopic description of degeneration.} \label{fig:MacroDescr}
\end{figure}

\begin{proof}
The first part of the statement is a special case of \cref{lem:CurvEstBoundedInd}. Then, possibly extracting a further subsequence, we can assume that the sets $\set_j$ converge to a set $\set_\infty$ (of cardinality at most $I$) and, thanks to \crefnoname{thm:LamCptness}, the surfaces $\ms_j$ converge to a free boundary minimal lamination $\hat\lam$ in $\amb\setminus \set_\infty$ smoothly away from $\set_\infty$. 
However, \cref{thm:RemSingLimLam} ensures that the lamination $\hat\lam$ extends smoothly through $\set_\infty$; namely, there exists a smooth free boundary minimal lamination $\lam$ in $\amb$ extending $\hat\lam$. The last claim is straightforward.
\end{proof}

%%%%%%%%%%%%%%%%%%%%%%%%%%%%%%%%%%%%%%%%%%%%%%%%%%%%%%%%%%%%%%%%
%%% Microscopic behavior
%%%%%%%%%%%%%%%%%%%%%%%%%%%%%%%%%%%%%%%%%%%%%%%%%%%%%%%%%%%%%%%%

\section{Microscopic behavior} \label{sec:LocalDegeneration}

In this section we study the behavior of our minimal surfaces at small scales around the points where concentration of curvature occurs, that is to say around the points in $\set_\infty$ in \cref{cor:ExistenceBlowUpSetAndCurvatureEstimate}.

\subsection{Setting description} \label{HypN}

We denote by $(\mathfrak{N})$ the following set of assumptions:
\begin{enumerate} [label={\normalfont(\roman*)}]
\item $g_j$ is a sequence of metrics on $\amb_j^3 \eqdef \Xi(a_j)\cap \{\abs x< R_j\}\subset \R^3$, with $0\ge a_j\ge -\infty$ and $R_j\to\infty$, locally smoothly converging to the Euclidean metric as 
$j\to\infty$.
\item $\ms_j^2\subset \amb_j$ is a sequence of properly embedded \emph{edged} minimal surfaces (with $\partial\ms_j\subset\partial\amb_j$) that have free boundary with respect to $\Pi(a_j)$. 
\item \label{N:BlowUpSet} $\ind(\ms_j)\le I$ for some natural constant $I>0$ independent of $j$, and
$\set_j\subset \ms_j\cap B_{{\mu_0}}(0)$, {for $\mu_0$ given by \cref{cor:14TopInfo}}, is a sequence of non-empty smooth blow-up sets\footnote{Hereafter we denote by $B_r(p)$ the ball of center $p$ and radius $r$ in the metric $g_j$, without specifying $j$ when this is clear from the context. Also, note that the setting in \cref{def:BlowUpSets} is slightly different from the setting here, but the definition can be easily adapted to this context.} 
with $\abs {\set_j}\le I$ and \[\abs{\A_{\ms_j}}(x) d_{g_j}(x, \set_j\cup 
(\partial\ms_j\setminus \Pi(a_j) )) \le C\] for all $x\in \ms_j$, for some constant 
$C>0$ independent of $j$.
\end{enumerate}

\begin{figure}[htpb]
\centering
%%%%%%%%%%%%%%%%%%%%%%%%%%%%%%%%%%%%%%%%%%%%%%%%%%%%%
%%% micro_setting.tikz
%%%%%%%%%%%%%%%%%%%%%%%%%%%%%%%%%%%%%%%%%%%%%%%%%%%%%
\begin{tikzpicture}[scale=1.1]
\pgfmathsetmacro{\R}{2}
\pgfmathsetmacro{\r}{1}

% No boundary

\coordinate (OA) at (0,0);

\draw[gray] (OA) circle (\R);
\node[gray] at (-0.8*\R,0.8*\R) {$\amb_j$};
\draw[gray] (OA) circle (\r);
\node[gray] at (1.1*\r,-1*\r) {\footnotesize $B_{\mu_0}(0)$};

\node at (-0.3*\R, 0.65*\R) {$\ms_j$};

\def\h{{-0.15, 0.37, 0.5, 0.55}}

\foreach \i in {0,...,3}
  \draw ({-sqrt(\R*\R*(1-\h[\i]*\h[\i]))},{\R*\h[\i]}) -- ({sqrt(\R*\R*(1-\h[\i]*\h[\i]))},{\R*\h[\i]});

% Lower cat
\pgfmathsetmacro{\ha}{-0.3*\R}
\pgfmathsetmacro{\xa}{sqrt(\R*\R-\ha*\ha)}
\pgfmathsetmacro{\hb}{-0.20*\R}
\pgfmathsetmacro{\xb}{sqrt(\R*\R-\hb*\hb)}

\draw[gray] (0.15,{(\ha+\hb)/2}) arc (180:0:0.05 and 0.02);
\draw (0.15,{(\ha+\hb)/2}) arc (180:360:0.05 and 0.02);
\draw plot [smooth, tension=0.35] coordinates {(-\xa,\ha) (-0.05,\ha) (0.15,{(\ha+\hb)/2}) (-0.05,\hb) (-\xb,\hb)};
\draw plot [smooth, tension=0.4] coordinates {(\xb,\hb) (0.45,\hb) (0.25,{(\hb+\ha)/2}) (0.45,\ha) (\xa,\ha)};

% Middle cat
\pgfmathsetmacro{\hd}{-0.08*\R}
\pgfmathsetmacro{\xd}{sqrt(\R*\R-\hd*\hd)}
\pgfmathsetmacro{\he}{0.10*\R}
\pgfmathsetmacro{\xe}{sqrt(\R*\R-\he*\he)}

\draw[gray] (-0.15,{(\he+\hd)/2}) arc (180:0:0.1 and 0.04);
\draw (-0.15,{(\he+\hd)/2}) arc (180:360:0.1 and 0.04);
\draw plot [smooth, tension=0.4] coordinates {(-\xe,\he) (-0.4,\he) (-0.15,{(\he+\hd)/2}) (-0.4,\hd) (-\xd,\hd)};
\draw plot [smooth, tension=0.4] coordinates {(\xe,\he) (0.3,\he) (0.05, {(\he+\hd)/2}) (0.3,\hd) (\xd,\hd)};

% Higher cat 
\pgfmathsetmacro{\hf}{0.15*\R}
\pgfmathsetmacro{\xf}{sqrt(\R*\R-\hf*\hf)}
\pgfmathsetmacro{\hg}{0.20*\R}
\pgfmathsetmacro{\xg}{sqrt(\R*\R-\hg*\hg)}

\draw[gray] (0.25,{(\hf+\hg)/2}) arc (180:0:0.05 and 0.02);
\draw (0.25,{(\hf+\hg)/2}) arc (180:360:0.05 and 0.02);
\draw plot [smooth, tension=0.25] coordinates {(-\xg,\hg) (0.1,\hg) (0.25,{(\hg+\hf)/2}) (0.1,\hf) (-\xf,\hf)};
\draw plot [smooth, tension=0.32] coordinates {(\xg,\hg) (0.5,\hg) (0.35,{(\hg+\hf)/2}) (0.5,\hf) (\xf,\hf)};

\begin{scope}[xshift=70*\R]
% Boundary

\pgfmathsetmacro{\hl}{-0.20*\R}
\pgfmathsetmacro{\xl}{sqrt(\R*\R-\hl*\hl)}
\coordinate (OA) at (0,0);

\pgfmathsetmacro{\t}{atan(-\hl/(\xl))}
\draw[gray] (\xl,\hl) arc(-\t:180+\t:\R);
\draw[gray, dashed] (-\xl,\hl) arc(180+\t:360-\t:\R);
\node[gray] at (0.8*\R,0.8*\R) {$\amb_j$};

\pgfmathsetmacro{\s}{atan(-\hl/(sqrt(\r*\r-\hl*\hl)))}
\draw[gray] ({(sqrt(\r*\r-\hl*\hl))},\hl) arc(-\s:180+\s:\r);
\draw[gray, dashed] (-{(sqrt(\r*\r-\hl*\hl))},\hl) arc(180+\s:360-\s:\r);
\node[gray] at (-1.1*\r,-1*\r) {\footnotesize $B_{\mu_0}(0)$};

\draw[thick] (-\xl,\hl) -- (\xl,\hl);
\node[gray] at (0.78*\R,\hl+0.12*\R) {$\Pi(a_j)$};

\node at (-0.43*\R, 0.65*\R) {$\ms_j$};

\def\c{{-0.3,-0.25, 0.3, -0.6,0.5}}
 
\foreach \i in {0,...,4}
 \draw ({\R*\c[\i]}, {sqrt(\R*\R*(1-\c[\i]*\c[\i]))}) -- ({\R*\c[\i]}, \hl);
 
% Left cat
\pgfmathsetmacro{\ca}{-0.2*\R}
\pgfmathsetmacro{\ya}{sqrt(\R*\R-\ca*\ca)}
\pgfmathsetmacro{\cb}{-0.10*\R}
\pgfmathsetmacro{\yb}{sqrt(\R*\R-\cb*\cb)}

\draw[gray] ({(\ca+\cb)/2}, {\hl+0.05*\R}) arc (90:180:0.03 and {0.05*\R});
\draw ({(\ca+\cb)/2}, {\hl+0.05*\R}) arc (90:0:0.03 and {0.05*\R});
\draw plot [smooth, tension=0.4] coordinates {(\ca,\ya) (\ca, {\hl+0.17*\R}) ({(\ca+\cb)/2}, {\hl+0.05*\R}) (\cb,{\hl+0.17*\R}) (\cb,\yb)};

% Right cat
\pgfmathsetmacro{\cc}{0*\R}
\pgfmathsetmacro{\yc}{sqrt(\R*\R-\cc*\cc)}
\pgfmathsetmacro{\cd}{0.16*\R}
\pgfmathsetmacro{\yd}{sqrt(\R*\R-\cd*\cd)}

\draw[gray] ({(\cc+\cd)/2}, {\hl+0.25*\R}) arc (90:270:0.05 and {0.05*\R});
\draw ({(\cc+\cd)/2}, {\hl+0.25*\R}) arc (90:-90:0.05 and {0.05*\R});
\draw plot [smooth, tension=0.5] coordinates {(\cc,\yc) (\cc, {\hl+0.4*\R}) ({(\cc+\cd)/2}, {\hl+0.25*\R}) (\cd,{\hl+0.4*\R}) (\cd,\yd)};
\draw plot [smooth, tension=1.8] coordinates {(\cc,\hl) ({(\cc+\cd)/2}, {\hl+0.15*\R}) (\cd,\hl)};
\end{scope}

\begin{scope}[xshift=140*\R]
% Boundary ``bad''

\pgfmathsetmacro{\hl}{0.01*\R}
\pgfmathsetmacro{\xl}{sqrt(\R*\R-\hl*\hl)}

\coordinate (OA) at (0,0);

\pgfmathsetmacro{\t}{atan(-\hl/(\xl))}
\draw[gray] (\xl,\hl) arc(-\t:180+\t:\R);
\draw[gray, dashed] (-\xl,\hl) arc(180+\t:360-\t:\R);
\node[gray] at (0.8*\R,0.8*\R) {$\amb_j$};

\pgfmathsetmacro{\s}{atan(-\hl/(sqrt(\r*\r-\hl*\hl)))}
\draw[gray] ({(sqrt(\r*\r-\hl*\hl))},\hl) arc(-\s:180+\s:\r);
\draw[gray, dashed] (-{(sqrt(\r*\r-\hl*\hl))},\hl) arc(180+\s:360-\s:\r);
\node[gray] at (-1.1*\r,-1*\r) {\footnotesize $B_{\mu_0}(0)$};

\draw[thick] (-\xl,\hl) -- (\xl,\hl);
\node[gray] at (0.76*\R,\hl-0.15*\R) {$\Pi(a_j)$};

\node at (-0.3*\R, 0.65*\R) {$\ms_j$};

\def\h{{0.37, 0.5, 0.55}}

\foreach \i in {0,...,2}
  \draw ({-sqrt(\R*\R*(1-\h[\i]*\h[\i]))},{\R*\h[\i]}) -- ({sqrt(\R*\R*(1-\h[\i]*\h[\i]))},{\R*\h[\i]});

% Middle cat
\pgfmathsetmacro{\hd}{-0.08*\R}
\pgfmathsetmacro{\xd}{sqrt(\R*\R-\hd*\hd)}
\pgfmathsetmacro{\he}{0.10*\R}
\pgfmathsetmacro{\xe}{sqrt(\R*\R-\he*\he)}

\draw[gray] (-0.15,{(\he+\hd)/2}) arc (180:0:0.1 and 0.04);
\draw plot [smooth, tension=0.4] coordinates {(-\xe,\he) (-0.4,\he) (-0.15,{(\he+\hd)/2})};
\draw plot [smooth, tension=0.4] coordinates {(\xe,\he) (0.3,\he) (0.05, {(\he+\hd)/2})};

% Higher cat 
\pgfmathsetmacro{\hf}{0.15*\R}
\pgfmathsetmacro{\xf}{sqrt(\R*\R-\hf*\hf)}
\pgfmathsetmacro{\hg}{0.20*\R}
\pgfmathsetmacro{\xg}{sqrt(\R*\R-\hg*\hg)}

\draw[gray] (0.25,{(\hf+\hg)/2}) arc (180:0:0.05 and 0.02);
\draw (0.25,{(\hf+\hg)/2}) arc (180:360:0.05 and 0.02);
\draw plot [smooth, tension=0.25] coordinates {(-\xg,\hg) (0.1,\hg) (0.25,{(\hg+\hf)/2}) (0.1,\hf) (-\xf,\hf)};
\draw plot [smooth, tension=0.32] coordinates {(\xg,\hg) (0.5,\hg) (0.35,{(\hg+\hf)/2}) (0.5,\hf) (\xf,\hf)};
\end{scope}
\end{tikzpicture}
%%%%%%%%%%%%%%%%%%%%%%%%%%%%%%%%%%%%%%%%%%%%%%%%%%%%%
%%%%%%%%%%%%%%%%%%%%%%%%%%%%%%%%%%%%%%%%%%%%%%%%%%%%%
\caption{Different possible situations in setting \hypN{}.} \label{fig:MicroSett}
\end{figure}
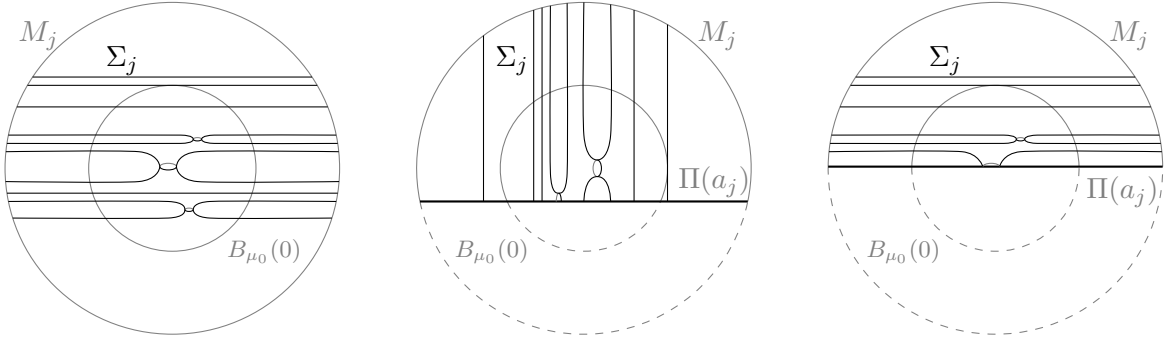

\begin{proposition} \label{prop:StructureLimLam}
Let us consider the set of assumptions \hypN{}. Then, up to subsequence, the smooth blow-up sets $\set_j$ converge to a finite set of points $\set_\infty$, of cardinality at most $I$, and there 
is a free boundary minimal lamination $\lam$ in $\Pi(a)$ (where we assume the existence of $a = \lim_{j\to\infty}a_j\in \cc{-\infty} 0$) consisting of parallel planes or half-planes and such that $\ms_j$ locally converges (in the sense of laminations) to $\lam$ away from $\set_\infty$.
\end{proposition}
\begin{proof}
First observe that, up to subsequence, we can assume that $\set_j$ converge to a finite set of points $\set_\infty$ of cardinality at most $I$ and contained in the open unit ball centered at the origin.
Then, because of the curvature assumptions we are making, thanks to \cref{prop:LimLamInR3} we gain smooth (subsequential) convergence to a free boundary minimal lamination $\hat\lam$ in $\Pi(a)\setminus\set_\infty$;
in fact $\hat\lam$ extends to a smooth lamination $\lam$ of $\Pi(a)$. 
We now need to rule out alternative \ref{llir:multone} of the proposition, namely the possibility that $\lam$ consists of a (single) non-flat, two-sided, properly embedded, free boundary minimal surface.
However, in this case the convergence must be locally smooth and graphical with multiplicity one at (all points of) such unique leaf: hence this would imply locally uniform curvature estimates for the sequence $\Sigma_j$, which is in contradiction with the presence of a smooth blow-up set in \hypN{}, though.
Thus the only possibility is that $\lam$ is a lamination in $\Xi(a)$ of parallel planes or half-planes, which concludes the proof.
\end{proof}

Later on, we will separately study the components where `bad things' happen (but which are in finite number) and the others. Therefore it will be useful to introduce the following definition. 
\begin{definition}
In the setting \hypN{}, let us denote by $\ms_j^{\euno}$ the union of the connected components of 
$\ms_j\cap B_1(0)$ that contain at least one point in $\set_j$ (informally: the ones with the necks in \crefnoname{fig:MicroSett}) and with
$\ms_j^{\edue}$ the union of the connected components of $\ms_j\cap B_1(0)$ that do not.
\end{definition}
\begin{remark}
It is sufficient to work in $B_1(0)$, since the information about $\ms_j$ in $B_1(0)^c$ are then obtained in the applications thanks to suitable Morse-theoretic arguments (see \crefnoname{sec:MorseTheory}).
\end{remark}

\begin{remark} \label{rem:ConvSigma1}
Observe that the surfaces $\ms_j^{\euno}$ locally converge (in the sense of laminations) to the union of the leaves of the lamination $\lam$ (given by \cref{prop:StructureLimLam}) passing through points of $\set_\infty$, the convergence happening away from $\set_\infty$. Indeed, the convergence to any other component of $\lam$ in $B_1(0)$ is uniformly smooth (in the sense of laminations), but each component of $\ms_j^{\euno}$ contains a point where the curvature diverges.

In particular, if the number of components of $\ms_j^{\euno}\cap(\partial B_1(0)\setminus\Pi(a_j))$ is uniformly bounded, then $\ms_j^{\euno}\cap(\partial B_1(0)\setminus\Pi(a_j))$ is $\mu_0$-strongly equatorial (as per \cref{def:StrongEq}) for any $j$ sufficiently large.
\end{remark}

\subsection{Neck components}

In this section, we deal with the behavior of $\ms_j^{\euno}$, which is the part of $\ms_j\cap B_1(0)$ that `carries the topology' of the surface $\ms_j$. 
The components $\ms_j^{\edue}$ are instead well-controlled, in the sense that they have uniformly bounded curvature and they are topological discs. We postpone the investigation of these properties of $\ms_j^{\edue}$ to the proof \cref{thm:GlobalDeg}.

The following proposition is essentially the base case of the induction to prove \cref{prop:LocDeg}, which is the full description of what happens around the origin along the sequence $\ms_j^{\euno}$. 

\begin{lemma} \label{lem:LocDegOnePt}
Let us assume to be in the setting \hypN{} with $\abs {\set_j} = 1$ for all $j$.
Then there exists $\kappa(I)\ge0$ (depending on $I$) such that, for $j$ sufficiently large, the following assertions hold true:
\begin{enumerate} [label={\normalfont(\arabic*)}]
\item The surfaces $\ms_j^{\euno}$ have genus at most $\kappa(I)$.
\item {The surfaces $\ms_j^{\euno}$ intersect both $\partial B_{1}(0)\setminus \Pi(a_j)$ and $\Pi(a_j)$ transversely in at most $\kappa(I)$ components.}
\end{enumerate}

\end{lemma}
\begin{proof}
Let $\set_j = \{p_j\}$, $\set_\infty = \{p_\infty\}$ and $\lambda_j \eqdef \abs{\A_{\ms_j}}(p_j)$. By definition of smooth blow-up set, the surfaces $\bu\ms_j\eqdef \lambda_j(\ms_j-p_j)$ converge up to subsequence to a complete, non-flat, properly embedded, free boundary minimal surface $\bu\ms_\infty$ in $\Xi(a)$ for some $0\ge a\ge -\infty$. Furthermore $\bu\ms_\infty$ has index at most $I$.
Therefore, by \cref{prop:GeometryOfHalfBubble}, the genus, number of ends and number of boundary components of $\bu\ms_\infty$ are all bounded by $\kappa(I)$.

Consider $\mu_0$ given by \cref{cor:14TopInfo} and take $R_0>0$ such that 
\begin{equation} \label{eq:CurvEstSigmaBarInfty}
\abs{\A_{\bu\ms_\infty}} (x) d_{\R^3}(0,x) < {\mu_0}
\end{equation}
for $x\in\bu\ms_\infty \setminus B_{R_0}(0)$.
{Assume $R_0$ large enough that the genus and the number of connected components of both $\bu\ms_\infty\cap { (\partial B_{R_0}(0)\setminus \Pi(a))}$ and $\bu\ms_\infty\cap \Pi(a)$ are bounded by $\kappa(I)$.
{Moreover, thanks to \cite[Proposition 1]{Sch83_uniqueness}, we can also suppose that $\bu\ms_\infty\cap (\partial B_{R_0}(0)\setminus\Pi(a))$ is $(\mu_0/2)$-strongly equatorial (as per \cref{def:StrongEq}).}
Hence observe that, for $j$ sufficiently large, $\ms_j\cap B_{R_0/\lambda_j}(p_j)$ has genus and number of boundary components on $\partial B_{R_0/\lambda_j}(p_j)\setminus \Pi(a_j)$ both bounded by $\kappa(I)$, and $\ms_j\cap(\partial B_{R_0/\lambda_j}(p_j)\setminus\Pi(a_j))$ is $\mu_0$-strongly equatorial.}
In order to transfer this information to all $B_1(0)$ we want to prove that the estimate
\begin{equation}\label{eq:14EstMsJ}
\abs{\A_{\ms_j}}(x) d_{g_j} (p_j,x) < {\mu_0}
\end{equation}
holds for every $x\in \ms_j \cap ( B_{1}(0)\setminus B_{R_0/\lambda_j}(p_j) )$, for $j$ sufficiently large.

To this purpose, it is enough to prove that there exists $\delta>0$ such that the estimate holds in $\ms_j \cap ( B_\delta(p_j)\setminus B_{R_0/\lambda_j}(p_j) )$. Then we deduce the desired estimate using that $\ms_j$ converges (in the sense of laminations) in $B_1(0)\setminus B_\delta (p_j)$ to a lamination consisting of planes (indeed $p_\infty\in B_\delta(p_j)$ for $j$ sufficiently big).

Assume by contradiction that such $\delta>0$ does not exist. Then one would find a sequence $z_j\in \ms_j\setminus B_{R_0/\lambda_j}(p_j)$ such that $\delta_j\eqdef d_{g_j}(p_j,z_j) \to 0$ and $\abs{\A_{\ms_j}}(z_j)\delta_j \ge {\mu_0}$.
Then consider $\check\ms_j \eqdef \delta_j^{-1}(\ms_j-p_j)$, for which we have \begin{equation}\label{eq:BoundBelowCurvature}
\abs{\A_{\check\ms_j}}(\delta_j^{-1}(z_j-p_j)) \ge {\mu_0}\point
\end{equation}
Note that there cannot possibly be a uniform curvature bound for the sequence $\check\ms_j$ around $0$, otherwise the scales $\lambda_j$ and $\delta^{-1}_j$ would be comparable, hence the surfaces $\check\ms_j$ would converge to a homothety of $\bu\ms_\infty$, but this is not possible for the choice of $z_j$ together with \eqref{eq:CurvEstSigmaBarInfty}.
As a result, possibly extracting a subsequence (which we do not rename) $\check\ms_j$ still fulfills the assumptions of the setting \hypN{} and therefore it converges to a lamination consisting of planes away from $0$. However, observe that this implies $\abs{\A_{\check\ms_j}}(\delta_j^{-1}(z_j-p_j))\to 0$, which contradicts \eqref{eq:BoundBelowCurvature} and thus proves \eqref{eq:14EstMsJ}.

Thus, all the assumptions of \cref{cor:14TopInfo} are satisfied and therefore the genus and the number of connected components of both $\ms_j^{\euno}\cap(\partial B_1(0)\setminus \Pi(a_j))$ and $\ms_j^{\euno}\cap\Pi(a_j)$ are bounded by $\kappa(I)$ (possibly renaming $\kappa(I)$ as the double of the constant introduced above).
\end{proof}

We can now proceed and prove the corresponding result for any set $\set_j$ of uniformly bounded cardinality.

\begin{proposition} \label{prop:LocDeg}
Assume to be in the setting \hypN{}. 
Then there exists $\kappa(I)\ge0$ such that, for $j$ sufficiently large, the following assertions hold true:
\begin{enumerate} [label={\normalfont(\arabic*)}]
\item The surfaces $\ms_j^{\euno}$ have genus at most $\kappa(I)$.
\item {The surfaces $\ms_j^{\euno}$ intersect both $\partial B_{1}(0)\setminus \Pi(a_j)$ and $\Pi(a_j)$ transversely in at most $\kappa(I)$ components.}
\end{enumerate}
\end{proposition}
\begin{proof}
Let us proceed by induction on $I> 0$. 
The case $I=1$ has been treated in \cref{lem:LocDegOnePt}, thus let us assume $I>1$.
We distinguish two cases, and we first consider the case when $\abs{\set_\infty} \ge2$. 
Choose $\delta>0$ small enough that $\min_{p,q\in \set_\infty}d_g(p,q)\ge 4\delta$, moreover fix one point $p_\infty\in \set_\infty$. Since $\ind(\ms_j\cap B)\ge 1$ for every connected component $B$ of $B_\delta(\set_\infty)$, then $\ind(\ms_j\cap B_\delta(p_\infty)) \le I-1$.

Now choose positive numbers $r_j\to 0$ such that $\set_j\subset B_{\mu_0r_j}(\set_\infty)$ and 
\[\liminf_{j\to\infty} \left(r_j \min_{p\in \set_j} \abs{\A_{\ms_j}} (p)\right) = \infty\point\]
Note that the surfaces $r_j^{-1}(\ms_j-p_\infty)$ still fulfill the assumptions \hypN{} (with blow-up sets that are rescalings of the blow-up sets $\set_j$) thanks to the choice of $r_j$ and thus we can apply the inductive hypothesis to these surfaces. In particular we obtain that the components $\ms_j^{\euno}\cap B_{r_j}(p_\infty)$ have genus at most $\kappa(I-1)$ and {intersect both $\partial B_{r_j}(p_\infty){\setminus \Pi(a_j)}$ and $\Pi(a_j)$ transversely in at most $\kappa(I-1)$ components.}

We now prove that, choosing $\delta>0$ possibly smaller, we have that
\[
\abs{A_{\ms_j}}(x) d_{g_j}(x,p_\infty) < {\mu_0}
\]
for all $x \in B_\delta(p_\infty)\setminus B_{r_j}(p_\infty)$.
If this is not the case, then there exists a sequence of points $z_j\in \ms_j$ with $\delta_j\eqdef d_{g_j}(z_j,p_\infty) >r_j$, $\delta_j\to0$ and $\abs{A_{\ms_j}}(z_j)\delta_j\ge {\mu_0}$. The rescaled surfaces $\bu \ms_j \eqdef \delta_j^{-1}(\ms_j-p_\infty)$ still satisfy \hypN{} with blow-up set $\delta_j^{-1}(\set_j-p_\infty)$ and therefore they converge away from $0$ to a lamination consisting of parallel planes by \cref{prop:StructureLimLam}. In particular $\abs{A_{\bu\ms_j}}(\delta_j^{-1}(z_j-p_\infty)) \to 0$, which contradicts the choice of $z_j$.

{Note that, choosing $j$ sufficiently large, we can assume that $\ms_j^{\euno}\cap(\partial B_{r_j}(p_\infty)\setminus \Pi(a_j))$ is $\mu_0$-strongly equatorial by \cref{rem:ConvSigma1}.
As a result, we can invoke \cref{cor:14TopInfo} and conclude that the components $\ms_j^{\euno}\cap B_{\delta}(p_\infty)$ also have genus at most $\kappa(I-1)$ and intersect both $\partial B_{\delta}(p_\infty)\setminus \Pi(a_j)$ and $\Pi(a_j)$ transversely in at most $2\kappa(I-1)$ components.}

Now we follow the very same argument on each ball of radius $\delta$ (as above) and centered at a point of $\set_{\infty}$. We obtain analogous bounds, hence (keeping in mind that we have uniform curvature estimates away from such balls) we can exploit the topological bounds we have gained
to get that $\ms_j^{\euno}$ converges graphically smoothly with finite multiplicity in $B_1(0)\setminus B_\delta(\set_\infty)$ to the leaves passing through $\set_\infty$ of the limit lamination $\lam$ in \cref{prop:StructureLimLam} and conclude with the basic topological tools presented in \crefnoname{sec:EulChar}.

Therefore, now we have to deal only with the case $\abs{\set_\infty} = 1$.
We can assume $\abs{\set_j}\ge 2$, since otherwise we could just apply \cref{lem:LocDegOnePt}. Take $p_j,q_j\in\set_j$ that realize the maximum distance between points in $\set_j$ and define $r_j\eqdef  2 d_{g_j}(p_j,q_j) \mu_0^{-1}\to 0$.
The sequence of surfaces $\check\ms_j\eqdef r_j^{-1}(\ms_j-p_j)$ still satisfy the assumptions \hypN{} with index $I$. Moreover $\abs{\check \set_\infty}\ge 2$, therefore we can apply the first part of the proposition to $\check\ms_j$, obtaining all the desired information in $B_{r_j}(p_j)$. However, now we can argue as above to obtain the ${\mu_0}$-curvature estimate in $\ms_j\cap (B_{1}(0)\setminus B_{r_j}(p_j))$ and transfer the information to $B_1(0)$.
\end{proof}

\subsection{Fine description} \label{sec:GlobalScheme}

We are now ready to put together all the information obtained above and present a fine description of degeneration for a sequence of surfaces with bounded index.

\begin{theorem} \label{thm:GlobalDeg}
Given $I\in \N$ there exists $\kappa(I)\ge 0$ such that the following assertions hold true.

Let $(\amb^3,g)$ be a three-dimensional Riemannian manifold with boundary that satisfies \hypP{}. Let $\ms_j^2\subset \amb$ be a sequence of compact, properly embedded, free boundary minimal surfaces with $\ind(\ms_j)\le I$, for a fixed constant $I\in\N$.
 In the setting of \cref{cor:ExistenceBlowUpSetAndCurvatureEstimate}, one can find a constant $\eps_0>0$ (depending on the sequence in question) so that $\min_{p,q\in\set_\infty}d_g(p,q)\ge 4\eps_0$ and such that, taken $\eps\leq \eps_0$ and defined $\ms_j^{\euno}$ the union of the components of $\ms_j\cap B_{\eps}(\set_\infty)$ which contain at least one point of $\set_j$ and $\ms_j^{\edue}$ the union of the other components of $\ms_j\cap B_{\eps}(\set_\infty)$, for $j$ sufficiently large: 
\begin{enumerate} 
\item 
\begin{enumerate} [ref=(\arabic{enumi}\alph*)]
\item No component of $\ms_j^{\euno}$ is a disc or a half-disc.
\item The genus of $\ms_j^{\euno}$ is bounded by $\kappa(I)$.
\item \label{gd:NumCompSigmaI}
{$\ms_j^{\euno}$ intersects both $\partial B_\eps(\set_\infty)\setminus\partial M$ and $\partial\amb\cap B_\eps(\set_\infty)$ transversely in at most $\kappa(I)$ components.}
\item \label{gd:AreaBoundSigmaI}
$\ms_j^{\euno}$ has uniformly bounded area, namely
\[ 
\limsup_{j\to\infty}\ \area(\ms_j^{\euno}) \le 4\pi \kappa(I) \eps^2\point
\]
\end{enumerate}
\item 
\begin{enumerate} [ref=(\arabic{enumi}\alph*)]
\item Each component of $\ms_j^{\edue}$ is a disc or a half-disc.
\item $\ms_j^{\edue}$ has uniformly bounded curvature, that is 
\[
\limsup_{j\to\infty}\,\sup_{x\in\ms_j^{\edue}}\ \abs{\A_{\ms_j}}(x) < \infty \point
\]
\item Each component of $\ms_j^{\edue}$ has area uniformly bounded, namely
\[ 
\limsup_{j\to\infty}\, \sup_{\substack{C\subset \ms_j^{\edue}\\ \text{connected}}} \area(C) \le 2\pi \eps^2\point
\]
 \label{gd:AreaBoundSigmaII}
\end{enumerate}
\end{enumerate}
\end{theorem}

\begin{proof}

Let $\eps_0>0$ be such that $\min_{p,q\in\set_\infty}d_g(p,q)\ge 4\eps_0$ and each component of $B_{\eps_0}(\set_\infty)$ admits a chart to the unit ball in $\R^3$ (around the points of $\set_\infty$ in $\amb\setminus\partial\amb$) or in $\Xi(0)$ (around the points of $\set_\infty$ in $\partial\amb$) where the limit lamination (cf. statement of \cref{cor:ExistenceBlowUpSetAndCurvatureEstimate}) $\lam$ is sufficiently close (in the sense of laminations) to a union of parallel planes or half-planes. 
Moreover, choose $r_j\to 0$ such that $\set_j\subset B_{ \mu_0r_j}(\set_\infty)$ and
\[
\liminf_{j\to\infty}\ \left( r_j \min_{p\in\set_j}\ \abs{\A_{\ms_j}}(p) \right) = \infty \point
\]
Observe that, fixing a point $p_\infty\in \set_\infty$, the rescaled surfaces $r_j^{-1}(\ms_j - p_\infty)$ in the ambient manifolds $r_j^{-1}(B_{\eps_0}(p_\infty)-p_\infty)$ (seeing everything in chart) fit in the setting \hypN{}.
Therefore, by \cref{prop:LocDeg}, we obtain that the surfaces $\ms_j^{\euno} \cap B_{r_j}(p_\infty)$ have genus at most $\kappa(I)$ and intersect $\partial\amb$ and $ \partial B_{r_j}(p_\infty)\setminus\partial M$ transversely in at most $\kappa(I)$ components.

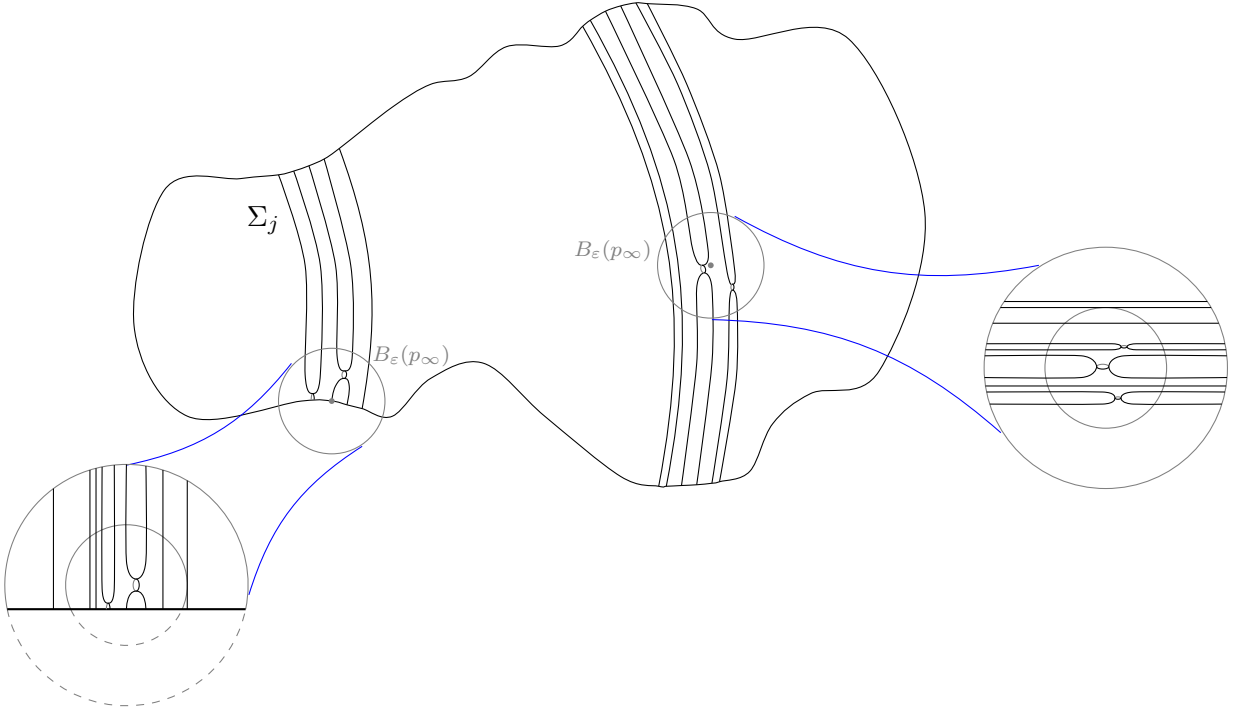
\begin{figure}[htpb]
\centering
%%%%%%%%%%%%%%%%%%%%%%%%%%%%%%%%%%%%%%%%%%%%%%%%%%%%%
%%% global_descr.tikz
%%%%%%%%%%%%%%%%%%%%%%%%%%%%%%%%%%%%%%%%%%%%%%%%%%%%%

\begin{tikzpicture}
\coordinate (A1) at (0,0); 
\coordinate (B1) at (0.2,0.05); 
\coordinate (C1) at (0.4,0.12);
\coordinate (D1) at (0.6,0.21);
\coordinate (E1) at (0.8,0.35);

\coordinate (F1) at (4,1.97);
\coordinate (G1) at (4.1,2.05);
\coordinate (H1) at (4.3,2.17);
\coordinate (I1) at (4.5,2.24);
\coordinate (J1) at (4.7,2.28);
\coordinate (K1) at (4.85,2.27);

\coordinate (F2) at (5,-4.13);
\coordinate (G2) at (5.1,-4.13);
\coordinate (H2) at (5.3,-4.12);
\coordinate (I2) at (5.5,-4.11);
\coordinate (J2) at (5.7,-4.09);
\coordinate (K2) at (5.8,-4.07);

\coordinate (C2) at (0.7,-3);
\coordinate (D2) at (0.9,-3.05);
\coordinate (E2) at (1.1,-3.1);

\coordinate (X) at (-1.5,-0.2);
\coordinate (Y) at (-1.2, -3.2);
\coordinate (U) at (7.5, 1.8);
\coordinate (V) at (7.9, -2.6);

\draw plot [smooth cycle, tension=0.6] coordinates {(A1) (B1) (C1) (D1) (E1)
			      (1.5,0.9) (2,1.2) (2.5,1.3) (3,1.7) (3.7,1.71)
			      (F1) (G1) (H1) (I1) (J1) (K1)
			      (5.5,2.2) (6,1.8)
			      (U) (8.5,-0.5) (V)
			      (7,-2.9) (6.5, -3.3) (6.2, -3.93) 
			      (K2) (J2) (I2) (H2) (G2) (F2)
			      (4.5,-4) (3.5,-3.2) (2.7, -2.5) (2,-2.7) (1.5, -3.2) 
			      (E2) (D2) (C2)
			      (0.3, -3)
			      (Y) (-1.9,-2) (X) (-0.5,-0.05)};
			      
\draw (F1) to [bend left=20] (F2);
\draw (G1) to [bend left=20] (G2);

\draw plot [smooth, tension = 0.5] coordinates {(H1) (5.2,0.2) (5.45,-1) (5.57, -1.2) (5.65, -1) (5.4,0.2) (I1)};
\draw plot [smooth, tension = 0.5] coordinates {(H2) (5.5,-2.5) (5.5,-1.5) (5.6, -1.3) (5.7, -1.5) (5.7,-2.5) (I2)};
\pgfmathsetmacro{\t}{atan(0.03/0.1)}
\begin{scope}[rotate around={\t:(5.57,-1.2)}]
\draw[gray] (5.57, -1.2) arc (90:270:0.03 and {sqrt(0.03*0.03+0.1*0.1)/2});
\draw (5.57, -1.2) arc (90:-90:0.03 and {sqrt(0.03*0.03+0.1*0.1)/2});
\end{scope}

\draw plot [smooth, tension = 0.5] coordinates {(J1) (5.25,1.2) (5.68,0) (5.9,-1.29) (5.97, -1.45) (6, -1.27) (5.8,0) (5.37,1.2) (K1)};
\draw plot [smooth, tension = 0.5] coordinates {(J2)  (5.9,-2.7) (5.93, -1.7) (5.97, -1.53) (6.03,-1.7) (6.02, -2.7) (K2)};
\draw[gray] (5.97, -1.45) arc (90:270:0.02 and {0.04});
\draw (5.97, -1.45) arc (90:-90:0.02 and {0.04});

\draw[gray] (0.45, -2.9) arc (90:182:0.02 and {0.08});
\draw (0.45, -2.9) arc (90:0:0.02 and {0.08});
\draw plot [smooth, tension = 0.55] coordinates {(A1) (0.35,-1.2) (0.35,-2.65) (0.45, -2.9) (0.55, -2.65) (0.55,-1.2) (B1)};

\pgfmathsetmacro{\t}{atan(0.01/0.1)}
\begin{scope}[rotate around={-\t:(0.87, -2.6)}]
\draw[gray] (0.87,-2.6) arc (90:270:0.03 and {sqrt(0.01*0.01+0.1*0.1)/2});
\draw (0.87,-2.6) arc (90:-90:0.03 and {sqrt(0.01*0.01+0.1*0.1)/2});
\end{scope}
\draw plot [smooth, tension = 0.55] coordinates {(C1) (0.75,-1.2) (0.77,-2.4) (0.87, -2.6) (0.97, -2.4) (0.95,-1.2) (D1)};
\draw plot [smooth, tension = 0.9] coordinates {(C2) (0.75, -2.8) (0.86, -2.7) (0.93, -2.8) (D2)};
\draw (E1) to [bend left = 15] (E2);
\node at (-0.2,-0.6) {$\ms_j$};

\pgfmathsetmacro\R{0.7}
\pgfmathsetmacro\r{0.5}
\fill[gray] (C2) circle [radius=0.1em];
\draw[gray] (C2) circle (\R);

\node[gray] at (1.75, -2.4) {\scriptsize $B_\varepsilon(p_\infty)$};

\fill[gray] (5.685, -1.2) circle [radius=0.1em];
\draw[gray] (5.685, -1.2) circle (\R);
\node[gray] at (4.4, -1) {\scriptsize $B_\varepsilon(p_\infty)$};

\draw[blue] (0.17,-2.5) to [bend left=20] (-2,-3.85);
\draw[blue] (1.1,-3.6) to [bend right=20] (-0.39,-5.57);

\draw[blue] (6,-0.55) to [bend right=20] (10,-1.2);
\draw[blue] (5.7,-1.92) to [bend left=20] (9.5,-3.42);

\begin{scope}[scale=0.8,xshift=-2.5cm,yshift=-6.8cm]
% Boundary

\pgfmathsetmacro{\R}{2}
\pgfmathsetmacro{\r}{1}

\pgfmathsetmacro{\hl}{-0.20*\R}
\pgfmathsetmacro{\xl}{sqrt(\R*\R-\hl*\hl)}
\coordinate (OA) at (0,0);

\pgfmathsetmacro{\t}{atan(-\hl/(\xl))}
\draw[gray] (\xl,\hl) arc(-\t:180+\t:\R);
\draw[gray, dashed] (-\xl,\hl) arc(180+\t:360-\t:\R);

\pgfmathsetmacro{\s}{atan(-\hl/(sqrt(\r*\r-\hl*\hl)))}
\draw[gray] ({(sqrt(\r*\r-\hl*\hl))},\hl) arc(-\s:180+\s:\r);
\draw[gray, dashed] (-{(sqrt(\r*\r-\hl*\hl))},\hl) arc(180+\s:360-\s:\r);

\draw[thick] (-\xl,\hl) -- (\xl,\hl);

\def\c{{-0.3,-0.25, 0.3, -0.6,0.5}}

\foreach \i in {0,...,4}
  \draw ({\R*\c[\i]}, {sqrt(\R*\R*(1-\c[\i]*\c[\i]))}) -- ({\R*\c[\i]}, \hl);

% Left cat
\pgfmathsetmacro{\ca}{-0.2*\R}
\pgfmathsetmacro{\ya}{sqrt(\R*\R-\ca*\ca)}
\pgfmathsetmacro{\cb}{-0.10*\R}
\pgfmathsetmacro{\yb}{sqrt(\R*\R-\cb*\cb)}

\draw[gray] ({(\ca+\cb)/2}, {\hl+0.05*\R}) arc (90:180:0.03 and {0.05*\R});
\draw ({(\ca+\cb)/2}, {\hl+0.05*\R}) arc (90:0:0.03 and {0.05*\R});
\draw plot [smooth, tension=0.4] coordinates {(\ca,\ya) (\ca, {\hl+0.17*\R}) ({(\ca+\cb)/2}, {\hl+0.05*\R}) (\cb,{\hl+0.17*\R}) (\cb,\yb)};

% Right cat
\pgfmathsetmacro{\cc}{0*\R}
\pgfmathsetmacro{\yc}{sqrt(\R*\R-\cc*\cc)}
\pgfmathsetmacro{\cd}{0.16*\R}
\pgfmathsetmacro{\yd}{sqrt(\R*\R-\cd*\cd)}

\draw[gray] ({(\cc+\cd)/2}, {\hl+0.25*\R}) arc (90:270:0.05 and {0.05*\R});
\draw ({(\cc+\cd)/2}, {\hl+0.25*\R}) arc (90:-90:0.05 and {0.05*\R});
\draw plot [smooth, tension=0.5] coordinates {(\cc,\yc) (\cc, {\hl+0.4*\R}) ({(\cc+\cd)/2}, {\hl+0.25*\R}) (\cd,{\hl+0.4*\R}) (\cd,\yd)};
\draw plot [smooth, tension=1.8] coordinates {(\cc,\hl) ({(\cc+\cd)/2}, {\hl+0.15*\R}) (\cd,\hl)};
\end{scope}

\begin{scope}[scale=0.8,yshift=-3.2cm, xshift=13.6cm]
\pgfmathsetmacro{\R}{2}
\pgfmathsetmacro{\r}{1}

% No boundary
\coordinate (OA) at (0,0);
\draw[gray] (OA) circle (\R);
\draw[gray] (OA) circle (\r);

\def\h{{-0.15, 0.37, 0.5, 0.55}}

\foreach \i in {0,...,3}
  \draw ({-sqrt(\R*\R*(1-\h[\i]*\h[\i]))},{\R*\h[\i]}) -- ({sqrt(\R*\R*(1-\h[\i]*\h[\i]))},{\R*\h[\i]});  

% Lower cat
\pgfmathsetmacro{\ha}{-0.3*\R}
\pgfmathsetmacro{\xa}{sqrt(\R*\R-\ha*\ha)}
\pgfmathsetmacro{\hb}{-0.20*\R}
\pgfmathsetmacro{\xb}{sqrt(\R*\R-\hb*\hb)}

\draw[gray] (0.15,{(\ha+\hb)/2}) arc (180:0:0.05 and 0.02);
\draw (0.15,{(\ha+\hb)/2}) arc (180:360:0.05 and 0.02);
\draw plot [smooth, tension=0.35] coordinates {(-\xa,\ha) (-0.05,\ha) (0.15,{(\ha+\hb)/2}) (-0.05,\hb) (-\xb,\hb)};
\draw plot [smooth, tension=0.4] coordinates {(\xb,\hb) (0.45,\hb) (0.25,{(\hb+\ha)/2}) (0.45,\ha) (\xa,\ha)};

% Middle cat
\pgfmathsetmacro{\hd}{-0.08*\R}
\pgfmathsetmacro{\xd}{sqrt(\R*\R-\hd*\hd)}
\pgfmathsetmacro{\he}{0.10*\R}
\pgfmathsetmacro{\xe}{sqrt(\R*\R-\he*\he)}

\draw[gray] (-0.15,{(\he+\hd)/2}) arc (180:0:0.1 and 0.04);
\draw (-0.15,{(\he+\hd)/2}) arc (180:360:0.1 and 0.04);
\draw plot [smooth, tension=0.4] coordinates {(-\xe,\he) (-0.4,\he) (-0.15,{(\he+\hd)/2}) (-0.4,\hd) (-\xd,\hd)};
\draw plot [smooth, tension=0.4] coordinates {(\xe,\he) (0.3,\he) (0.05, {(\he+\hd)/2}) (0.3,\hd) (\xd,\hd)};

% Higher cat 
\pgfmathsetmacro{\hf}{0.15*\R}
\pgfmathsetmacro{\xf}{sqrt(\R*\R-\hf*\hf)}
\pgfmathsetmacro{\hg}{0.20*\R}
\pgfmathsetmacro{\xg}{sqrt(\R*\R-\hg*\hg)}

\draw[gray] (0.25,{(\hf+\hg)/2}) arc (180:0:0.05 and 0.02);
\draw (0.25,{(\hf+\hg)/2}) arc (180:360:0.05 and 0.02);
\draw plot [smooth, tension=0.25] coordinates {(-\xg,\hg) (0.1,\hg) (0.25,{(\hg+\hf)/2}) (0.1,\hf) (-\xf,\hf)};
\draw plot [smooth, tension=0.32] coordinates {(\xg,\hg) (0.5,\hg) (0.35,{(\hg+\hf)/2}) (0.5,\hf) (\xf,\hf)};
\end{scope}
\end{tikzpicture}
%%%%%%%%%%%%%%%%%%%%%%%%%%%%%%%%%%%%%%%%%%%%%%%%%%%%%
%%%%%%%%%%%%%%%%%%%%%%%%%%%%%%%%%%%%%%%%%%%%%%%%%%%%%
\caption{Blow-up around the points of degeneration.} \label{fig:GlobalDescr}
\end{figure}
Then note that, exactly as in the proof of \cref{prop:LocDeg}, one can choose $\eps_0>0$, possibly smaller than before, in such a way that 
\[
\abs{\A_{\ms_j}}(x) d_g(x,\set_\infty) < {\mu_0}
\]
for $x\in \ms_j \cap (B_{\eps_0}(\set_\infty) \setminus B_{r_j}(\set_\infty))$, for $j$ sufficiently large.
{Therefore, thanks to \cref{rem:ConvSigma1}, we can apply \cref{cor:14TopInfo} and thus transfer the information on the genus and the boundary components of $\ms_j^{\euno}\cap B_{r_j}(p_\infty)$ to $\ms_j^{\euno}\cap B_{\eps_0}(p_\infty)$ (possibly taking $\kappa(I)$ as double the constant given by \cref{prop:LocDeg}).}

Let us now prove that the components of $\ms_j^{\edue}$ have uniformly bounded curvature, i.e.
\[
\limsup_{j\to\infty}\sup_{x\in \ms_j^{\edue}} \abs{A_{\ms_j}}(x) < \infty \point
\]
Assume by contradiction that (possibly after passing to a subsequence) there exists a sequence $z_j\in \ms_j^{\edue}$ satisfying
\[
\lambda_j^{\edue}\eqdef \abs{\A_{\ms_j}}(z_j) = \sup_{x\in \ms_j^{\edue}} \abs{\A_{\ms_j}}(x) \to \infty \point
\]
Observe that, by \eqref{eq:CurvEstSj}, this implies that (up to subsequence) the sequence $z_j$ converges to some point $p_\infty\in \set_\infty$. In particular the distance between $z_j$ and $\partial\ms_j^{\edue}\setminus\partial\amb$ is bounded from below by a positive constant, hence $\lambda_j^{\edue}(B_{\eps_0}(z_j)-z_j)$ is an exhausting sequence of domains of $\Pi(a)$ for some $0\ge a \ge -\infty$.

Now consider the rescaled surfaces $\bu\ms_j\eqdef \lambda_j^{\edue}(\ms_j -z_j)$: they have bounded curvature away from a finite set of points and that $\abs{A_{\bu\ms_j}}(0) = 1$. Therefore, by \crefnoname{thm:LamCptness} and \cref{prop:LimLamInR3}, $\bu\ms_j$ converges locally smoothly with multiplicity one (in the sense of graphs) to a complete, non-flat, connected, properly embedded, two-sided, free boundary minimal surface $\bu \ms_\infty$ in $\Xi(a)$, for some $0\ge a \ge-\infty$.

The limit of the surfaces $\lambda_j^{\edue}(\ms_j^{\euno}-z_j)$ must be non-empty since \[\limsup_{j\to\infty} \min_{p\in \set_j} \ \lambda_j^{\edue} d_{g}(z_j,p) < \infty \] by \eqref{eq:CurvEstSj}. 
However, this contradicts the fact that $\lambda_j^{\edue}(\ms_j -z_j)$ converges to $\bu\ms_\infty$ with multiplicity one since $\lambda_j^{\edue}(\ms_j^{\euno} -z_j)$ and $\lambda_j^{\edue}(\ms_j^{\edue} -z_j)$ are two different connected components of $\lambda_j^{\edue}(\ms_j^{\euno}-z_j)$ and the limit of both of them is non-empty.

Given the bound on the curvature of the components of $\ms_j^{\edue}$, the information on their topology follows easily (possibly taking $\eps_0$ smaller), using that the convergence of $\ms_j^{\edue}$ must be smooth (in the sense of lamination) everywhere in $B_{\eps_0}(\set_\infty)$.
Therefore each component of $\ms_j^{\edue}$ converges smoothly with multiplicity one to a leaf of $\lam$ (by simply connectedness of the components of $\lam$ in $B_\eps(p_\infty)$).

There only remains to prove the bounds \ref{gd:AreaBoundSigmaI} and \ref{gd:AreaBoundSigmaII} on the area.
Fixing $\eps\le\eps_0$, the surfaces $\ms_j^{\euno}$ intersects $\partial B_{\eps}(p_\infty)\setminus \partial\amb$ in at most $\kappa(I)$ components that are simple closed curves or arcs. This implies that $\ms_j^{\euno} \cap (B_\eps(p_\infty)\setminus B_{\eps/2}(p_\infty))$ converges graphically smoothly with finite multiplicity bounded by $\kappa(I)$ to the leaf of $\lam$ in $B_\eps(p_\infty)$ passing through $p_\infty$ (similarly to \cref{rem:ConvSigma1}). As a result, choosing $\eps_0$ sufficiently small such that all the leaves of $\lam$ in $B_{\eps}(p_\infty)$ are sufficiently close to discs or half-discs, we have that 
\[
\limsup_{j\to\infty}\ \area(\ms_j^{\euno} \cap (B_{\eps}(\set_\infty)\setminus B_{\eps/2}(\set_\infty))) \le 2\kappa(I)\pi\eps^2\point
\]
Finally the estimate on $\area(\ms_j^{\euno}\cap B_\eps(\set_\infty))$ follows from the monotonicity formula since, for $\eps>0$ sufficiently small, it holds
\[
\area(\ms_j^{\euno}\cap B_\eps(\set_\infty)) \le 2 \area(\ms_j^{\euno}\cap (B_{\eps}(\set_\infty)\setminus B_{\eps/2}(\set_\infty))) \point
\]
The area estimate for the components of $\ms_j^{\edue}$ is even easier since $\ms_j^{\edue}$ converges (in the sense of laminations) to $\lam$ everywhere in $B_\eps(\set_\infty)$ (in particular, each component of $\ms_j^{\edue}$ converges smoothly with multiplicity one to a leaf of $\lam$ in $B_\eps(\set_\infty)$, as observed above), thus we omit the details.
\end{proof}

\section{Surgery procedure}\label{sec:Surgery}

As a corollary of the degeneration description in \cref{thm:GlobalDeg}, we are able to perform surgeries on the surfaces $\ms_j$ to obtain new surfaces `similar' to $\ms_j$ but with bounded curvature.

\begin{corollary} \label{cor:Surgery}

Given $I\in \N$ there exists $\tilde\kappa(I)\ge 0$ such that the following assertions hold true.

In the setting of \cref{thm:GlobalDeg}, and for the same value of $\eps_0$,
for all $0< \eps\le \eps_0$ and $j$ sufficiently large, there exist properly embedded surfaces $\tilde \ms_j\subset \amb$ satisfying the following properties:
\begin{enumerate} [label={\normalfont(\arabic*)}]
\item $\tilde\ms_j$ coincides with $\ms_j$ outside $B_\eps(\set_\infty)$.
\item The curvature of $\tilde\ms_j$ is uniformly bounded, i.e.
\[
\limsup_{j\to\infty}\sup_{x\in \tilde\ms_j} \abs{\A_{\tilde\ms_j}}(x) < \infty \point
\]
\item The genus, the number of boundary components, the area and the number of connected components of $\tilde\ms_j$ are controlled by the ones of $\ms_j$, namely \label{s:SimChar}
\begin{align*}
&\genus(\ms_j) - \tilde\kappa(I) \le \genus(\tilde\ms_j) \le \genus(\ms_j)\comma \\
&\bdry( \ms_j) - \tilde\kappa(I) \le \bdry(\tilde\ms_j) \le \bdry( \ms_j) \comma \\
&\area( \ms_j) - \tilde\kappa(I) \le \area(\tilde\ms_j) \le \area(\ms_j) + \tilde\kappa(I) \comma \\
&\abs{\pi_0(\ms_j)} \le \abs{\pi_0(\tilde\ms_j)} \le \abs{\pi_0(\ms_j)} + \tilde\kappa(I) \point
\end{align*}
\item The surfaces $\tilde\ms_j$ locally converge (in the sense of laminations) to the lamination $\lam$ in \cref{cor:ExistenceBlowUpSetAndCurvatureEstimate}.

\end{enumerate}

\end{corollary}
\begin{proof}
Consider $p_\infty\in \set_\infty$; then, if $p_\infty\in\amb\setminus\partial\amb$, we can perform the surgery as in \cite[Corollary 1.19]{ChoKetMax17} (possibly restricting $\eps_0$). Therefore, let us assume that $p_\infty\in \partial\amb$. Pick the leaf $L$ of $ \lam \cap B_\eps(p_\infty)$ passing through $p_\infty$, which satisfies $\partial L = L\cap\partial\amb$ thanks to property \hypP{}.
Then fix a diffeomorphism $\psi:B_\eps(p_\infty)\to B_3(0)\cap \Xi(0)\subset \R^3$ such that $\psi$ maps $B_{\eps/3}(p_\infty)$ diffeomorphically onto $B_1(0)\cap \Xi(0)$ and $L$ onto the flat half-disc $\{ x^3 = 0\} \cap (B_3(0)\cap \Xi(0))$.

Consider all the connected components of $\ms_j\cap B_\eps(p_\infty)$ which are converging smoothly to $L$ in $B_\eps(p_\infty)\setminus \bar B_{\eps/3}(p_\infty)$. These include all the neck components, called $\ms_j^{\euno}$ in \cref{thm:GlobalDeg}, as observed in the proof of the theorem.
Note that the convergence to the leaf $L$ in $B_\eps(p_\infty)\setminus \bar B_{\eps/3}(p_\infty)$ might possibly occur with infinite multiplicity; however, the convergence of the components relative to $\ms_j^{\euno}$ occurs with uniformly bounded multiplicity (for example thanks to the area bound \ref{gd:AreaBoundSigmaI} in \cref{thm:GlobalDeg}, or from the proof of \ref{gd:AreaBoundSigmaI} itself).

Let $\Gamma_j$ be the image through $\psi$ of the union of the necks components of $\ms_j\cap B_\eps(p_\infty)$ together with the disc components directly above or below a neck component.
We are going to perform our surgery on $\Gamma_j$ leaving invariant its disc components, in this way we are sure that also all the other disc components remain untouched.
For the sake of convenience, let us identify $L$ with its image in $B_3(0)\cap \Xi(0)$ and let us denote by $D(2)$ the disc with radius $2$ and center $0$ in $L$, that is $D(2)\eqdef L\cap \{\abs x < 2\}$, and $A(2,1)$ the annulus between radii $1$ and $2$ in $L$, namely $A(2,1)\eqdef L\cap \{1<\abs x < 2\}$. 
Then choose $\chi:\R\to \cc 01$ a smooth, non-decreasing cutoff function with $\chi(t) = 0$ for $t\le 5/4$ and $\chi(t) = 1$ for $t\ge 7/4$.

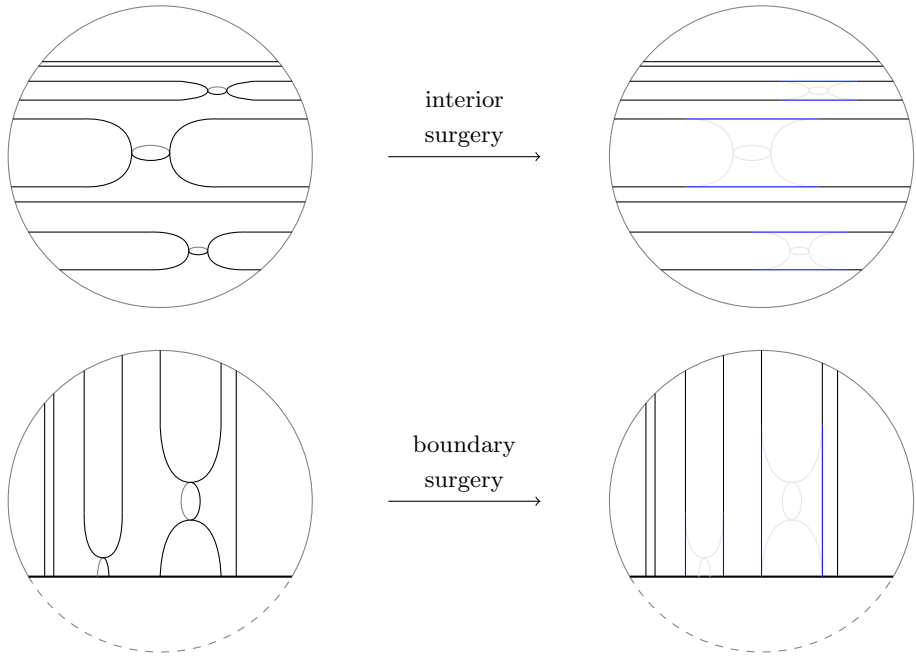
\begin{figure}[htpb]
\centering
%%%%%%%%%%%%%%%%%%%%%%%%%%%%%%%%%%%%%%%%%%%%%%%%%%%%%
%%% surgery_micro.tikz
%%%%%%%%%%%%%%%%%%%%%%%%%%%%%%%%%%%%%%%%%%%%%%%%%%%%%
\begin{tikzpicture}[scale=2.5]
\pgfmathsetmacro{\R}{2}
\pgfmathsetmacro{\r}{0.8}

% No boundary

\coordinate (OA) at (0,0);

\draw[gray] (OA) circle (\r);
\def\h{{-0.3, 0.6, 0.63}}

\foreach \i in {0,...,2}
  \draw ({-sqrt(\r*\r*(1-\h[\i]*\h[\i]))},{\r*\h[\i]}) -- ({sqrt(\r*\r*(1-\h[\i]*\h[\i]))},{\r*\h[\i]});

% Lower cat
\pgfmathsetmacro{\ha}{-0.3*\R}
\pgfmathsetmacro{\xa}{sqrt(\r*\r-\ha*\ha)}
\pgfmathsetmacro{\hb}{-0.20*\R}
\pgfmathsetmacro{\xb}{sqrt(\r*\r-\hb*\hb)}

\begin{scope}
\draw[gray]  (0.15,{(\ha+\hb)/2}) arc (180:0:0.05 and 0.02);
\draw  (0.15,{(\ha+\hb)/2}) arc (180:360:0.05 and 0.02);
\draw plot [smooth, tension=2] coordinates {(-0.05,\ha) (0.15,{(\ha+\hb)/2}) (-0.05,\hb)};
\draw plot [smooth, tension=2] coordinates {(0.45,\hb) (0.25,{(\hb+\ha)/2}) (0.45,\ha)};

\draw (-\xa,\ha) -- (-0.05,\ha);
\draw (-0.05,\hb) -- (-\xb,\hb);
\draw (\xa,\ha) -- (0.45,\ha);
\draw (0.45,\hb) -- (\xb,\hb);

% Middle cat
\pgfmathsetmacro{\hd}{-0.08*\R}
\pgfmathsetmacro{\xd}{sqrt(\r*\r-\hd*\hd)}
\pgfmathsetmacro{\he}{0.10*\R}
\pgfmathsetmacro{\xe}{sqrt(\r*\r-\he*\he)}
\draw[gray] (-0.15,{(\he+\hd)/2}) arc (180:0:0.1 and 0.04);
\draw  (-0.15,{(\he+\hd)/2}) arc (180:360:0.1 and 0.04);
\draw  plot [smooth, tension=2] coordinates {(-0.4,\he) (-0.15,{(\he+\hd)/2}) (-0.4,\hd)};
\draw  plot [smooth, tension=2] coordinates {(0.3,\he) (0.05, {(\he+\hd)/2}) (0.3,\hd)};
\draw (-\xe,\he) -- (-0.4,\he);
\draw (-0.4,\hd) -- (-\xd,\hd);
\draw (\xe,\he) -- (0.3,\he);
\draw (0.3,\hd) -- (\xd,\hd);

% Higher cat 
\pgfmathsetmacro{\hf}{0.15*\R}
\pgfmathsetmacro{\xf}{sqrt(\r*\r-\hf*\hf)}
\pgfmathsetmacro{\hg}{0.20*\R}
\pgfmathsetmacro{\xg}{sqrt(\r*\r-\hg*\hg)}
\draw[gray]  (0.25,{(\hf+\hg)/2}) arc (180:0:0.05 and 0.02);
\draw  (0.25,{(\hf+\hg)/2}) arc (180:360:0.05 and 0.02);
\draw plot [smooth, tension=1.5] coordinates {(0.1,\hg) (0.25,{(\hg+\hf)/2}) (0.1,\hf)};
\draw plot [smooth, tension=1.5] coordinates {(0.5,\hg) (0.35,{(\hg+\hf)/2}) (0.5,\hf)};
\draw (-\xg,\hg) -- (0.1,\hg);
\draw (\xg,\hg) -- (0.5,\hg);
\draw (0.1,\hf) -- (-\xf,\hf);
\draw (0.5,\hf) -- (\xf,\hf);
\end{scope}

\draw[->] (1.2,0) -- (2,0);
\node at (1.6,0.2) [text width=1.5cm,align=center]{\footnotesize interior surgery};

\begin{scope}[xshift=45*\R]
% No boundary

\coordinate (OA) at (0,0);
\draw[gray] (OA) circle (\r);
\def\h{{-0.3, 0.6, 0.63}}

\foreach \i in {0,...,2}
  \draw ({-sqrt(\r*\r*(1-\h[\i]*\h[\i]))},{\r*\h[\i]}) -- ({sqrt(\r*\r*(1-\h[\i]*\h[\i]))},{\r*\h[\i]});

% Lower cat
\pgfmathsetmacro{\ha}{-0.3*\R}
\pgfmathsetmacro{\xa}{sqrt(\r*\r-\ha*\ha)}
\pgfmathsetmacro{\hb}{-0.20*\R}
\pgfmathsetmacro{\xb}{sqrt(\r*\r-\hb*\hb)}

\begin{scope}
\draw[gray!20]  (0.15,{(\ha+\hb)/2}) arc (180:0:0.05 and 0.02);
\draw[gray!20]  (0.15,{(\ha+\hb)/2}) arc (180:360:0.05 and 0.02);
\draw[gray!20] plot [smooth, tension=2] coordinates {(-0.05,\ha) (0.15,{(\ha+\hb)/2}) (-0.05,\hb)};
\draw[gray!20] plot [smooth, tension=2] coordinates {(0.45,\hb) (0.25,{(\hb+\ha)/2}) (0.45,\ha)};

\draw (-\xa,\ha) -- (-0.05,\ha);
\draw (-0.05,\hb) -- (-\xb,\hb);
\draw (\xa,\ha) -- (0.45,\ha);
\draw (0.45,\hb) -- (\xb,\hb);

% Middle cat
\pgfmathsetmacro{\hd}{-0.08*\R}
\pgfmathsetmacro{\xd}{sqrt(\r*\r-\hd*\hd)}
\pgfmathsetmacro{\he}{0.10*\R}
\pgfmathsetmacro{\xe}{sqrt(\r*\r-\he*\he)}

\draw[gray!20] (-0.15,{(\he+\hd)/2}) arc (180:0:0.1 and 0.04);
\draw[gray!20]  (-0.15,{(\he+\hd)/2}) arc (180:360:0.1 and 0.04);
\draw[gray!20]  plot [smooth, tension=2] coordinates {(-0.4,\he) (-0.15,{(\he+\hd)/2}) (-0.4,\hd)};
\draw[gray!20]  plot [smooth, tension=2] coordinates {(0.3,\he) (0.05, {(\he+\hd)/2}) (0.3,\hd)};

\draw (-\xe,\he) -- (-0.4,\he);
\draw (-0.4,\hd) -- (-\xd,\hd);
\draw (\xe,\he) -- (0.3,\he);
\draw (0.3,\hd) -- (\xd,\hd);

% Higher cat 
\pgfmathsetmacro{\hf}{0.15*\R}
\pgfmathsetmacro{\xf}{sqrt(\r*\r-\hf*\hf)}
\pgfmathsetmacro{\hg}{0.20*\R}
\pgfmathsetmacro{\xg}{sqrt(\r*\r-\hg*\hg)}

\draw[gray!20]  (0.25,{(\hf+\hg)/2}) arc (180:0:0.05 and 0.02);
\draw[gray!20]  (0.25,{(\hf+\hg)/2}) arc (180:360:0.05 and 0.02);
\draw[gray!20] plot [smooth, tension=1.5] coordinates {(0.1,\hg) (0.25,{(\hg+\hf)/2}) (0.1,\hf)};
\draw[gray!20] plot [smooth, tension=1.5] coordinates {(0.5,\hg) (0.35,{(\hg+\hf)/2}) (0.5,\hf)};

\draw (-\xg,\hg) -- (0.1,\hg);
\draw (\xg,\hg) -- (0.5,\hg);
\draw (0.1,\hf) -- (-\xf,\hf);
\draw (0.5,\hf) -- (\xf,\hf);

\begin{scope}[solid,blue]
\draw[blue] (-0.05,\ha) -- (0.45,\ha);
\draw[blue] (-0.05,\hb) -- (0.45,\hb);
\draw[blue] (-0.4,\he) -- (0.3,\he);
\draw[blue] (-0.4,\hd) -- (0.3,\hd);
\draw[blue] (0.1,\hg) -- (0.5,\hg);
\draw[blue] (0.1,\hf) -- (0.5,\hf);
\end{scope}
\end{scope}
\end{scope}

\begin{scope}[yshift=-26*\R]
% Boundary
\pgfmathsetmacro{\hl}{-0.20*\R}
\pgfmathsetmacro{\xl}{sqrt(\r*\r-\hl*\hl)}
\coordinate (OA) at (0,0);
\pgfmathsetmacro{\s}{atan(-\hl/(sqrt(\r*\r-\hl*\hl)))}
\draw[gray] ({(sqrt(\r*\r-\hl*\hl))},\hl) arc(-\s:180+\s:\r);
\draw[gray, dashed] (-{(sqrt(\r*\r-\hl*\hl))},\hl) arc(180+\s:360-\s:\r);
\draw[thick] (-\xl,\hl) -- (\xl,\hl);

\def\c{{-0.76,-0.7, 0.5}}
\foreach \i in {0,...,2}
\draw ({\r*\c[\i]}, {sqrt(\r*\r*(1-\c[\i]*\c[\i]))}) -- ({\r*\c[\i]}, \hl);

% Left cat
\pgfmathsetmacro{\ca}{-0.2*\R}
\pgfmathsetmacro{\ya}{sqrt(\r*\r-\ca*\ca)}
\pgfmathsetmacro{\cb}{-0.10*\R}
\pgfmathsetmacro{\yb}{sqrt(\r*\r-\cb*\cb)}

\draw[gray] ({(\ca+\cb)/2}, {\hl+0.05*\R}) arc (90:180:0.03 and {0.05*\R});
\draw ({(\ca+\cb)/2}, {\hl+0.05*\R}) arc (90:0:0.03 and {0.05*\R});
\draw plot [smooth, tension=2] coordinates {(\ca, {\hl+0.17*\R}) ({(\ca+\cb)/2}, {\hl+0.05*\R}) (\cb,{\hl+0.17*\R})};

\draw (\ca,\ya) -- (\ca, {\hl+0.17*\R});
\draw (\cb,{\hl+0.17*\R}) -- (\cb,\yb);

% Right cat
\pgfmathsetmacro{\cc}{0*\R}
\pgfmathsetmacro{\yc}{sqrt(\r*\r-\cc*\cc)}
\pgfmathsetmacro{\cd}{0.16*\R}
\pgfmathsetmacro{\yd}{sqrt(\r*\r-\cd*\cd)}

\draw[gray] ({(\cc+\cd)/2}, {\hl+0.25*\R}) arc (90:270:0.05 and {0.05*\R});
\draw ({(\cc+\cd)/2}, {\hl+0.25*\R}) arc (90:-90:0.05 and {0.05*\R});
\draw plot [smooth, tension=1.8] coordinates {(\cc, {\hl+0.4*\R}) ({(\cc+\cd)/2}, {\hl+0.25*\R}) (\cd,{\hl+0.4*\R})};
\draw plot [smooth, tension=1.8] coordinates {(\cc,\hl) ({(\cc+\cd)/2}, {\hl+0.15*\R}) (\cd,\hl)};

\draw (\cc,\yc) -- (\cc, {\hl+0.4*\R});
\draw (\cd,{\hl+0.4*\R}) -- (\cd,\yd);

\draw[->] (1.2,0) -- (2,0);
\node at (1.6,0.2) [text width=1.5cm,align=center]{\footnotesize boundary surgery};

\begin{scope} [xshift=45*\R]
% Boundary

\pgfmathsetmacro{\hl}{-0.20*\R}
\pgfmathsetmacro{\xl}{sqrt(\r*\r-\hl*\hl)}
\coordinate (OA) at (0,0);

\pgfmathsetmacro{\s}{atan(-\hl/(sqrt(\r*\r-\hl*\hl)))}
\draw[gray] ({(sqrt(\r*\r-\hl*\hl))},\hl) arc(-\s:180+\s:\r);
\draw[gray, dashed] (-{(sqrt(\r*\r-\hl*\hl))},\hl) arc(180+\s:360-\s:\r);
\draw[thick] (-\xl,\hl) -- (\xl,\hl);

\def\c{{-0.76,-0.7, 0.5}}

\foreach \i in {0,...,2}
\draw ({\r*\c[\i]}, {sqrt(\r*\r*(1-\c[\i]*\c[\i]))}) -- ({\r*\c[\i]}, \hl);

% Left cat
\pgfmathsetmacro{\ca}{-0.2*\R}
\pgfmathsetmacro{\ya}{sqrt(\r*\r-\ca*\ca)}
\pgfmathsetmacro{\cb}{-0.10*\R}
\pgfmathsetmacro{\yb}{sqrt(\r*\r-\cb*\cb)}

\draw[gray!20] ({(\ca+\cb)/2}, {\hl+0.05*\R}) arc (90:180:0.03 and {0.05*\R});
\draw[gray!20] ({(\ca+\cb)/2}, {\hl+0.05*\R}) arc (90:0:0.03 and {0.05*\R});
\draw[gray!20] plot [smooth, tension=2] coordinates {(\ca, {\hl+0.17*\R}) ({(\ca+\cb)/2}, {\hl+0.05*\R}) (\cb,{\hl+0.17*\R})};

\draw[blue] (\ca, {\hl+0.17*\R}) -- (\ca,\hl);
\draw[blue] (\cb, {\hl+0.17*\R}) -- (\cb,\hl);
\draw (\ca,\ya) -- (\ca, {\hl+0.17*\R});
\draw (\cb,{\hl+0.17*\R}) -- (\cb,\yb);

% Right cat
\pgfmathsetmacro{\cc}{0*\R}
\pgfmathsetmacro{\yc}{sqrt(\r*\r-\cc*\cc)}
\pgfmathsetmacro{\cd}{0.16*\R}
\pgfmathsetmacro{\yd}{sqrt(\r*\r-\cd*\cd)}

\draw[gray!20] ({(\cc+\cd)/2}, {\hl+0.25*\R}) arc (90:270:0.05 and {0.05*\R});
\draw[gray!20] ({(\cc+\cd)/2}, {\hl+0.25*\R}) arc (90:-90:0.05 and {0.05*\R});
\draw[gray!20] plot [smooth, tension=1.8] coordinates {(\cc, {\hl+0.4*\R}) ({(\cc+\cd)/2}, {\hl+0.25*\R}) (\cd,{\hl+0.4*\R})};
\draw[gray!20] plot [smooth, tension=1.8] coordinates {(\cc,\hl) ({(\cc+\cd)/2}, {\hl+0.15*\R}) (\cd,\hl)};

\draw[blue] (\cc, {\hl+0.4*\R}) -- (\cc,{\hl});
\draw[blue] (\cd, {\hl+0.4*\R}) -- (\cd,{\hl});

\draw (\cc,\yc) -- (\cc, {\hl+0.4*\R});
\draw (\cd,{\hl+0.4*\R}) -- (\cd,\yd);
\end{scope}
\end{scope}
\end{tikzpicture}
%%%%%%%%%%%%%%%%%%%%%%%%%%%%%%%%%%%%%%%%%%%%%%%%%%%%%
%%%%%%%%%%%%%%%%%%%%%%%%%%%%%%%%%%%%%%%%%%%%%%%%%%%%%
\caption{Surgery at small scales.} \label{fig:SurgeryMicro}
\end{figure}

We can assume to have only two disc components in $\Gamma_j$, one at the top and one at the bottom. If this is not the case, we work separately on subsets of the components of $\Gamma_j$ in this form, eventually adding the extremal disc components if missing.
Pick a smooth function $w_j:D(2)\to \R$ such that 
\begin{itemize}
 \item the graph of $w_j$ is contained in $B_3(0)\cap \Xi(0)$ and it lies strictly between the two disc components of $\Gamma_j$;
 \item the graph of $w_j$ intersects transversely $\Pi(0)$;
 \item $w_j$ smoothly converges to $0$ as $j\to\infty$.
\end{itemize}
Moreover choose real numbers $\delta_j\to 0$ such that $w_j+\delta$ has the same properties for every $0\le \delta \le\delta_j$.

Now, since the convergence of $\Gamma_j$ to $L$ is smooth graphical with finite multiplicity on $A(2,1)$, we can find functions $u_{j,1},\ldots,u_{j,n(j)}:A(2,1)\to \R$ so that the non-disc components of $\Gamma_j$ on the cylinder over $A(2,1)$ are exactly the graphs of $u_{j,k}$ for $k=1,\ldots,n(j)$.
Moreover observe that
\begin{itemize}
\item $n(j)$ is uniformly bounded for $j\to\infty$ for what we said before;
\item for all $h\in\N$ we have $\sup_{k=1,\ldots,n(j)} \norm{u_{j,k}}_{C^h(A(2,1))}\to 0$ as $j\to\infty$;
\item by embeddedness of $\Gamma_j$ we can assume $u_{j,1}(x) < u_{j,2}(x) < \ldots < u_{j,n(j)}(x)$ for all $x\in A(2,1)$.
\end{itemize}
We then interpolate, i.e. we define
\[
\tilde u_{j,k}(x) \eqdef \chi(\abs{x}) u_{j,k}(x) + (1-\chi(\abs x)) \left(w_j(x) + \frac{k}{n(j)} \delta_j\right)\comma 
\]
for all $x\in A(2,1)$, and the surface $\tilde \Gamma_j$ as the union of the graphs of $\tilde u_{j,k}$ for $k=1,\ldots,n(j)$ inside the cylinder over $D(2)$ and coinciding with $\Gamma_j$ outside that cylinder.
Now just define $\tilde\ms_j$ as $\ms_j$ outside $B_\eps(\set_\infty)$ and as the preimages of $\tilde\Gamma_j$ through $\psi$ inside the balls $B_\eps(\set_\infty)$.
All the properties required to $\tilde\ms_j$ are easily fulfilled by construction thanks to the bounds on the genus and the number of boundary components of $\ms^{\euno}_j$ inside $B_\eps(\set_\infty)$, proven in \cref{thm:GlobalDeg}.
\end{proof}

%%%%%%%%%%%%%%%%%%%%%%%%%%%%%%%%%%%%%%%%%%%%%%%%%%%%%%%%%%%%%%%%
%%% Diameter bounds for stable FBMS
%%%%%%%%%%%%%%%%%%%%%%%%%%%%%%%%%%%%%%%%%%%%%%%%%%%%%%%%%%%%%%%%

\section{Diameter bounds for stable free boundary minimal surfaces} \label{sec:CptnessStable}

In this section we prove \cref{prop:StableImpliesCptness}, which is key to derive the area bounds claimed in the statement of \cref{thm:AreaBound}. The main idea of the proof is that, under the assumptions of the proposition, the stable minimal surface $\ms$  admits a (complete) conformal metric with non-negative curvature and convex boundary. Moreover, at least one of the two inequalities is strict: either the curvature is strictly positive or the boundary is strictly convex, hence we expect the size of $\Sigma$ to be bounded as a consequence of this property.
In order to implement this heuristic idea, we employ several different techniques, mainly from \cite{Fis85} and \cite{Whi86} and, of course, from the same result in the closed case \cite{SchYau83}, in the form proposed in \cite[Proposition 2.12]{Car15}.

\begin{proof} [Proof of \cref{prop:StableImpliesCptness}]

Suppose by contradiction that $\ms$ is non-compact\footnote{It is straightforward to note that the same proof goes through, to provide the bound for $\text{diam}(\ms)$ in the case when $\Sigma$ is a compact, properly embedded free boundary minimal surface: indeed in that case we know that $\ms$ is a (two-sided) disc, and it suffices to take $\omega$ the first eigenfunction of the Jacobi operator $\jac_{\ms}$ (subject to the usual oblique boundary condition).}. Then by standard arguments as in \cite{FisSch80} (cf. also \cite[Section 2.2.2]{Vol15}), there exists a positive function $\omega$ on $\ms$ such that
\begin{equation} \label{eq:EllProb}
\begin{cases} 
\jac_\ms (\omega) = \lapl \omega + (\frac 12 \Scal_g + \frac 12 \abs \A^2 - K)\omega = 0 & \text{on $\ms$}\\
\frac{\partial \omega}{\partial \eta} = -\II^{\partial\amb}(\nu,\nu) \omega &\text{on $\partial\ms$} \point
\end{cases}
\end{equation}
Recall that we can choose $\omega$ strictly positive in all $\ms$ (including on $\partial\ms$) by the strong maximum principle and the Hopf boundary point lemma.

Consider on $\ms$ the conformal change of metric $\tilde g = \omega^2g$ where $g$ (which we also denote by $\scal\cdot\cdot$) is the metric on $\ms$ that is induced by the ambient metric on $\amb$.
Then we know (see e.g. \cite[pp. 126-127]{FisSch80}) that the curvature of $(\ms,\tilde g)$ is given by
\begin{equation}\label{eq:ChangeSectional}
\begin{split}
\tilde K &= \omega^{-2}(K-\lapl \log \omega) = \omega^{-2}\left( K - \frac{\omega\lapl\omega - \abs{\grad \omega}^2 }{\omega^2} \right) \\
&=\omega^{-2}\left(\frac 12 \Scal_g + \frac 12 \abs A^2\right) + \frac{\abs{\grad \omega}^2}{\omega^4} \ge \omega^{-2}\frac{\Scal_g}2 \ge\omega^{-2}\frac{\varrho_0}2\ge 0\point
\end{split}
\end{equation}
Now let $\tau$ be a local choice of unit vector tangent to $\partial\ms$ in $(\ms,g)$ and, as usual, let $\eta$ be the outward unit normal to $\partial\ms$.
After the conformal change of coordinates, $\tau/\omega$ and $\eta/\omega$ are an orthonormal basis. Hence we can compute the geodesic curvature of $\partial\ms$ in $(\ms,\tilde g)$ as follows
\begin{equation}\label{eq:ChangeGeodCurv}
\begin{split}
\tilde k &= -\tilde g({\tilde\cov_{\tau/\omega}\tau/\omega},{\eta/\omega}) = -\omega^2 \scal{\tilde\cov_{\tau/\omega}\tau/\omega}{\eta/\omega}
= -\omega^{-1}\scal{\tilde \cov_\tau \tau}\eta \\
&=-\omega^{-1}\left(\scal{\cov_\tau\tau}{\eta} - \omega^{-1} \DerParz{\omega}\eta\right) = -\omega^{-1}(\II^{\partial\amb}(\tau,\tau) + \II^{\partial\amb}(\nu,\nu)) \\ &= \omega^{-1}H^{\partial\amb}\ge \omega^{-1}\sigma_0\ge 0\point
\end{split}
\end{equation}
Therefore $(\ms,\tilde g)$ is a surface with non-negative Gaussian curvature and convex boundary and, as claimed above, either of the two functions $\tilde{K}$ or $\tilde k$ is strictly positive.

\begin{lemma}
The surface $(\ms,\tilde g)$ with boundary is complete (as a metric space).
\end{lemma}
\begin{proof}
The proof is similar to that of Theorem 1 in \cite{Fis85} with some precautions to be taken when dealing with the boundary.
Consider a point $x_0\in \ms\setminus \partial\ms$ and let $B_R(x_0)\subset \ms$ be the intrinsic ball of center $x_0$ and radius $R>0$ with respect to the complete metric $g$. Then $\{B_R(x_0)\}_{R>0}$ is an exhaustion of $\ms$.

Now consider the shortest geodesic $\gamma_R$ in the metric $\tilde g$ connecting $x_0$ to the closure of $\partial B_R(x_0)\setminus\partial\ms$, which exists by the following argument. Let us consider $\omega_R = \omega + \epsilon_R$ where $0\le\epsilon_R\le 1$ is a smooth function with $\epsilon_R = 0$ in $B_R(x_0)$ and $\epsilon_R = 1$ in $B_{R+1}(x_0)^c$. Then $\omega_R$ is bounded from below and thus the metric $\tilde g_R \eqdef \omega_R^2 g$ is complete on $\ms$. 
Based on the fact, and on the convexity of the boundary (that we just saw above), there exists a length-minimizing geodesic connecting $x_0$ to the closure of $\partial B_R(x_0)\setminus \partial\ms$ with respect to the metric $\tilde g_R$. Since this geodesic is obviously contained in $B_R(x_0)$ it is also length-minimizing with respect to $\tilde g$, since $\tilde g$ coincides with $\tilde g_R$ in $B_R(x_0)$.

Let us assume $\gamma_R$ to be parametrized by arclength with respect to $g$. By standard compactness arguments, we can find a sequence $R_i\to\infty$ such that $\gamma_{R_i}$ locally smoothly converge to a curve $\gamma: \co 0{\infty} \to \ms$ such that $\gamma(0)=x_0$, minimizing length with respect to $\tilde g$ between any two of its points, and that is also parametrized by arc length with respect to $g$.
Observe that $\gamma$ cannot touch the boundary $\partial\ms$ since it starts at an interior point and $\partial\ms$ is convex.

To prove the completeness of $(\ms,\tilde g)$ it is now sufficient to prove that $\gamma$ has infinite length with respect to $\tilde g$, that is to say
\[
\int_0^{\infty}\omega(\gamma(t))\de t = \infty \semicolon
\]
indeed, by construction, any divergent ray starting from $x_0$ has length equal or bigger than $\gamma$.
However, to prove this, we can apply exactly the same argument as in the first part of the proof of \cite[Theorem 1]{Fis85}, since $\gamma$ is actually contained in $\ms\setminus\partial\ms$. \end{proof}

Given the completeness of $(\ms,\tilde g)$, we can apply the Gauss-Bonnet theorem on its metric balls as follows.
Fix $x_0\in \ms$ and let consider $\Omega_r \eqdef \tilde B_r(x_0)$, the metric ball of radius $r$ and center $x_0$ in metric $\tilde g$.

\begin{remark}
We are now going to apply some results on geodesic balls of a Riemannian manifold taken from \cite[Section 4.4]{ShiShiTan03}, that are stated for complete manifolds without boundary.
The following observation clarifies why the same results hold in our case.

Given $r_0>0$, we can regard our manifold as a smooth subdomain of a complete manifold $(\check\ms,\check g)$ without boundary (the extension depending on $r_0>0$), such that, thanks to the convexity of $\partial \Sigma$, for $r$ close to $r_0$ the set
$\tilde B_r(x_0)\subset\ms$ is the intersection of the metric ball $\check B_r(x_0)\subset \check \ms$ with respect to the metric $\check g$ with $\ms$.
\end{remark}

Thanks to this remark, we can invoke a classical theorem by Hartman \cite{Har64} (cf. also Theorem 4.4.1 in \cite{ShiShiTan03}) and obtain that the boundary of $\Omega_r$ is a piecewise smooth embedded closed curve for almost every $r>0$. 
Moreover the length $l(r)$ of $\partial\Omega_r$ is differentiable almost everywhere with derivative given by
\[
l'(r) = \int_{\partial\Omega_r\setminus\partial\ms} (\text{geodesic curvature of $\partial\Omega_r\setminus\partial\ms$}) + \sum (\text{exterior angles of $\Omega_r$});
\]
note that \[\limsup_{r\to\infty}\ l'(r)\ge 0\point\] since $l$ only attains positive values.

Let us now pick a radius $r$ for which $\partial\Omega_r$ is a piecewise smooth embedded closed curve. Then, by the Gauss-Bonnet theorem on $\Omega_r$, it holds that
\begin{multline*}
l'(r) = \int_{\partial\Omega_r\setminus \partial\ms} (\text{geodesic curvature of $\partial\Omega_r\setminus \partial\ms$})+ \sum (\text{exterior angles of $\Omega_r$}) =\\ = 2\pi \chi (\Omega_r) - \int_{\Omega_r} \tilde K \de\tilde \Haus^2 - \int_{\partial\Omega_r\cap \partial\ms} \tilde k \de\tilde\Haus^1\comma
\end{multline*}
where $\tilde\Haus^1$ and $\tilde \Haus^2$ are the Hausdorff measures with respect to $\tilde g$ in $\ms$.
Hence, taking the upper limit on both sides and using that $\tilde K, \tilde k\ge 0$, we can conclude that
\[
0\le 2\pi\limsup_{r\to\infty}\chi(\Omega_r) - \int_\ms \tilde K\de\tilde\Haus^2 - \int_{\partial\ms} \tilde k\de\tilde\Haus^1 \point
\]
Recalling that $\chi(\Omega_r)\le 1$, this implies that both the integrals $\int_\ms \tilde K\de\tilde\Haus^2$ and $\int_{\partial\ms}\tilde k\de\tilde\Haus^1$ are finite.

Now observe that $\d\tilde\Haus^1 = \omega\d\Haus^1$ and $\d\tilde\Haus^2 = \omega^2\d\Haus^2$ by definition of $\tilde g$, where $\d\Haus^1$ and $\d\Haus^2$ are the one-dimensional and two-dimensional Hausdorff measures on $(\ms,g)$.
Hence, applying \eqref{eq:ChangeSectional} and \eqref{eq:ChangeGeodCurv}, we obtain that
\[
\int_\ms\tilde K \de\tilde\Haus^2 \ge \int_\ms \frac{\varrho_0}{2\omega^2}\de\tilde\Haus^2 = \int_\ms \frac{\varrho_0}{2\omega^2}\omega^2\de\Haus^2 = \frac{\varrho_0}2 \Haus^2(\ms)
\]
and that
\[
\int_{\partial\ms} \tilde k\de\tilde\Haus^1  = \int_{\partial\ms} \frac{\sigma_0}\omega\de\tilde\Haus^1 = \int_{\partial\ms} \frac{\sigma_0}\omega \omega\de\Haus^1 = \sigma_0 \Haus^1(\partial\ms) \point
\]
In particular we deduce\footnotetext{Observe that this inequality in the compact case follows directly from the stability inequality applied to a constant function.} that
\begin{equation} \label{eq:GBFinal}
0< \frac{\varrho_0}2 \Haus^2(\ms) + \sigma_0 \Haus^1(\partial\ms) \le 2\pi\limsup_{r\to\infty}\chi(\Omega_r) \le 2\pi\point
\end{equation}
Note that this also proves that $\ms$ has well-defined Euler characteristic $\chi(\ms) = \limsup_{r\to\infty}\chi(\Omega_r)=1$.

In particular we have obtained that $\partial\ms$ is compact if $\sigma_0>0$. 
However, since we might only have $\sigma_0=0$ and, in any case, a priori we are not in the position to invoke the isoperimetric inequality in \cite{Whi09}, we actually need the following lemma.

\begin{lemma} \label{lem:DistBdryEst}
Let $\ms^2 \subset \amb^3$ be as in \cref{prop:StableImpliesCptness} and let $x_0\in \ms$. Then $x_0$ has distance from the boundary of $\ms$ bounded by a constant depending only on $\varrho_0 = \inf_\amb R_g$ and $\sigma_0 = \inf_{\partial\amb} H^{\partial\amb}$. Namely it holds
\[
d_\ms(x_0,\partial\ms) \le \min\left\{ \frac{2\sqrt 2 \pi}{\sqrt{3\varrho_0}}, \frac{4}{3\sigma_0} \right\} \point
\]
\end{lemma}
\begin{proof}
Let us consider the functional
\[
\tilde l(\gamma) \eqdef \int_\gamma\omega
\]
on the class of $W^{1,2}$-curves $\gamma$ lying in $\ms$ and connecting $x_0$ to a point in $\partial\ms$. Note that here with $\int_\gamma \omega$ we mean the integral with respect to the arc length (in metric $g$) of $\gamma$ and thus $\tilde l(\gamma)$ is instead the length of the curve $\gamma$ with respect to $\tilde g$.
In particular let $\gamma$ be a curve minimizing this functional. Observe that this curve is smooth and has finite length since $\omega$ is strictly positive on the closure of $\ms$.

We now compute the first and second variation of the functional $\tilde l$ along $\gamma$. Without loss of generality, we can assume that $\gamma$ is parametrized by arc-length, that is $\abs{\gamma'} = 1$, and it has length $l$.
Thus let us choose a variation $\alpha: \oo {-\eps}{\eps}\times \cc 0l  \to \ms$ with $\alpha(0,\cdot) = \gamma$, $\alpha(s,0) =x_0$ and $\alpha(s,1)\in\partial\ms$.
Computing the first variation we obtain\footnote{We always omit the dependence from $s$ and $t$ if clear from the context.}
\begin{align*}
\frac{\d}{\d s} \tilde l(\alpha(s, \cdot)) &= \frac{\d}{\d s} \int_0^l \omega(\alpha(s,t))\left|{\frac{\partial \alpha}{\partial t}(s,t)}\right| \de t
= \int_0^l \d\omega(\alpha)\left[ \DerParz{\alpha}{s} \right]\left|{\DerParz{\alpha}{t}}\right| + \omega(\alpha) \frac{\scal{\cov_{\DerParz \alpha s} \DerParz\alpha t}{\DerParz\alpha t}}{\left|{\DerParz{\alpha}{t}}\right|} \de t\\
&= \int_0^l \Bigg( \d\omega(\alpha)\left[ \DerParz{\alpha}{s} \right]\left|{\DerParz{\alpha}{t}}\right| 
-\d\omega(\alpha)\left[ \DerParz\alpha t \right] \frac{\scal{ \DerParz\alpha s}{\DerParz\alpha t}}{\left|{\DerParz{\alpha}{t}}\right|}  -\omega(\alpha) \frac{\scal{\DerParz\alpha s}{\cov_{\DerParz\alpha t}\DerParz\alpha t}}{\left|{\DerParz{\alpha}{t}}\right|} 
+{}\\
&\pheq+\omega(\alpha) \frac{\scal{\DerParz\alpha s}{\DerParz\alpha t}  \scal{\cov_{\DerParz\alpha t}\DerParz\alpha t}{\DerParz\alpha t}}{\left|{\DerParz{\alpha}{t}}\right|^3} \Bigg) \de t  +  \omega(\alpha(s,l)) \frac{\scal{ \DerParz\alpha s(s,l)}{\DerParz\alpha t(s,l)}}{\left|{\DerParz{\alpha}{t}}(s,l)\right|} \point
\end{align*}
In particular, evaluating at $s=0$ and setting $X(t)=\DerParz\alpha s(t,0)$, we have that
\[
\begin{split}
0 = \frac{\d}{\d s}\Big|_{s=0} \tilde l(\alpha(s,\cdot)) &= \int_0^l \Big(\d\omega(\gamma)[X] - \d\omega(\gamma)[\gamma'] \scal{\gamma'}{X} -\omega(\gamma)\Big\langle{\cov_{\gamma'}\gamma'},{X}\Big\rangle +{}\\
&\pheq +\omega(\gamma) \Big\langle{\cov_{\gamma'}\gamma'},{\gamma'}\Big\rangle\scal{\gamma'}{X} \Big)\de t + \omega(\gamma(l)) \scal{\gamma'(l)}{X(l)}\\
&= \int_0^l \Big(\d\omega(\gamma)[X] - \d\omega(\gamma)[\gamma'] \scal{\gamma'}{X} -\omega(\gamma)\Big\langle{\cov_{\gamma'}\gamma'},{X}\Big\rangle\Big)\de t  + \omega(\gamma(l)) \scal{\gamma'(l)}{X(l)}\comma
\end{split}
\]
which holds for all variations $\alpha$ as above. Note that we have used that $0 = \frac{\d}{\d t}\abs{\gamma'}^2 = 2\scal{\cov_{\gamma'}\gamma'}{\gamma'}$.

As a result, since $X(l)$ is tangent to $\partial\ms$, we have that
\[
\begin{cases}
\grad\omega(\gamma) - \d\omega(\gamma)[\gamma'] \gamma' -\omega(\gamma)\cov_{\gamma'}\gamma' = 0 \\
\gamma'(l)\perp \partial\ms \point
\end{cases}
\]

We can then compute the second variation for $s=0$, obtaining
\[
\begin{split}
\frac{\d^2}{\d s^2}\Big|_{s=0} \tilde l(\alpha(s,\cdot)) 
&=\int_0^l \Big\langle \cov_{\DerParz\alpha s}\Big|_{s=0} \bigg( 
\grad\omega(\alpha) \left|{\DerParz{\alpha}{t}}\right|^4 - \d\omega(\alpha)\left[{\DerParz{\alpha}{t}}\right] \left|{\DerParz{\alpha}{t}}\right|^2 {\DerParz{\alpha}{t}} +{}\\
&\pheq - \omega(\alpha) \left|{\DerParz{\alpha}{t}}\right|^2 \cov_{\DerParz\alpha t}\DerParz\alpha t + \omega(\alpha) \Big\langle{\cov_{\DerParz\alpha t}\DerParz\alpha t },{\DerParz\alpha t}\Big\rangle \DerParz\alpha t
\bigg),
X\Big\rangle \de t+{}\\
&\pheq + \omega(\gamma(l)) \left( \Big\langle{\cov_{\DerParz\alpha s}\Big|_{s=0}\DerParz\alpha s (s,l)},{\gamma'(l)}\Big\rangle + \Big\langle{X(l)},{\cov_{\DerParz\alpha s}\Big|_{s=0} \DerParz\alpha t(s,l)} \Big\rangle\right) \point
\end{split}
\]
Now assume that $X(t)$ is of the form $\psi(t)\tau$ where $\tau$ is a unit vector field orthogonal to $\gamma'$. Then we have that
\[
\begin{split}
&\frac{\d^2}{\d s^2}\Big|_{s=0} \tilde l(\alpha(s,\cdot)) = \int_0^l (\psi^2\lapl \omega - \psi^2\omega'' - \psi\psi'\omega' - \psi\psi''\omega - \psi^2K\omega)\de t + \omega(l)(\psi\psi' + \psi^2\II^{\partial\amb}(\tau,\tau))\\
&= \int_0^l \left(\psi^2\jac_\ms\omega - \frac 12\psi^2(R_g + \abs \A^2) \omega - \psi^2\omega'' - \psi\psi'\omega' - \psi\psi''\omega\right)\de t  + \omega(l)(\psi\psi'(l) + \psi^2(l)\II^{\partial\amb}(\tau,\tau))
\point
\end{split}
\]

Let $h:\cc 0l\to \R$ be the first (positive) eigenfunction for the eigenvalue problem
\[
\begin{cases}
\psi'' + \omega^{-1}\omega'\psi' + (\frac 12 R_g + \frac 12 \abs \A^2- \omega^{-1}\jac_\ms\omega + \omega^{-1}\omega'') \psi + \lambda\psi = 0\\
\psi(0) = 0\\
\psi'(l) = -\II^{\partial\amb}(\tau,\tau)\psi(l) \point
\end{cases}
\]
Then, since $\frac{\d^2}{\d s^2}|_{s=0} \tilde{l}(\alpha(s,\cdot))\ge 0$ for all variations $\alpha$, we have in particular that
\begin{multline*}
 h^{-1}h'' + \omega^{-1}\omega'h^{-1}h' + \frac 12 \varrho_0 +\omega^{-1}\omega''\le \\
 \le h^{-1}h'' + \omega^{-1}\omega'h^{-1}h' + \frac 12 R_g + \frac 12 \abs \A^2- \omega^{-1}\jac_\ms\omega + \omega^{-1}\omega''=-\lambda \le 0.
\end{multline*}
Therefore, multiplying this inequality by a test function $\xi\in C^\infty(\cc 0l)$ with $\xi(0)=0$ and integrating by parts, we obtain
\begin{align*}
0&\ge \int_0^l \left(h^{-1}h'' + \omega^{-1}\omega'h^{-1}h' + \frac 12 \varrho_0 + \omega^{-1}\omega'' \right)\xi^2 \de t\\
& = \int_0^l  \left( h^{-2}(h')^2+ \omega^{-2}(\omega')^2+ \omega^{-1}\omega'h^{-1}h'  \right)\xi^2 + \frac 12\varrho_0\xi^2 - 2 (h^{-1}h' + \omega^{-1}\omega')\xi\xi' \de t+ {}\\
&\pheq + (h^{-1}(l)h'(l)+\omega^{-1}(l)\omega'(l)) \xi^2(l) \\ 
& = \int_0^l  \frac 12\left( h^{-2}(h')^2+ \omega^{-2}(\omega')^2 \right)\xi^2+ \frac 12\varrho_0 \xi^2 + \frac 12\left(\frac{\d}{\d t}(\log(\omega h))\right)^2\xi^2 \de t+{}\\
&\pheq- 2\int_0^l \frac{\d}{\d t}(\log(\omega h))\xi\xi' \de t- (\II^{\partial\amb}(\tau,\tau)+\II^{\partial\amb}(\nu,\nu)) \xi^2(l) \\
& \ge \int_0^l  \frac 12\left( h^{-2}(h')^2+ \omega^{-2}(\omega')^2 \right)\xi^2+ \frac 12\varrho_0 \xi^2 + \frac 12\left(\frac{\d}{\d t}(\log(\omega h))\right)^2\xi^2 \de t+{} \\
&\pheq- 2\int_0^l \frac{\d}{\d t}(\log(\omega h))\xi\xi' \de t+ \sigma_0 \xi^2(l)\comma
\end{align*}
where we have used the boundary assumptions on $h$ and $\omega$, noting that $\omega'(l) = \scal{\grad\omega}{\gamma'(l)} = \DerParz\omega\eta$, since $\gamma'(l) = \eta(\gamma(l))$ ($\gamma'(l)$ is orthogonal to $\partial\ms$ and outward-pointing).
Furthermore, recall that $-\II^{\partial\amb}(\tau,\tau)-\II^{\partial\amb}(\nu,\nu) = H^{\partial\amb} \ge\sigma_0$.

We can then rely on the inequality
\[
2\left| \frac{\d}{\d t}(\log(\omega h))\xi\xi' \right| \le \frac 12\left( h^{-2}(h')^2+ \omega^{-2}(\omega')^2 \right)\xi^2+ \frac 12\left(\frac{\d}{\d t}(\log(\omega h))\right)^2\xi^2 + \frac 43(\xi')^2
\]
to conclude that
\[
\frac 12 \varrho_0\int_0^l \xi^2 \de t +  \sigma_0 \xi^2(l) \le \frac 43 \int_0^l (\xi')^2 \de t
\]
for all $\xi$ as above.
This proves that 
\[
l\le \delta_0 \eqdef \min \left\{\frac{2\sqrt 2\pi}{\sqrt{3\varrho_0}}, \frac 4{3\sigma_0}\right\} < \infty \comma
\]
as in \cite[Proposition 2.12]{Car15} for the first term and for example by taking $\xi(t) = t$ for the second term.
As a result we have proven that $\ms$ is contained in the $\delta_0$-neighborhood of $\partial\ms$ with respect to the intrinsic distance.
\end{proof}

 If $\varrho_0>0$ pick two points $x_0,y_0\in\ms\setminus\partial\ms$ and consider the curve $\gamma$ minimizing $\tilde l(\gamma)$ and connecting $x_0$ to $y_0$. Since $\gamma$ is a minimizing curve in the metric $\tilde g$ of $\ms$, then it cannot touch the convex boundary $\partial\ms$. Therefore we can follow the very same argument as in \cite[Proposition 2.12]{Car15} to prove that the length of $\gamma$ in the metric $g$ is bounded by $\frac{2\sqrt 2 \pi}{\sqrt{3\varrho_0}}$.

If $\sigma_0>0$ we conclude by observing that \cref{lem:DistBdryEst}, together with the compactness of $\partial\ms$, proves the compactness of $\ms$. The diameter estimate follows by simply combining equation \eqref{eq:GBFinal}, with \cref{lem:DistBdryEst}.
\end{proof}

%%%%%%%%%%%%%%%%%%%%%%%%%%%%%%%%%%%%%%%%%%%%%%%%%%%%%%%%%%%%%%%%
%%% Area bound for `positive geometry'
%%%%%%%%%%%%%%%%%%%%%%%%%%%%%%%%%%%%%%%%%%%%%%%%%%%%%%%%%%%%%%%%

\section{Ambient manifolds with `positive geometry'} \label{sec:AreaBound}
We can now capitalize our efforts and present the proof of \cref{thm:AreaBound}, which crucially exploits \cref{prop:StableImpliesCptness}.

\begin{proof}[Proof of \cref{thm:AreaBound}]
	We only need to prove the area bound, since all other conclusions then follow from \crefnoname{thm:TopFromArea}.
	Assume by contradiction that there exists a sequence of connected, compact, properly embedded, free boundary minimal surfaces $\Sigma_j^2\subset\amb$ with non-empty boundary and with $\ind(\ms_j)\le I$ and $\area(\ms_j)\to \infty$. Then, there exists a point $x_0\in\amb$ such that $\Haus^2(\ms_j\cap B_r(x_0))\to\infty$ for all $r>0$.
	
	Denote by $\lam$ the limit lamination given by \cref{cor:ExistenceBlowUpSetAndCurvatureEstimate} and consider the leaf $L\in\lam$ passing through $x_0$. Thanks to \cref{thm:RemSingLimLam}, $L$ must have stable universal cover (the other case is excluded since the area is diverging around $x_0$), which we can represent by a stable free boundary minimal immersion $\varphi: \ms\to M$.
	By means of a variation of the pull-back construction presented e.g. in \cite[Section 6]{AmbCarSha18-Compactness}, we can then reduce to applying \cref{prop:StableImpliesCptness} so to conclude that $\ms$ must be a disc, hence the map $\varphi$ must be an embedding (therefore $L$ is a stable, free boundary minimal surface in $(M,g)$).
	
	Let then $\tilde\ms_j$ be the sequence obtained from $\ms_j$ by means of the surgery procedure, as per
	\cref{cor:Surgery}. Observe that the new sequence satisfies uniform curvature bounds, and still has diverging area. Also, since \cref{cor:Surgery} provides a uniform bound $\tilde{\kappa}(I)+1$ on the number of connected components of $\tilde\ms_j$, we can select and rename $\tilde\ms_j$ so that it is connected for every $j\in\mathbb{N}$, and the area concentrates near the point $x_0$.
	
	Now, fix $r>0$ sufficiently small, and assume to consider a connected component of the intersection $\tilde\ms_j\cap B_r(x_0)$ which smoothly converges with multiplicity one to $L\cap B_r(x_0)$. Since $L$ is a disc, a standard monodromy argument allows to conclude that, a posteriori, (the whole component) $\tilde\ms_j$ converges to $L$ smoothly with multiplicity one. From this fact, we derive a uniform bound for the areas of $\tilde\ms_j$, which is a contradiction.
\end{proof}

%%%%%%%%%%%%%%%%%%%%%%%%%%%%%%%%%%%%%%%%%%%%%%%%%%%%%%%%%%%%%%%%
%%% Counterexamples
%%%%%%%%%%%%%%%%%%%%%%%%%%%%%%%%%%%%%%%%%%%%%%%%%%%%%%%%%%%%%%%%

\section{Non-compact families of free boundary minimal surfaces of fixed topology}\label{sec:Counterexample}

\cref{thm:Counterexample} is proven via a suitable, rather explicit, gluing construction aimed at attaching some elementary blocks. In all cases we shall now list, the word \emph{block} refers to an ambient manifold together with a sequence of minimal surfaces (closed or having free boundary) satisfying additional requirements.

\subsection{The building blocks: Spiraling spheres}\label{subs:BlocksSphere}

\begin{lemma}[cf. {\cite[Proposition 8]{ColDeL05}}] \label{lem:intblock}
On $S^3$ there exists a Riemannian metric $g_0$ of positive scalar curvature such that $(S^3, g_0)$ contains a sequence of minimal spheres $\Sigma_j$ with arbitrarily large area and index, and converging to a singular lamination $\lam$ whose singular set consists of exactly two points lying on a leaf which is a strictly stable (two-dimensional) sphere. Furthermore, the metric in question coincides with the unit round metric in a neighborhood of two distinct points $x_0$ and $y_0$, and on both those neighborhoods such minimal spheres can be completed to a local foliation by great spheres parallel at $x_0$ and $y_0$ respectively.
\end{lemma} 

\begin{remark}
An important aspect is to clarify what we mean when writing that a local foliation is \emph{parallel at a point} (cf. \cite[Definition 7]{ColDeL05}). Given $z\in\Omega$, an open subset of the round 3-sphere, and $\mathscr{F}$, a local foliation of $\Omega$ by great spheres, we say that the foliation is parallel at $z\in F$ (for some $ F\in\mathscr{F}$, to be called the \emph{central leaf}) if 
\[
\sup_{w\in F} d(w, F')=d(z, F')
\]
for any $ F'\in\mathscr{F}$. 

The geometric picture this definition captures is easily described. Take $S^3\subset \R^4$ isometrically embedded as unit sphere and consider the foliation of $S^3\setminus \left\{ (0,0,0,\pm 1)\right\}$ consisting of the $2$-spheres obtained by slicing via vertical hyperplanes. Given any point $x=(x^1,x^2,x^3,0)$ and any open set $\Omega\ni x$ then the restriction of the above foliation to $\Omega$ is parallel at $x$, and the unique great sphere passing through $x$ is the central leaf.
\end{remark}

\begin{remark}\label{rmk:half}
For later use (cf. \cref{rmk:bmodel} and \crefnoname{subs:BlocksAnnulus}), it is helpful to introduce some related terminology. We consider the set
$\Omega^+\eqdef \left\{x \in \Omega \st  x^4\geq 0 \right\}
$
and the corresponding foliation $\mathscr{F}^+$ obtained by considering the intersection of each leaf of $\mathscr{F}$ with the domain $\Omega^+$.
We will say that $\mathscr{F}^+$ is a foliation of $\Omega^+$ by half great spheres, parallel at $x$.
\end{remark}

\begin{remark}\label{rmk:index}
It follows from the construction (see, specifically, the second paragraph of \cite[pp. 30]{ColDeL05}) that for any open set $\Omega$ containing the limit sphere, one has that $\ind(\Sigma_j\cap\Omega)\to\infty$ as $j\to\infty$, where it is understood that one only considers variations that are compactly supported in $\Omega$. Of course, it is also true that $\area({\Sigma_j\cap\Omega})\to\infty$ as $j\to\infty$.
\end{remark}

\subsection{The building blocks: Minimal tori}\label{subs:BlocksTorus}

We first provide the relevant statement and then mention the key points in the construction, to the extent this is needed in \crefnoname{subs:BlocksAnnulus} below to produce, for any given $\notb>1$, a Riemannian metric of positive scalar curvature on the 3-ball so that the resulting $3$-manifold contains a family of free boundary surfaces of genus $0$ and exactly $\notb$ boundary components.

\begin{lemma} [cf. {\cite[Lemma 12]{ColDeL05}}]
\label{lem:tori} 
On $S^3$ there exists a Riemannian metric $g_{1}$ of positive scalar curvature such that:
\begin{enumerate} [label={\normalfont(\arabic*)}]
\item {$(S^3, g_{1})$ contains a family of minimal tori $\Pi_\theta$ parametrized by $\theta\in (-\theta_0,\theta_0)$, for some $\theta_0>0$;}
\item {the metric in question coincides with the unit round metric in a neighborhood of given points $x_1$ and $y_1$ and on both those neighborhoods such minimal tori provide a local foliation by great spheres parallel at $x_1$ and $y_1$, respectively.}	
\end{enumerate}	
\end{lemma}

The construction can be schematically described as follows:
\begin{itemize}
\item{On the (topological) product manifold $\cc{-\pi/2}{\pi/2}\times S^1\times S^1$, one can consider the equivalence relation $\sim$ given by
\[
(-\pi/2,p,q) \sim (-\pi/2,p,q') \quad \forall \ p \in S^1, \ \forall \ q,q'\in S^1\comma
\]
and
\[
(\pi/2,p,q)\sim (\pi/2, p', q) \quad \forall \ p, p' \in S^1,\  \forall \ q\in S^1\comma
\]
that corresponds to `collapsing vertical (resp. horizontal) fibers on $\left\{-\pi/2\right\}\times S^1\times S^1$ (resp. on $\left\{\pi/2\right\}\times S^1\times S^1$)'. Set $\tilde{M}=M/\sim $, this manifold can be endowed with a smooth Riemannian metric $\tilde{g}$ so that $(\tilde{M},\tilde{g})$ is a $3$-sphere of positive scalar curvature, and in a neighborhood of $\left\{0\right\}\times S^1\times S^1$ the metric is isometric to the Riemannian product of $S^2\times S^1$. Hence $(\tilde{M},\tilde{g})$ contains a one-parameter family of \emph{totally geodesic} tori all having two circles in common, say $\left\{0\right\}\times \left\{0\right\}\times S^1$ and $\left\{0\right\}\times \left\{\pi\right\}\times S^1$ (where we are conveniently identifying the round unit $S^1$ with the interval $\cc{0}{2\pi}$ with endpoints attached).
}
\item{At this stage one can perform a \emph{local} modification of the metric $\tilde{g}$ near the points $(0,\pi/2,0)$ and $(0,3\pi/2,0)$ so to make it round; with the family of tori being locally isometric to a family of standard great spheres in such neighborhoods. The construction is performed by explicitly interpolating between the metric of round $S^3$ and the product metric of $S^2\times S^1$.}	
\end{itemize}	

\begin{remark}\label{rmk:divide}
We observe that:
\begin{itemize}
\item {In $(S^3,g_1)$ the surface 
\[
\tilde{\Sigma}\eqdef (\cc{-\pi/2}{\pi/2}\times\left\{\pi/2\right\}\times S^1 \cup \cc{-\pi/2}{\pi/2}\times\left\{3\pi/2\right\}\times S^1) /\sim
\]  
is a totally geodesic $2$-sphere, which divides the closed manifold into two (pairwise isometric) three-dimensional balls.}
\item {Denoted by $\Omega$ one of such balls, for any $\theta\in(-\theta_0,\theta_0)$ the intersection $\Xi_\theta\eqdef \Pi_\theta \cap \overline{\Omega}$ is checked to be a free boundary minimal annulus.}
\end{itemize}

Hence, the Riemannian manifold $(\overline{\Omega}, g_1)$ contains a family of minimal annuli $\Xi_{\theta}$ parametrized by $\theta\in (-\theta_0,\theta_0)$, for some $\theta_0>0$; the metric $g_1$ coincides with the unit round metric in a neighborhood of boundary points $x_1$ and $y_1$ (these points belonging to the two boundary circles of $\Xi_{0}$) and on both those neighborhoods $\partial\Omega$ is isometric to an equatorial $2$-sphere and the minimal annuli provide \emph{half of} a local foliation by great spheres parallel at $x_1$ and $y_1$, respectively.
\end{remark}

\subsection{The building blocks: Free boundary minimal discs}\label{subs:BlocksDisk}

Let us now, instead, move to the free boundary models. We prove this ancillary result.

\begin{lemma}\label{lem:fbpiece}For any $\eps\in (0,\pi/4)$ there exists a smooth Riemannian metric $g_2=g_2(\eps)$ on the closed ball $\overline{B^3}= (\cc{0}{1+\pi/2}\times S^2)/\sim$ with coordinates $(r,\omega)$ (where $\sim$ is the equivalence relation collapsing $\left\{1+\pi/2\right\}\times S^2$ to a point) having positive scalar curvature, and such that the following properties are satisfied:
\begin{enumerate} [label={\normalfont(\arabic*)}]
\item {$g_2$ coincides with the unit round metric of $S^3$ on the domain $(\cc{1+\eps}{1+\pi/2}\times S^2)/\sim$, and coincides with the cylindrical metric on the domain $([0,1/3]\times S^2)/\sim$;}	
\item {the resulting manifold $(\overline{B^3}, g_2)$ contains a one-parameter family $\Delta^{\euno}_\theta$ of embedded free boundary minimal discs (here $\theta\in (-\theta_0,\theta_0)$ for some $\theta_0>0$);}
\item {there exist two points $x^{\euno}_2, y^{\euno}_2$ with $r(x^{\euno}_2)= r(y^{\euno}_2)=1/2$ and $\omega(x^{\euno}_2)=-\omega(y^{\euno}_2)$, and open neighborhoods $\Omega(x^{\euno}_2)$, $\Omega(y^{\euno}_2)$ respectively, where $g_2$ is round (isometric to domains of the unit round metric of $S^3$) and the minimal discs above restrict to give two local foliations by great spheres that are parallel at $x^{\euno}_2, y^{\euno}_2$ respectively.}
\end{enumerate}
\end{lemma}

\begin{proof}
	
Consider the cylinder $I\times S^2$ (where $I=[0,1]$ and the sphere $S^2$ is endowed with the standard unit round metric), fix two antipodal points $p, q\in S^2$ and consider the family of circles passing through $(0,p)$ and $(0,q)$. Said $\Gamma$ any of those circles, then $I\times \Gamma$ is a smooth, free boundary minimal surface $\Delta$ in $I\times S^2$ (with two boundary components). Now, cap off $I\times S^2$ by identifying its upper boundary with the boundary of a hemisphere in $S^3$ (with its unit round metric). The resulting metric $\hat{g}$ is smooth away from the interface (and has, near it, the form of a warped product with the warping factor being a $C^{1,1}$ function of the distance coordinate from the interface). Furthermore, one can attach a two-dimensional half-hemisphere to $\Delta$ so to get a free boundary minimal surface, which we shall not rename, that is not yet smooth along the connecting circle but has a mild singularity there. 
Now, since (in the resulting 3-manifold) the mean curvature of the interface on both sides match (for the interface is totally geodesic on both sides), we can apply the smoothing theorem by P. Miao \cite{Mia02} in the simplest possible case (see in particular Theorem 1 therein) to get a smooth metric $\check{g}$ on the closed 3-dimensional ball, that coincides with $\hat{g}$ away from a small neighborhood of the gluing interface, and whose scalar curvature is positive. In fact, by the very way the construction is defined, namely by fiberwise convolution, it follows that the regularized metric $\check{g}$ takes the form of a warped product, i.e. we have (in the coordinates $(r,\omega)$ introduced in the statement)
\[
\check{g}=\d r^2+f^2(r)g_{S^2}
\]
where 
\[
f(r)=
\begin{cases}
1 & \text{if $r\in \cc{0}{1-\eps}$} \\
\sin(r-1+\pi/2) & \text{if $r\in \cc{1+\eps}{1+\pi/2}$}
\end{cases}
\]
for small $\eps>0$. The \emph{set} $\Delta$ is then a totally geodesic surface in this smooth Riemannian manifold, and since the metric has not been modified near the boundary component which has not been capped off one has that the free boundary condition is still fulfilled.

At this stage, let us consider the construction above starting from a one parameter family of circles $\Gamma_{\theta}$ (for $\theta$ varying in a subset of $S^1$, say $\theta\in (-\theta_0,\theta_0)$ for some $\theta_0>0$) passing through the points $(0,p), (0,q) \in \left\{0\right\} \times S^2$ and let $p', q' \in S^2$ be two antipodal points (on the great circle equidistant from $p$ and $q$) chosen so that $(0,p'), (0,q')$ have small neighborhoods that are foliated by such a family. Let $\Delta^{\euno}_{\theta}$ be the corresponding free boundary minimal surfaces.

Since the cylindrical metric has, so far, not been modified away from a small neighborhood of the gluing interface we can consider the points $x^{\euno}_2=(1/2,p')$, $y^{\euno}_2=(1/2,q')$, and open neighborhoods thereof (named $\Omega(x^{\euno}_2)$, $\Omega(y^{\euno}_2)$ respectively) that are foliated by the surfaces $\Delta^{\euno}_{\theta}$ as $\theta$ varies. Hence, we perform a local deformation of the metric in each of these neighborhoods to make it round (for which it is enough to follow, without any modification, the argument given in the second part of Appendix B of \cite{ColDeL05}). Possibly renaming those open neighborhoods (to be taken slightly smaller than we had originally defined), the metric $g_2=g_2(\eps)$, resulting from such local modifications, satisfies all desired properties.
\end{proof}

\begin{remark}\label{rmk:bmodel}
Said $(M^{\euno}_2, g^{\euno}_2)$ the manifold constructed in \cref{lem:fbpiece}, we will also need the following variant: the local modifications of the metric are performed at \emph{one interior point} (as above) and at \emph{one boundary point}. Thereby, one obtains a 3-manifold $(M^{\euno}_{2,\text{bdry}}, g^{\euno}_{2,\text{bdry}})$ of positive scalar curvature which contains a family of free boundary minimal discs, still denoted by ${\Delta^{\euno}_{\theta}}$, and having two points:
\begin{itemize}
\item $x^{\euno}_{2,\text{bdry}}$ (in the interior) having a full neighborhood where such minimal discs provide a local foliation by great spheres (parallel at $x^{\euno}_{2,\text{bdry}}$) and 
\item $y^{\euno}_{2,\text{bdry}}$ (on the boundary) having a half neighborhood where such minimal discs provide  a local foliation by half great spheres (parallel at $y^{\euno}_{2,\text{bdry}}$).
\end{itemize}
With respect to the notation employed in the proof above, one can take (for instance)
\[
x^{\euno}_{2,\text{bdry}}=(1/2,p')\comma \quad  y^{\euno}_{2,\text{bdry}}=(0,q') \point
\]
\end{remark}

\subsection{The building blocks: Free boundary minimal \texorpdfstring{$k$}{k}-annuli}\label{subs:BlocksAnnulus}

In order to prove \cref{thm:Counterexample}, we also need to construct metrics on the closed 3-ball having positive scalar curvature and containing families of free boundary minimal surfaces of genus zero and any pre-assigned number $\notb$ of boundary components. So far, this has only been accomplished for $\notb=1$ (in \crefnoname{subs:BlocksDisk}) and for $\notb=2$, as a result of \cref{rmk:divide}. To proceed we need the following \emph{free boundary analogue} of Lemma 11 in \cite{ColDeL05}.

\begin{lemma}\label{lem:join}
Let $\delta>0$ and let $\Omega^+_1\eqdef  B_{\delta}(u_1)$, $\Omega^+_2\eqdef  B_{\delta}(u_2)$ be two (relatively) open half-balls of round unit hemispheres $N^+_1, N^+_2$, centered at boundary points $u_1, u_2$ respectively. Suppose that $\mathscr{F}^+_1$ and $\mathscr{F}^+_2$ are (locally defined) foliations of $\Omega_1$ and, respectively, $\Omega_2$ by half great spheres parallel at $u_1$ and, respectively, $u_2$. Then we can join those hemispheres to obtain a smooth Riemannian manifold with boundary (having the topology of $\overline{B^3}$) of positive scalar curvature, and (possibly by considering smaller neighborhoods) the leaves of $\mathscr{F}^+_1$ and $\mathscr{F}^+_2$ can pairwise be matched to obtain free boundary minimal surfaces in the ambient manifold.
\end{lemma}

\begin{proof}
Let us double each of the two given manifolds with boundary to round spheres $N_1$ and $N_2$ and consider the families $\mathscr{F}_1, \mathscr{F}_2$ obtained by extending $\mathscr{F}^+_1$ and $\mathscr{F}^+_2$ in the obvious fashion (each leaf of $\mathscr{F}^+_i$ is extended to an equatorial $2$-sphere in $N_i$, for $i=1,2$). Since our construction is purely local, this does not affect the generality of the argument.

We know (by Lemma 11 in \cite{ColDeL05}) that one can construct a connected sum $N_1\# N_2$ and pairwise match the leaves of the foliations $\mathscr{F}_1$ and $\mathscr{F}_2$. The connecting neck $N$, diffeomorphic to $\cc{-\tau}{\tau}\times S^2$, can be described by means of spherical coordinates
\[ 
(r,\phi,\theta) \in \cc{-\tau}{\tau}\times \cc 0\pi \times S^1
\]
and each minimal surface that we produce is, when restricted to the neck, the lift of a \emph{graphical} curve $\sigma:\cc{-\tau}{\tau}\to \cc{0}{\pi}$ solving a suitable ODE (that is nothing but a geodesic equation in a degenerate metric). In other words, each such minimal surface takes (in the neck) the form
\[
\Sigma\eqdef \left\{(r,\sigma(r),\theta) \st r\in \cc{-\tau}\tau \comma \theta\in S^1 \right\} \point
\] 
That being said, these coordinates can be chosen so that the condition $\theta\in S^1_+$ (for $S^1_+\subset S^1$ a half-circle that is fixed now and for all) identifies the half-sphere $\partial\Omega^+_1$, and the same conclusion holds true for $\partial\Omega^+_2$ as well. Now, if we define 
$
N^+\eqdef \left\{(r,\phi,\theta) \in \cc{-\tau}{\tau}\times \cc 0\pi \times S^1_{+} \right\}
$,
the totally geodesic $2$-sphere $(\partial N^+_1 \setminus\Omega^+_1 ) \cup \partial_{\ell} N^+ \cup (\partial N^+_2\setminus \Omega^+_2), \text{where} \ \partial_{\ell} N^{+}=[{-\tau},{\tau}]\times \cc 0\pi \times \partial S^1_{+}$, divides the connected sum into two, mutually isometric balls. If we consider the one, among those, containing $N^+_1 \setminus\Omega^+_1$ (which we might call the \emph{upper copy}) it is straightforward to check that
\[
\Sigma^+\eqdef \left\{(r,\sigma(r),\theta) \st r\in\cc{-\tau}{\tau}\comma \theta\in S^1_+ \right\}
\] 
is indeed a free boundary minimal surface. This correspondence holds true for any closed minimal surface that is obtained by means of the \emph{wire-matching} argument by Colding-De Lellis; in particular, the matching is certainly possible for the central leaves and a family of nearby leaves, so the proof is complete.
\end{proof}

We shall now present the main, straightforward application of this gluing lemma.

\begin{lemma}
\label{lem:2model} Given any $\notb\geq 2$ there exists, 
on $\overline{B^3}$, a Riemannian metric $g^{\ebi}_{2}$ of positive scalar curvature such that:
\begin{enumerate} [label={\normalfont(\arabic*)}]
\item {$(\overline{ B^3}, g^{\ebi}_{2})$ contains a family of embedded, free boundary minimal surfaces $\Delta^{\ebi}_\theta$, having genus zero and $b$ boundary components, parametrized by $\theta\in (-\theta_0,\theta_0)$ for some $\theta_0>0$;}
\item {the metric in question coincides with the unit round metric in a neighborhood of given points $x^{\ebi}_2$ (in the interior) and $y^{\ebi}_2$ (on the boundary) and on both those neighborhoods such minimal annuli provide a local foliation by (half-)great spheres parallel at $x^{\ebi}_2$ and $y^{\ebi}_2$, respectively.}	
\end{enumerate}	
\end{lemma}

\begin{proof}
When $\notb=2$ this follows by applying \cref{lem:join} to the blocks $(M^{\euno}_{2,\text{bdry}}, g^{\euno}_{2,\text{bdry}})$ near $y^{\euno}_{2,\text{bdry}}$ (see \cref{rmk:bmodel}) and $(\overline{\Omega},g_1)$ near $x_1$ (see \cref{rmk:divide}).
The case $\notb>2$ follows by simply repeating the operation, namely joining the resulting manifold with further copies of $(\overline{\Omega},g_1)$. 
\end{proof}

\subsection{Construction of the counterexamples}

\begin{proof}[Proof of \crefnoname{thm:Counterexample}] It is convenient to divide the argument in two steps.

{\vspace{2mm}\textbf{Step 1.}}
Given integers $\notg\geq 0$ and $\notb>0$ as in the statement, let us consider:
\begin{itemize}
\item {one copy of the Riemannian 3-sphere of positive scalar curvature as per \crefnoname{lem:intblock}, which we shall refer to as $(M_0,g_0)$ and let $(x_0, y_0)$ be the pair of points mentioned in that statement;}
\item{$\notg$ copies of the Riemannian 3-sphere of positive scalar curvature produced by \crefnoname{lem:tori}, which we shall refer to as $(M^1_1, g^1_1),\ldots ,(M^{\notg}_1, g^{\notg}_1)$, and let $(x^1_1, y^1_1),\ldots ,(x_1^{\notg},y_1^{\notg})$ be the pairs of points mentioned in that statement, respectively; each $(M_1^{i}, g_1^i)$ contains a family of minimal tori that provide a foliation by great $2$-spheres of suitably small neighborhoods of $x_1^i$ and $y_1^i$;}
\item {one copy of the Riemannian (closed) $3$-ball of positive scalar curvature and totally geodesic boundary produced via \cref{lem:2model}, which we shall refer to as $(M^{\ebi}_2, g^{\ebi}_2)$ and let $(x^{\ebi}_2, y^{\ebi}_2)$ be the pair of points mentioned in that statement.}
\end{itemize}	

\begin{figure}[htpb]
\centering
%%%%%%%%%%%%%%%%%%%%%%%%%%%%%%%%%%%%%%%%%%%%%%%%%%%%%
%%% counterexample.tikz
%%%%%%%%%%%%%%%%%%%%%%%%%%%%%%%%%%%%%%%%%%%%%%%%%%%%%
\begin{tikzpicture}[scale=0.34]
\newcommand\connect{%
    \draw (-0.4,0.2) to[bend right=30] (0.4,0.2);
    \draw (-0.4,-0.2) to[bend left=30] (0.4,-0.2);
}
\pgfmathsetmacro\R{3}
\pgfmathsetmacro\a{0.6}
\pgfmathsetmacro\b{1.3}
\pgfmathsetmacro\c{1.2}

\begin{scope} [xshift=0]
\coordinate (C1) at (0,0);

\draw (C1) circle (\R);
\begin{scope}[blue]
\draw (-0.8,0.05) to[bend left=20] (0.8,0.05);
  \draw (-0.9,.1) to[bend right=25] (0.9,.1);
\draw (C1) ellipse ({\R-0.2} and {\R/3});
\end{scope}
\node at ($(C1)+(\a*\R,\b*\R)$) {$(\amb_1^1,g_1^1)$};
\node at ($(C1)+(-\c*\R,0)$) {$x_1^1$};
\node at ($(C1)+(\c*\R,0)$) {$y_1^1$};
\end{scope}

\begin{scope} [xshift=140] \connect \end{scope}

\begin{scope} [xshift=280]
\coordinate (C1) at (0,0);
\draw (C1) circle (\R);
\begin{scope}[blue]
\draw (-0.8,0.05) to[bend left=20] (0.8,0.05);
\draw (-0.9,.1) to[bend right=25] (0.9,.1);
\draw (C1) ellipse ({\R-0.2} and {\R/3});
\end{scope}
\node at ($(C1)+(\a*\R,\b*\R)$) {$(\amb_1^2,g_1^2)$};
\node at ($(C1)+(-\c*\R,0)$) {$x_1^2$};
\node at ($(C1)+(\c*\R,0)$) {$y_1^2$};
\end{scope}

\begin{scope} [xshift=420] \connect \end{scope}

\begin{scope} [xshift=560]
\coordinate (C1) at (0,0);
\draw (C1) circle (\R);
\begin{scope}[blue]
\draw plot [smooth cycle, tension = 0.55] coordinates {(-0.9*\R, 0) (-0.65*\R, 0.4*\R) (-0.5*\R, 0*\R) (0,0.5*\R) (0.3*\R, -0.1*\R) (0.6*\R, 0.3*\R) (0.95*\R, 0) (0.6*\R, 0*\R) (0.3*\R, -0.5*\R) (-0.05*\R,0.05*\R) (-0.5*\R, -0.4*\R) (-0.7*\R, 0.1*\R)};
\end{scope}
\node at ($(C1)+(\a*\R,\b*\R)$) {$(\amb_0,g_0)$};
\node at ($(C1)+(-\c*\R,0)$) {$x_0$};
\node at ($(C1)+(\c*\R,0)$) {$y_0$};
\end{scope}

\begin{scope} [xshift=700] \connect \end{scope}

\begin{scope} [xshift=865]
\coordinate (C1) at (0,-0.5*\R);
\pgfmathsetmacro\r{1.1*\R}

\draw ($(C1)+(-\r,{0.4*\r})$) arc (180:0:{\r});
\draw ($(C1)+(-\r,{0.4*\r})$) -- ($(C1)+(-\r,0)$);
\draw ($(C1)+(\r,{0.4*\r})$) -- ($(C1)+(\r,0)$);

\draw ($(C1)+(-\r,0)$) arc (180:360:{\r} and {0.3*\r});
\draw[dashed] ($(C1)+(-\r,0)$) arc (180:0:{\r} and {0.3*\r});

\begin{scope}[blue]

\coordinate (C2) at (-0.7*\r,-0.5*\R);
\coordinate (C3) at (-0.25*\r,-0.5*\R);
\coordinate (C4) at (0.25*\r,-0.5*\R);
\coordinate (C5) at (0.7*\r,-0.5*\R);

\pgfmathsetmacro\l{0.15*\R}

\draw ($(C2)+(-\l,0)$) arc (180:360:{\l} and {0.3*\l});
\draw[dashed] ($(C2)+(-\l,0)$) arc (180:0:{\l} and {0.3*\l});

\draw ($(C3)+(-\l,0)$) arc (180:360:{\l} and {0.3*\l});
\draw[dashed] ($(C3)+(-\l,0)$) arc (180:0:{\l} and {0.3*\l});

\draw ($(C4)+(-\l,0)$) arc (180:360:{\l} and {0.3*\l});
\draw[dashed] ($(C4)+(-\l,0)$) arc (180:0:{\l} and {0.3*\l});

\draw ($(C5)+(-\l,0)$) arc (180:360:{\l} and {0.3*\l});
\draw[dashed] ($(C5)+(-\l,0)$) arc (180:0:{\l} and {0.3*\l});

\draw plot [smooth, tension = 1] coordinates {($(C2)+(\l,0)$)  (-0.48*\r,-0.35*\R) ($(C3)+(-\l,0)$) };

\draw plot [smooth, tension = 1] coordinates {($(C3)+(\l,0)$)  (0*\r,-0.35*\R) ($(C4)+(-\l,0)$) };

\draw plot [smooth, tension = 1] coordinates {($(C4)+(\l,0)$)  (0.48*\r,-0.35*\R) ($(C5)+(-\l,0)$) };

\draw plot [smooth, tension = 1] coordinates {($(C2)+(-\l,0)$)  (-0.5*\r,0.3*\R) (0.5*\r,0.3*\R) ($(C5)+(\l,0)$) };

\end{scope}
\node at ($(0,0.1)+(\a*\R,\b*\R)$) {$(\amb_2^{\scriptscriptstyle(4)},g_2^{\scriptscriptstyle(4)})$};
\node at ($(-0.5,0)+(-\c*\R,0)$) {$x_2^{\scriptscriptstyle(4)}$};
\end{scope}

\begin{scope} [xshift=990,yshift=-40] \connect \end{scope}

\begin{scope} [xshift=1080,scale=1.2]
\coordinate (C1) at (0,-0.5*\R);
\pgfmathsetmacro\r{0.5*\R}

\draw ($(C1)+(-\r,0)$) arc (180:360:{\r} and {0.3*\r});
\draw[dashed] ($(C1)+(-\r,0)$) arc (180:0:{\r} and {0.3*\r});

\draw plot [smooth, tension = 0.75] coordinates {($(C1)+(-\r,0)$)  ($(C1)+(-0.9*\r,1*\R)$) ($(C1)+(0.2*\r,1.5*\R)$) ($(C1)+(1.8*\r,1.15*\R)$) ($(C1)+(3*\r,1.5*\R)$) ($(C1)+(3.9*\r,1.2*\R)$) ($(C1)+(3.5*\r,0.5*\R)$) ($(C1)+(2.5*\r,0.25*\R)$) ($(C1)+(1.3*\r,0.3*\R)$) ($(C1)+(\r,0)$)};

\begin{scope}[scale=0.8,rotate=10,yshift=20,xshift=10]
\draw (-0.8,0.05) to[bend left=20] (0.8,0.05);
  \draw (-0.9,.1) to[bend right=25] (0.9,.1);
\end{scope}

\begin{scope}[rotate=20,yshift=-10,xshift=110]
\draw (-0.8,0.05) to[bend left=20] (0.8,0.05);
  \draw (-0.9,.1) to[bend right=25] (0.9,.1);
\end{scope}

\node at (1.4*\R,-0.6*\R) {$(\amb,g_T)$};

\end{scope}

\draw [decorate,decoration={brace,mirror,amplitude=10pt}] (-3,-4) -- (13,-4) node [black,midway,yshift=-20] {$(\amb',g')$};

\draw [decorate,decoration={brace,mirror,amplitude=10pt}] (-3,-6.5) -- (23,-6.5) node [black,midway,yshift=-20] {$(\amb'',g'')$};

\draw [decorate,decoration={brace,mirror,amplitude=10pt}] (-3,-9) -- (33,-9) node [black,midway,yshift=-20] {$(\amb''',g''')$};
\end{tikzpicture}
%%%%%%%%%%%%%%%%%%%%%%%%%%%%%%%%%%%%%%%%%%%%%%%%%%%%%
%%%%%%%%%%%%%%%%%%%%%%%%%%%%%%%%%%%%%%%%%%%%%%%%%%%%%
\caption{Scheme of the construction in the proof of \crefnoname{thm:Counterexample} for $\notg=2$ and $\notb=4$.} \label{fig:CountEx}
\end{figure}

Invoking Lemma 11 in \cite{ColDeL05}, we proceed as follows (see \crefnoname{fig:CountEx}).
We first attach $(M^i_1,g^i_1)$, near $y_1^i$, to $(M^{i+1}_1, g^{i+1}_1)$, near $x_1^{i+1}$, as $i$ varies from $1$ to $\notg-1$; let $(M',g')$ be the resulting manifold (of positive scalar curvature and empty boundary); let $y'\in M'$ be the point corresponding to $y^{\notg}_1\in M^{\notg}_1$, thus with a neighborhood that is foliated by great spheres.
Similarly we attach $(M',g')$, near $y'$, to $(M_0, g_0)$ near $x_0$; let $(M'', g'')$ be the resulting manifold (of positive scalar curvature and empty boundary) and let $y''\in M''$ be the point corresponding to $y_0\in M_0$, thus with a neighborhood that is foliated by great spheres.
Lastly, we attach $(M'', g'')$, near $y''$, to $(M^{\ebi}_2, g^{\ebi}_2)$, near $x^{\ebi}_2$; let $(M''', g''')$ be the resulting manifold (of positive scalar curvature and totally geodesic boundary).
The manifold $(M''',g''')$ is connected, has the topology of a ball, and it contains a sequence of free boundary minimal surfaces of genus $\notg$, exactly $\notb$ boundary components, that have unbounded area and Morse index (cf. \cref{rmk:index} above).

\vspace{2mm}\textbf{Step 2.}
Let $M$ be as in the statement: possibly applying Lemma C.1 in \cite{CarLi19} we can, and we shall, assume that this manifold comes endowed with a Riemannian metric of positive scalar curvature, and such that $\partial M$ is strictly mean convex. At that stage we know, by virtue of Theorem 5.7 in \cite{GroLaw80a}, that there exists a new metric $g_{T}$ on $M$ still having positive scalar curvature but totally geodesic boundary (in fact this manifold can be \emph{doubled} to a smooth Riemannian manifold $(M_D, g_D)$ without boundary).
Hence, we just observe that one can perform the Gromov-Lawson connected sum of $(M''',g''')$ and $(M,g_T)$ so to obtain a compact 3-manifold with positive scalar curvature and totally geodesic boundary. 

The combination of the two steps above allows to complete the proof of the first assertion of \cref{thm:Counterexample}. Instead, to obtain strictly positive mean curvature, it suffices to have the previous construction followed by the perturbation argument given in Lemma C.1 in \cite{CarLi19}. 
\end{proof}

Concerning the claim in \cref{rmk:ConvexBC}, it suffices to observe that when $\notb=1$ one modifies the block $(M^{\euno}_2, g^{\euno}_2)$ constructed in \crefnoname{subs:BlocksDisk} as follows: by the statement of \cref{lem:fbpiece}, the metric we obtained equals that of the cylinder $I\times S^2$ near the boundary sphere, thus we can just consider a smooth warping factor $f_W:\cc 0\eps \times S^2\to \mathbb{R}$, only depending on the first coordinate, that is monotone decreasing and equals $f$ on $\cc{\eps/2}{\eps}$. For \emph{any} such choice the boundary is convex (umbilic, with constant mean curvature); furthermore if the derivative of $f_W$ is small enough then the scalar curvature of the ambient manifold shall still be positive.

%%%%%%%%%%%%%%%%%%%%%%%%%%%%%%%%%%%%%%%%%%%%%%%%%%%%%%%%%%%%%%%%
%%% Appendices
%%%%%%%%%%%%%%%%%%%%%%%%%%%%%%%%%%%%%%%%%%%%%%%%%%%%%%%%%%%%%%%%

\appendix

\section{Free boundary minimal surfaces and Morse index} \label{sec:Properness}

Let $(\amb^3,g)$ be a compact Riemannian manifold with non-empty boundary. Given an embedded surface $\ms^2$ in $\amb$ with $\partial\ms\subset \partial\amb$, we wish to compare different definitions of free boundary minimality and Morse index when one allows for an (arbitrary) contact set of $\ms$ along $\partial\amb$, as in \cref{fig:NonProp}. The aim of this section is to analyze this matter, providing some flavour of the rather subtle nature of the question. To avoid ambiguities, throughout this appendix we will always use the expression \emph{free boundary minimal surface} to refer to a surface with zero mean curvature that meets the boundary of the ambient manifold orthogonally along its own boundary.
\begin{remark}
Observe that local minimizers of the area functional are not necessarily free boundary minimal surfaces with respect to this definition. See \cref{fig:MinNonMin} for an illustrative picture of a local minimizer which is not a free boundary minimal surface.
\end{remark}

\begin{figure}[htpb]
\centering
\begin{minipage}[t]{.6\textwidth}
\centering
%%%%%%%%%%%%%%%%%%%%%%%%%%%%%%%%%%%%%%%%%%%%%%%%%%%%%
%%% unproperness.tikz
%%%%%%%%%%%%%%%%%%%%%%%%%%%%%%%%%%%%%%%%%%%%%%%%%%%%%
\begin{tikzpicture} [scale =0.5]
\pgfmathsetmacro{\R}{4}
\pgfmathsetmacro{\ta}{10}
\pgfmathsetmacro{\xa}{\R*sin(\ta)}
\pgfmathsetmacro{\ya}{\R*cos(\ta)}
\pgfmathsetmacro{\tb}{40}
\pgfmathsetmacro{\xb}{\R*sin(\tb)}
\pgfmathsetmacro{\yb}{\R*cos(\tb)}
\pgfmathsetmacro{\g}{0.02}

\draw (-\xa,-\ya) arc(-90-\ta:-360-90+\ta:\R);
\draw plot [smooth, tension = 0.45] coordinates {(-\xa,-\ya) ({-3/5*\xa},{-8/9*\ya}) (0,-\g) ({3/5*\xa},{-8/9*\ya}) (\xa,-\ya)};
\draw[blue,thick] (-\R,0) -- (\R,0);

\node at (0.85*\R,0.8*\R) {$\amb_0$};
\node[blue] at ({-3/4*\R}, 0.5) {$\ms$};

\begin{scope}[xshift=300]
\draw (-\xb,-\yb) arc(-90-\tb:-360-90+\tb:\R);
\draw plot [smooth, tension = 0.5] coordinates {(-\xb,-\yb) ({-5/6*\xb},{-\yb}) ({-3/4*\xb},{-1/8*\yb}) ({-2/3*\xb}, -\g)};
\draw plot [smooth, tension = 0.5] coordinates {({2/3*\xb}, -\g) ({3/4*\xb},{-1/8*\yb}) ({5/6*\xb},{-\yb}) (\xb,-\yb)};
\draw[blue,thick] (-\R,0) -- (\R,0);
\draw ({-2/3*\xb}, -\g) -- ({2/3*\xb}, -\g);

\node at (0.85*\R,0.8*\R) {$\amb_{r_0}$};
\node[blue] at ({-3/4*\R}, 0.5) {$\ms$};
\end{scope}
\end{tikzpicture}
%%%%%%%%%%%%%%%%%%%%%%%%%%%%%%%%%%%%%%%%%%%%%%%%%%%%%
%%%%%%%%%%%%%%%%%%%%%%%%%%%%%%%%%%%%%%%%%%%%%%%%%%%%%
\caption{Modified unit ball with non-properly embedded free boundary minimal surface.} \label{fig:NonProp}
\end{minipage}\hfill
\begin{minipage}[t]{.4\textwidth}
\centering
%%%%%%%%%%%%%%%%%%%%%%%%%%%%%%%%%%%%%%%%%%%%%%%%%%%%%
%%% min-non-min.tikz
%%%%%%%%%%%%%%%%%%%%%%%%%%%%%%%%%%%%%%%%%%%%%%%%%%%%%
\begin{tikzpicture} [scale =0.5]
\pgfmathsetmacro{\R}{4}
\pgfmathsetmacro{\r}{1}
\pgfmathsetmacro{\t}{30}
\pgfmathsetmacro{\x}{\R*sin(\t)}
\pgfmathsetmacro{\y}{\R*cos(\t)}
\pgfmathsetmacro{\xa}{\R/4*8/3}
\pgfmathsetmacro{\ya}{\R/4*sqrt(80/9)}
\pgfmathsetmacro{\xb}{\xa*sqrt(5/64)}
\pgfmathsetmacro{\yb}{\ya*sqrt(5/64)}
\pgfmathsetmacro{\s}{asin(2/3)}
\pgfmathsetmacro{\g}{0.02}

\draw (-\x,-\y) arc(-90-\t:-360-90+\t:\R);

\draw (-\r,-1.5*\r) arc(180:0:\r);
\draw plot [smooth, tension = 0.7] coordinates {(-\x,-\y) ({-1.1*\r},{-\y}) ({-\r},{-1.5*\r})};
\draw plot [smooth, tension = 0.7] coordinates {(\x,-\y) ({1.1*\r},{-\y}) ({\r},{-1.5*\r})};

\draw[blue,thick] (-\xa,-\ya) -- (-\xb,-\yb);
\draw[blue,thick] (\xa,-\ya) -- (\xb,-\yb);
\draw[blue,thick] (-\xb,-\yb) arc(180-\s:\s:\r);

\node at (0.85*\R,0.8*\R) {$\amb$};
\node[blue] at ({-1/2*\R}, {-1/3*\R}) {$\ms$};
\end{tikzpicture}
%%%%%%%%%%%%%%%%%%%%%%%%%%%%%%%%%%%%%%%%%%%%%%%%%%%%%
%%%%%%%%%%%%%%%%%%%%%%%%%%%%%%%%%%%%%%%%%%%%%%%%%%%%%
\caption{A local minimizer of the area functional that is not a minimal surface.} 
\label{fig:MinNonMin}
\end{minipage}
\end{figure}

\begin{example}
A useful example to keep in mind is the following one.
Let us denote by $D$ the horizontal equatorial disc in the three-dimensional unit ball $B^3\subset\mathbb{R}^3$.
Given any closed subset $C\subset D$ with $C\cap\partial D = \emptyset$, consider a smooth compact ambient manifold $\amb_C$ obtained from $B^3$ by removing a portion of the lower half-ball in such a way that $\partial\amb_{C}$ intersects the interior of $D$ exactly in $C$. This can be done for every such choice of $C$.
Observe that $D$ is a non-properly embedded free boundary minimal surface in $\amb_C$. As a special instance, we consider $\amb_C$ for $C = \bar B_{r_0}(0)\subset \R^2$ where $0\le r_0<1$ and we denote it simply by $\amb_{r_0}$ (see \cref{fig:NonProp}).
\end{example}

Given a surface $\ms\subset\amb$, there are several possible families of variations along which we can deform $\ms$. Some natural choices are:
\begin{align}
\vf_e(\amb,\ms) &\eqdef \{ X\in\vf(\amb) \st X(x) \in T_x\partial\amb\ \forall x\in\partial\ms  \}; \label{eq:vfe} \\
\vf_i(\amb, \ms) &\eqdef \{ X\in\vf(\amb) \st X(x) \in T_x\partial\amb\ \forall x\in\partial\ms\comma \ g(X(x), \hat\eta(x))\leq 0 \ \forall x\in\partial\amb \}; \label{eq:vfi} \\ 
\vf_\partial(\amb) &\eqdef \{ X\in\vf(\amb) \st X(x) \in T_x\partial\amb\ \forall x\in\partial\amb \}; \label{eq:vfd}\\
\vf_c(\amb,\ms) &\eqdef \{ X\in \vf(\amb) \st X(x) \in T_x\partial\amb\ \forall x\in\partial\ms\comma \  \supp(X)\cap \ms \Subset \ms\setminus (\operatorname{int}(\ms)\cap \partial\amb) \}. \label{eq:vfc}
\end{align}
Recall that $\hat \eta$ is the outward unit co-normal to $\partial\amb$.
Moreover observe that 
\[
\vf_e(\amb,\ms)\supset \vf_i(\amb,\ms) \supset \vf_\partial(M) \supset \vf_c(\amb,\ms).
\]
Hereafter, we will write $\vf_*(\amb,\ms)$ to denote any of the previous subsets of $\vf(M)$.
\begin{figure}[htpb]
\centering
%%%%%%%%%%%%%%%%%%%%%%%%%%%%%%%%%%%%%%%%%%%%%%%%%%%%%
%%% properness_scheme
%%%%%%%%%%%%%%%%%%%%%%%%%%%%%%%%%%%%%%%%%%%%%%%%%%%%%
\begin{tikzpicture} [scale=0.41]
\pgfmathsetmacro{\R}{4}
\pgfmathsetmacro{\ta}{10}
\pgfmathsetmacro{\xa}{\R*sin(\ta)}
\pgfmathsetmacro{\ya}{\R*cos(\ta)}
\pgfmathsetmacro{\tb}{40}
\pgfmathsetmacro{\xb}{\R*sin(\tb)}
\pgfmathsetmacro{\yb}{\R*cos(\tb)}
\pgfmathsetmacro{\g}{0.02}
\pgfmathsetmacro{\h}{0.1*\R}

\begin{scope}[xshift=0]
\draw (-\xb,-\yb) arc(-90-\tb:-360-90+\tb:\R);
\draw plot [smooth, tension = 0.5] coordinates {(-\xb,-\yb) ({-5/6*\xb},{-\yb}) ({-3/4*\xb},{-1/8*\yb}) ({-2/3*\xb}, -\g)};
\draw plot [smooth, tension = 0.5] coordinates {({2/3*\xb}, -\g) ({3/4*\xb},{-1/8*\yb}) ({5/6*\xb},{-\yb}) (\xb,-\yb)};
\draw[blue,thick] (-\R,0) -- (\R,0);
\draw ({-2/3*\xb}, -\g) -- ({2/3*\xb}, -\g);

\def\up{{-0.9, -0.7, -0.5, -0.3,-0.1, 0.1, 0.3, 0.5, 0.7, 0.9}}
\foreach \i in {0,...,9}
  \draw[->] (\up[\i]*\R, 0) -- (\up[\i]*\R, \h);

\def\down{{-0.8, -0.6, -0.4, -0.2, 0, 0.2, 0.4, 0.6, 0.8}}
\foreach \i in {0,...,8}
  \draw[->] (\down[\i]*\R, 0) -- (\down[\i]*\R, -\h);
  
\node at (-0.9*\R, 0.9*\R) {\eqref{eq:vfe}};
\end{scope}

\begin{scope}[xshift=300] 
\draw (-\xb,-\yb) arc(-90-\tb:-360-90+\tb:\R);
\draw plot [smooth, tension = 0.5] coordinates {(-\xb,-\yb) ({-5/6*\xb},{-\yb}) ({-3/4*\xb},{-1/8*\yb}) ({-2/3*\xb}, -\g)};
\draw plot [smooth, tension = 0.5] coordinates {({2/3*\xb}, -\g) ({3/4*\xb},{-1/8*\yb}) ({5/6*\xb},{-\yb}) (\xb,-\yb)};
\draw[blue,thick] (-\R,0) -- (\R,0);
\draw ({-2/3*\xb}, -\g) -- ({2/3*\xb}, -\g);

\def\up{{-0.9, -0.7, -0.5, -0.3,-0.1, 0.1, 0.3, 0.5, 0.7, 0.9}}
\foreach \i in {0,...,9}
  \draw[->] (\up[\i]*\R, 0) -- (\up[\i]*\R, \h);

\def\down{{-0.8, -0.6, 0.6, 0.8}}
\foreach \i in {0,...,3}
  \draw[->] (\down[\i]*\R, 0) -- (\down[\i]*\R, -\h);
  
\node at (-0.9*\R, 0.9*\R) {\eqref{eq:vfi}};
\end{scope}

\begin{scope}[xshift=600]
\draw (-\xb,-\yb) arc(-90-\tb:-360-90+\tb:\R);
\draw plot [smooth, tension = 0.5] coordinates {(-\xb,-\yb) ({-5/6*\xb},{-\yb}) ({-3/4*\xb},{-1/8*\yb}) ({-2/3*\xb}, -\g)};
\draw plot [smooth, tension = 0.5] coordinates {({2/3*\xb}, -\g) ({3/4*\xb},{-1/8*\yb}) ({5/6*\xb},{-\yb}) (\xb,-\yb)};
\draw[blue,thick] (-\R,0) -- (\R,0);
\draw ({-2/3*\xb}, -\g) -- ({2/3*\xb}, -\g);

\def\up{{-0.9, -0.7, -0.5, 0.5, 0.7, 0.9}}
\foreach \i in {0,...,5}
  \draw[->] (\up[\i]*\R, 0) -- (\up[\i]*\R, \h);

\def\down{{-0.8, -0.6, 0.6, 0.8}}
\foreach \i in {0,...,3}
  \draw[->] (\down[\i]*\R, 0) -- (\down[\i]*\R, -\h);

\node at (-0.9*\R, 0.9*\R) {\eqref{eq:vfd}};
\end{scope}

\begin{scope}[xshift=900]
\draw (-\xb,-\yb) arc(-90-\tb:-360-90+\tb:\R);
\draw plot [smooth, tension = 0.5] coordinates {(-\xb,-\yb) ({-5/6*\xb},{-\yb}) ({-3/4*\xb},{-1/8*\yb}) ({-2/3*\xb}, -\g)};
\draw plot [smooth, tension = 0.5] coordinates {({2/3*\xb}, -\g) ({3/4*\xb},{-1/8*\yb}) ({5/6*\xb},{-\yb}) (\xb,-\yb)};
\draw[blue,thick] (-\R,0) -- (\R,0);
\draw ({-2/3*\xb}, -\g) -- ({2/3*\xb}, -\g);

\def\up{{-0.9, -0.7, 0.7, 0.9}}
\foreach \i in {0,...,3}
  \draw[->] (\up[\i]*\R, 0) -- (\up[\i]*\R, \h);

\def\down{{-0.8, -0.6, 0.6, 0.8}}
\foreach \i in {0,...,3}
  \draw[->] (\down[\i]*\R, 0) -- (\down[\i]*\R, -\h);
  
\node at (-0.9*\R, 0.9*\R) {\eqref{eq:vfc}};
\end{scope}
\end{tikzpicture}
%%%%%%%%%%%%%%%%%%%%%%%%%%%%%%%%%%%%%%%%%%%%%%%%%%%%%
%%%%%%%%%%%%%%%%%%%%%%%%%%%%%%%%%%%%%%%%%%%%%%%%%%%%%
\caption{Visualization of the different possible sets of variations.} \label{fig:PropScheme}
\end{figure}
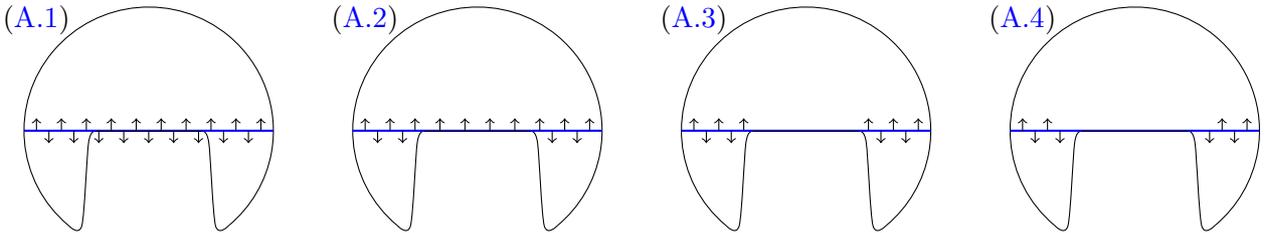

\begin{proposition} \label{prop:FirstSecMakeSense}
Let $\amb$, $\ms$ be as above and let $X\in \vf_e(\amb,\ms)$.
Consider $\check \amb$ to be a compact manifold in which $\amb$ embeds as a regular domain and such that $\ms$ is properly embedded in $\check \amb$, namely $\partial\ms = \ms\cap \partial\check\amb$. Moreover, let $\check X\in \vf(\check\amb)$ be a smooth extension of the vector field $X$ to all $\check \amb$. Then the first variation of the area of $\ms$ with respect to $\check X$ does not depend on the choice of the extensions $\check\amb$ and $\check X$, and is given by \eqref{eq:FirstVar}. 
Furthermore, if $\ms$ is a free boundary minimal surface, then we can compute the second variation of the area with respect to $\check X$ and the result again does not depend on the choice of the extensions and is given by \eqref{eq:SecondVar}.
\end{proposition}
\begin{proof}
The result follows by simply inspecting equations \eqref{eq:FirstVar} and \eqref{eq:SecondVar} for the first and second variation, respectively, applied to $\check X$.
Indeed, the only possible term in the formulae that is not obviously independent of the extensions is $\abs{\cov^\perp \check X^\perp}^2$. 
However, keeping in mind that the extensions $\check\amb$ and $\check X$ are smooth, the covariant derivatives of $\check X^{\perp}$ are uniquely determined by their values along $\ms$.
\end{proof}

\begin{remark}
Given \cref{prop:FirstSecMakeSense}, and thanks to \eqref{eq:FirstVar}, we can say that:
\begin{itemize}
\item If we consider as variations $\vf_e(\amb,\ms)$ or $\vf_i(\amb,\ms)$, the critical points are exactly all and only the free boundary minimal surfaces.
\item In the other two cases $\vf_\partial(\amb)$ and $\vf_c(\amb,\ms)$, free boundary minimal surfaces are critical points but the other implication is not true.
Indeed, being a critical point with respect to these variations does not impose \emph{any} condition on $\ms\cap (\partial\amb\setminus\partial\ms)$. In particular, surfaces as in \cref{fig:MinNonMin} are critical points.
\end{itemize}
\end{remark}

\subsection{Morse index - four definitions}
Given a free boundary minimal surface $\ms$, we can now try to define the Morse index of $\ms$ with respect to variations in $\vf_*(\amb,\ms)$.

In the cases when $\vf_*(\amb,\ms)$ is a vector space, we can just mimic the classical definition of Morse index. 
Namely, denoting by $\vf_*^R(\amb,\ms)\subset \Gamma(T\ms)$ the set of restrictions of elements of $\vf_*(\amb,\ms)$ to $\Gamma(T\ms)$, we define the Morse index $\ind_*(\ms)$ with respect to the variations in $\vf_*(\amb,\ms)$ as the maximal dimension of a linear subspace of $\Gamma(N\ms)\cap \vf_*^R(\amb,\ms)$ where the second variation of the area (given by \eqref{eq:SecondVar}) is negative definite.
In this way we define $\ind_e(\ms)$, $\ind_\partial(\ms)$ and $\ind_c(\ms)$.
\begin{remark}
This definition of Morse index $\ind_c(\ms)$ coincides with the one given by Guang-Wang-Zhou in \cite{GuaWanZho19} and employed in \cite{GuaLiWanZho19}.
\end{remark}

Observe that $\vf_i^R(\amb,\ms)$ is only a convex cone in $\vf(M)$, thus the conceptual scheme above is not immediately applicable. So, let us recall some terminology and employ it to suggest a natural candidate for the definition of $\ind_i(\ms)$.
\begin{definition}
Given a surface $\ms\subset\amb$, we say that a set $\Theta\subset\Gamma(N\ms)\cap \vf_i^R(\amb,\ms)$ is a \emph{convex subcone} if it is closed under linear combinations with non-negative coefficients.
The \emph{dimension} of a convex subcone $\Theta$ is the minimal dimension of a linear subspace of $\Gamma(N\ms)$ that contains $\Theta$.
\end{definition}

\begin{definition}
Given an embedded free boundary minimal surface $\ms\subset\amb$, we define $\ind_i(\ms)$ as the maximal dimension of a convex subcone of $\Gamma(N\ms)\cap \vf_i^R(\amb,\ms)$ where the second variation of the area functional is negative definite.
\end{definition}

Unfortunately, this definition is not always meaningful as shown in the following proposition.
\begin{proposition}
Let $\ms = D$ be the equatorial disc in the ambient manifold $\amb = \amb_{r_0}$ defined above, for some $0<r_0<1$.
Then $\ms$ is a free boundary minimal surface in $\amb$ and $\ind_i(\ms)=\infty$.
\end{proposition}
\begin{proof}
First of all, observe that $\Gamma(N\ms)\cap \vf_i^R(\amb,\ms)$ can be identified with the set of functions $\{f\in C^\infty(\ms) \st \text{$f\ge 0$ on $\bar B_{r_0}(0)$} \}$ and that the second variation of the area along any function $f$ in this set can be written as
\begin{equation} \label{eq:SecVarDisk}
Q_\ms(f,f) = \int_\ms\abs{\grad f}^2 \de \Haus^2 - \int_{\partial\ms} f^2\de\Haus^1 \point
\end{equation}

Now, given any $N\in\N$, let us consider functions $\rho_1,\ldots,\rho_N \in C^\infty(\ms)$, $0\le \rho_k\le1$ for $k=1,\ldots,N$, such that
\begin{enumerate} [label=(\roman*)]
\item $\supp(\rho_k)\subset B_{r_0}(0)\subset \ms$ for $k=1,\ldots,N$; \label{bf:supp}
\item $\supp(\rho_k)\cap \supp(\rho_h) = \emptyset$ for any $1\le k<h \le N$; \label{bf:disj}
\item $\int_\ms\abs{\grad\rho_k}^2\de\Haus^2<2\pi/N$ for $k=1,\ldots,N$; \label{bf:norm}
\item there exist $x_1,\ldots,x_N\in\ms$ such that $\rho_k(x_k) = 1$ for $k=1,\ldots,N$. \label{bf:value}
\end{enumerate}

Then define $\psi_k\eqdef 1-\sum_{h\not=k} \rho_h$ for $k=1,\ldots,N$ and denote by $\Theta\subset \Gamma(N\ms)\cap \vf_i^R(\amb,\ms)$ the convex subcone of dimension $N$ generated by $\{\psi_k\}_{k=1,\ldots,N}$ (observe that $\psi_k\ge 0$ in $\ms$). We want to prove that $Q_\ms$ is negative definite on $\Theta$, which would conclude the proof since $\ind_i(\ms)\ge \dim\Theta = N$.

Pick any $\psi\eqdef \sum_{k=1}^Na_k\psi_k \in \Theta$: using \ref{bf:disj} and \ref{bf:norm}, together with the fact that $\supp(\rho_k)\cap\partial\ms=\emptyset$ for $k=1,\ldots,N$ by \ref{bf:supp}, we have that
\[
\begin{split}
Q_\ms(\psi,\psi) &= \int_\ms \sum_{k=1}^N\abs{\grad\rho_k}^2\Bigg(\sum_{h\not=k}a_h\Bigg)^2 \de\Haus^2 - \int_{\partial\ms} \Bigg(\sum_{k=1}^N a_k \Bigg)^2\de\Haus^1\\
&\le \left(\sum_{k=1}^Na_k \right)^2\left( \int_\ms \sum_{k=1}^N\abs{\grad\rho_k}^2 \de\Haus^2 -2\pi\right) < 0 \comma
\end{split}
\]
which concludes the proof.
\end{proof}

\subsection{Morse index - basic comparison results} We now present the simplest inequality involving $\ind_e(\ms)$, $\ind_\partial(\ms)$, $\ind_c(\ms)$ and compute these three quantities in an explicit example.

\begin{proposition} \label{prop:IneqIndices}
Given an embedded free boundary minimal surface $\ms\subset\amb$, we have that $\ind_e(\ms)$, $\ind_\partial(\ms)$ and $\ind_c(\ms)$ are well-defined (finite) numbers and it holds
\[
\ind_e(\ms) \ge \ind_\partial(\ms) = \ind_c(\ms)\point
\]
\end{proposition}
\begin{proof}
First observe that, by \cref{prop:FirstSecMakeSense}, $\ind_e(\ms)$ coincides with the Morse index of $\ms$ seen as a properly embedded free boundary minimal surface in any extension $\check\amb$ of $\amb$ as in the statement of the proposition. Therefore $\ind_e(\ms)$ is a well-defined number.
Moreover, since $\vf_c(\amb,\ms) \subset \vf_\partial(\amb) \subset \vf_e(\amb,\ms)$, we have that 
$\ind_c(\ms)$ and $\ind_\partial(\ms)$ are well-defined as well, and $\ind_c(\ms)\le \ind_\partial(\ms)\le \ind_e(\ms)$.

We only need to prove that $\ind_c(\ms)$ is actually equal to $\ind_\partial(\ms)$. 
We first show that $\vf_c^R(\amb,\ms)$ is dense in $\vf_\partial^R(\amb)$ with respect to the $H^1$ norm, which is a consequence of the following lemma thanks to a standard partition of unity argument on $\Omega\eqdef \ms\setminus (\operatorname{int}(\ms)\cap \partial\amb)$ (in particular around $\partial\Omega\setminus\partial\ms$).

\begin{lemma}
Let $\Omega\subset \R^n$ be a bounded open domain and consider $u\in \operatorname{Lip}(\R^n)$ such that $u=0$ in $\Omega^{c}$. Then there exists a sequence of functions $u_k\in C^\infty_c(\Omega)$ that converge to $u$ in the $H^1(\R^n)$ norm, i.e.
\[
\int_{\R^n}\abs{u-u_k}^2 + \abs{\grad u-\grad {u_k}}^2 \de x  \to 0 \point 
\]
\end{lemma}
\begin{proof}
Let $u_+,u_-\in \operatorname{Lip}(\R^n)$ be the positive and the negative part of $u$, respectively. Moreover, let us define $\rho \eqdef d(\cdot,\Omega^c)$, for which we have that $\rho\in \operatorname{Lip}(\R^n)$, $\rho = 0$ in $\Omega^c$ and $\rho>0$ in $\Omega$. The functions $u_1\eqdef \rho + u_+, u_2\eqdef \rho + u_- \in \operatorname{Lip}(\R^n)$ are zero in $\Omega^c$ and are strictly positive in $\Omega$. Note that it is sufficient to prove the result for these two functions.

Therefore, without loss of generality, let us assume that $u\ge d(\cdot,\Omega^c)>0$ in $\Omega$. Moreover, to simplify the notation, let us assume that the Lipschitz constant of $u$ is $1$. Then, let us consider $u_\eps\eqdef (u-\eps)_+ \in \operatorname{Lip}(\R^n)$. Observe that $u_\eps(x) = 0$ for every $x\in\Omega$ such that $d(x,\Omega^c) \le \eps$.
Furthermore, it holds that (see e.g. \cite[Theorem 3.38]{AmbCarMas18})
\[
\int_{\overline\Omega} \abs{u-u_\eps}^2 \de x +  \int_{\overline\Omega} \abs{\grad u-\grad u_\eps}^2 \de x \le \eps^2\abs{\Omega} + \int_{\{u< \eps\}} \abs{\grad u}^2 \de x\comma
\]
which converges to $0$ as $\eps\to 0$ since $u\in H^1(\R^n)$ and $\bigcap_{\eps>0}\{u<\eps\} = \{u=0\}$.
Thus, let us choose a sequence $\eps_k\to 0$ and define $u_k\eqdef u_{\eps_k} \ast \varphi_{\eps_k/2}$, where $\varphi_{\eps}(x) \eqdef \eps^{-1}\varphi(x/\eps)$ and $\varphi\in C^\infty_c(B_1(0))$ is a cutoff function with $\varphi(0)=1$, $\varphi\le 1$.
The conclusion is then straightforward.
\end{proof}

Now, let $\Delta$ be an $(\ind_\partial(\ms))$-dimensional vector subspace of $\Gamma(N\ms)\cap \vf_\partial^R(\amb)$ where $Q_\ms$ is negative definite.
By density of $\vf_c^R(\amb,\ms)$ in $\vf_\partial^R(\amb)$, we can find an $(\ind_\partial(\ms))$-dimensional vector subspace $\tilde\Delta\subset \Gamma(N\ms)\cap \vf_c^R(\amb,\ms)$ such that $\tilde\Delta\cap \{ V\in \Gamma(N\ms)\st \norm V_{H^1} = 1\}$ is as close as we want to $\Delta\cap \{ V\in \Gamma(N\ms)\st \norm V_{H^1} = 1\}$ in the $H^1(\bar\ms)$ norm. In particular, we can choose $\tilde\Delta$ such that $Q_\ms$ is negative definite in $\tilde\Delta\cap \{ V\in \Gamma(N\ms)\st \norm V_{H^1} = 1\}$, and so in $\tilde\Delta$.
Therefore, we have that $\ind_\partial(\ms) = \dim\Delta=\dim\tilde\Delta \le \ind_c(\ms)$, which concludes the proof.
\end{proof}

We conclude this section by showing that the inequality above is actually strict in cases of interest.

\begin{proposition}
Let $\ms = D$ be the equatorial disc in the ambient manifold $M=M_{r_0}$ with $0\le r_0 <1$, defined above. Then $\ms$ is a free boundary minimal surface in $M_{r_0}$ with $\ind_e(\ms) = 1$ for all $0\le r_0 < 1$ and
\[
\ind_\partial(\ms) = \ind_c(\ms) = \begin{cases}
                                   1 & \text{if $r_0<e^{-1}$}\comma\\
                                   0 & \text{if $r_0\ge e^{-1}$} \point
                                   \end{cases}
\]
\end{proposition}
Informally, this means that the free boundary minimal surfaces in \cref{fig:NonProp} are respectively unstable (the one on the left) and stable (the one on the right) with respect to the variations $\vf_\partial(\amb)$ and $\vf_c(\amb,\ms)$.

\begin{proof}
First observe that $\ind_e(\ms)=1$, since it coincides with the index of the equatorial disc in $B^3$. Therefore, we only need to compute $\ind_\partial(\ms)= \ind_c(\ms)$, which are equal by \cref{prop:IneqIndices}.

Since $\ind_\partial(\ms) \le \ind_e(\ms)=1$, it is sufficient to determine whether $\ind_\partial(\ms)$ is $0$ or $1$, i.e. whether there exists $f\in C^\infty(\ms)$ such that $f =0$ on $\bar B_{r_0}(0)$ and $Q_\ms(f,f) < 0$.
For any such function $f\in C^\infty(\ms)$ it holds that
\begin{align*}
Q_\ms(f,f) &= \int_0^{2\pi} \int_{r_0}^1 (\abs{\partial_r f(r,\theta)}^2 + r^{-2} \abs{\partial_\theta f(r,\theta)}^2)r \de r \de \theta - \int_0^{2\pi} f(1,\theta)^2\de\theta\\
&\ge \int_0^{2\pi} \int_{r_0}^1 \abs{\partial_r f(r,\theta)}^2 r \de r \de \theta - \int_0^{2\pi} f(1,\theta)^2\de\theta= \int_0^{2\pi}\left( \int_{r_0}^1 \abs{\partial_r f(r,\theta)}^2 r \de r - f(1,\theta)^2\right)\de\theta \point
\end{align*}
However, observe that by the Cauchy-Schwartz inequality
\[
\abs{f(1,\theta)} = \abs{f(1,\theta) - f(r_0,\theta)} = \left \lvert \int_{r_0}^1 \partial_r f(r,\theta) \de r \right\rvert \le \left(\int_{r_0}^1 \frac 1r\de r \right)^{1/2} \left( \int_{r_0}^1 \abs{\partial_r f(r,\theta)}^2 r\de r\right)^{1/2}\comma
\]
thus
\[
f(1,\theta)^2 \le \ln(r_0^{-1}) \int_{r_0}^1 \abs{\partial_r f(r,\theta)}^2 r\de r \comma
\]
which implies that
\[
Q_\ms(f,f) \ge \left(\frac1{\ln(r_0^{-1})} - 1 \right) \int_0^{2\pi} f(1,\theta)^2\de\theta 
\]
with equality if and only if $f = f_{r_0}$, where $f_{r_0}(r,\theta)\eqdef c (\ln(r) - \ln (r_0))$ in $B_1(0)\setminus B_{r_0}(0)$ for some constant $c\in\R$ and $f_{r_0}(r,\theta) \eqdef 0$ in $\bar B_{r_0}(0)$.
Therefore, the index is $1$ if and only if $\ln(r_0^{-1}) >  1$ (observe that, rigorously, a smoothing of the function $f_{r_0}$ is needed), that means $r_0<e^{-1}$ as we wanted.
\end{proof}

\section{Reflecting free boundary minimal surfaces in a half-space} \label{sec:ReflPrinc}

We start recalling a standard reflection lemma, which is useful to transfer information about minimal surfaces in $\R^3$ to free boundary minimal surfaces in a half-space. The proof consists in a well-known argument based on elliptic estimates.
\begin{lemma} \label{lem:ReflectionPrinciple}
If $\ms^2$ is an embedded free boundary minimal surface in $\Xi\subset \R^3$ (that is $\ms$ has zero mean curvature and meets $\Pi$ orthogonally along $\partial\ms$), 
then the union of $\ms$ with its reflection with respect to $\Pi$ is a smooth, embedded minimal surface in $\R^3$ without boundary.
\end{lemma}

Here is instead a regularity result for free boundary minimal immersions in a half-space of $\R^3$.

\begin{proposition} \label{cor:MinSurfInR3WithFiniteIndex}
Let $\ms\subset \Xi\subset\R^3$ be a complete free boundary minimal injective immersion with finite index and with $\partial\ms = \Pi\cap \ms$. Then $\ms$ is two-sided, has finite total 
curvature and is properly embedded.
\end{proposition}
\begin{proof}
This is a consequence of Theorem B.1 in \cite{ChoKetMax17}, after applying the reflection principle \cref{lem:ReflectionPrinciple} and the argument of \cite[Section 2.2]{AmbBuzCarSha18}, which implies that the reflected surface has finite index.
\end{proof}

We further apply \cite[Section 2.3]{AmbBuzCarSha18} together with \cite{ChoMax16} to obtain topological information from index bounds for free boundary minimal surfaces in a half-space of $\R^3$.

\begin{definition}
Given a complete, connected, properly embedded, free boundary minimal surface $\ms^2$ in $\Xi\subset \R^3$, the number of ends of $\ms$ is the number of connected components of $\ms$ outside any sufficiently large compact set. We will denote this number by $\mathrm{ends}(\ms)$.
\end{definition}
Observe that the number of ends of a properly embedded free boundary minimal surface in $\Xi\subset \R^3$ as above is indeed well-defined (see \cite[Remark 26]{AmbBuzCarSha18}).

\begin{proposition} \label{prop:GeometryOfHalfBubble}
Given $I\ge 0$, there exists $\kappa(I)\ge 0$ such that every complete, connected, properly embedded, free boundary minimal surface  $\ms^2$ in $\Xi\subset \R^3$ of index at most $I$ 
has genus, number of ends and number of boundary components (in $\Pi$) all bounded by $\kappa(I)$.
\end{proposition}
\begin{proof}
Let $\check\ms$ be the union of $\ms$ with its reflection with respect to $\Pi$, as above.
Then, thanks to \cite[Section 2.3]{AmbBuzCarSha18} (in particular equation (2.5) therein), $\check\ms$ is a complete, connected, properly embedded, minimal surface of $\R^3$ with index less or equal than $2I$.
Hence, using the main estimate in \cite{ChoMax16}, we obtain that
\[
\frac 23 (\genus(\check\ms) + \mathrm{ends}(\check\ms)) -1 \le 2I\comma
\]
where $\mathrm{ends}(\check\ms)$ denotes the number of ends of $\check\ms$.

Moreover, by Lemma 29 in \cite{AmbBuzCarSha18}, it holds that $\chi(\check\ms)-\mathrm{ends}(\check\ms) = 2(\chi(\ms)-\mathrm{ends}(\ms))$. Therefore, we obtain
\[
\chi(\ms)-\mathrm{ends}(\ms) = \frac 12(\chi(\check\ms)-\mathrm{ends}(\check\ms)) = 1 - \genus(\check\ms) - \mathrm{ends}(\check\ms) \ge 1 -\frac 32(2I+1) = -3I-\frac 12 \point
\]
Observe that $\chi(\ms) = 2-2\genus(\ms) - \xi(\ms)$, where $\xi(\ms)$ is the number of boundary components of $\ms\cap B^{\R^3}_R(0)$, for any $R>0$ sufficiently large. Thus we get
\[
2\genus(\ms) + \xi(\ms) +\mathrm{ends}(\ms) \le 3I+\frac 52\comma
\]
from which it follows directly that $\genus(\ms)$ and $\mathrm{ends}(\ms)$ are both bounded by $3I+5/2$.

Finally, note that $\bdry(\ms)\le \xi(\ms)+\mathrm{ends}(\ms)$ and so
\[
3I+\frac 52 \ge 2\genus(\ms) + \xi(\ms) +\mathrm{ends}(\ms) \ge \xi(\ms) +\mathrm{ends}(\ms) \ge \bdry(\ms)\comma
\]
which concludes the proof once we choose $\kappa(I) = 3I +5/2$. 
\end{proof}

\section{Some Morse-theoretic arguments} \label{sec:MorseTheory}

In this section we collect a few lemmata that will be useful to obtain topological information at intermediate scales. The basic idea is that if a surface fulfills suitable curvature estimates then it is `locally simple'.

\begin{lemma} \label{lem:GradEst}
Let $M^3 = \{\abs x\le 2\}\cap \Xi(a)\subset \R^3$ be endowed with the Euclidean metric. Fix $p\in B_{1/4}(0)\cap M$ and $r>0$.
Let $\ms^2\subset \amb\setminus B_r(p)$ be a connected embedded surface, having free boundary with respect to $\Pi(a)$, with $\partial\ms = \ms\cap \partial (\amb\setminus B_r(p))$. Assume that for every $x\in\ms$ it holds $\abs{\A_\ms}(x) \abs{x-p} \le \eps$, for some constant $0<\eps\le 1/3$. Moreover suppose that $\ms$ intersects $\partial B_r(p)$.
Then, denoting by $f:\ms\to\R$ the function $f(x)\eqdef \abs{x-p}^2$, we have that
\[
\abs{\grad_\ms f(x)} \ge 2(1-\eps) (\abs{x-p} - r) \semicolon  
\]
furthermore, if $\ms\cap\partial B_r(p)$ contains a compact component $\Gamma$, then in fact
\[
\abs{\grad_\ms f(x)} \ge 2(1-\eps) (\abs{x-p} - r) + \min_{y\in \Gamma}\ \abs{\grad_\ms f(y)}\point
\]
\end{lemma}
\begin{proof}
Consider a point $x\in\ms\setminus\partial\amb$ and take a unit-speed geodesic $\gamma:\cc 0l\to \ms$ such that $\gamma(0)\in \ms\cap \partial B_r(p)$ and $\gamma(l)=x$.
Note that $\gamma$ exists since $\ms$ is connected and a geodesic starting inside $\ms$ cannot touch $\ms\cap\partial\amb$ tangentially. Indeed $\ms\cap\Pi(a)$ is a union of geodesics in $\ms$ thanks to the free boundary condition and a simple computation shows that $\ms\cap(\partial\amb\setminus\Pi(a))$ is strictly convex in $\ms$ by the curvature estimate.

Then observe that $\grad_\ms f(x)$ is the projection of the vector $2(x-p)$ on $\ms$ for all $x\in\ms$.
In addition, for every $v\in T_x\ms$, it holds
\[
(\mathrm{Hess}_\ms f)_x(v,v)  = 2(\abs v ^2 - \scal{A_\ms(x)(v,v)}{x-p}) \ge 2\abs v ^2 (1 - \abs{A_\ms}(x)\abs{x-p}) \ge  2\abs v^2 (1-\eps) \point
\]
Thus we have that
\[
\frac{\d}{\d t} (f\circ \gamma) = \scal{\grad_\ms f(\gamma)}{\gamma'} \quad\text{and} \quad \frac{\d^2}{\d t^2}(f\circ\gamma) = (\mathrm{Hess}_\ms f)_\gamma (\gamma',\gamma') \ge 2(1-\eps) \point
\]
Therefore, since clearly $\scal{\grad_\ms f(\gamma(0))}{ \gamma'(0)} \ge 0$, we obtain that
\[
\begin{split}
\abs{\grad_\ms f(x)} &\ge \scal{\grad_\ms f(x)}{\gamma'(l)} = \frac{\d}{\d t}\Big|_{t=l} (f\circ \gamma) \ge 2(1-\eps)l \ge  2(1-\eps)(\abs{x-p}-r)\point
\end{split}
\]

Now assume that $\ms\cap\partial B_r(p)$ contains a compact component $\Gamma$. Take a point $x\in\ms$ and consider a \emph{length minimizing} unit-speed \emph{curve} $\gamma:\cc 0l \to\ms$ with $\gamma(0) \in \Gamma$ and $\gamma(l) =x$. 
Note that $\gamma$ cannot touch $\ms\cap \partial B_r(p)$ at time larger than zero. Indeed, if that happened, there would exist an intermediate time $t_0>0$ such that $f\circ \gamma$ has a (local) maximum at $t_0$ and $\gamma$ is a geodesic between $0$ and $t_0$. However this fact leads to a contradiction with the previous computation since we would have at the same time $(f\circ\gamma)'(t_0) =0$ and $(f\circ\gamma)'(t_0) \ge 2(1-\eps)t_0>0$. Therefore $\gamma$ is a geodesic, since it cannot touch $\partial\ms$ in its interior.

Performing the same computation as above, this time using that $\scal{\grad_\ms f(\gamma(0))}{\gamma'(0)} = \abs{\grad_\ms f(\gamma(0))}$ (by minimality of $\gamma$), we thus obtain
\[
\abs{\grad_\ms f(x)} \ge 2(1-\eps)(\abs{x-p}-r) + \abs{\grad_\ms f(\gamma(0))} \ge 2(1-\eps)(\abs{x-p}-r) + \min_{y\in\Gamma}\ \abs{\grad_\ms f(y)}\comma
\]
which concludes the proof.
\end{proof}

\begin{corollary} \label{cor:14PuncturedBall}
Let $g$ be a metric on $\{\abs x\le 2\}\cap \Xi(0)\subset\R^3$ sufficiently close to the Euclidean one, and denote by $M^3\eqdef B_1(0)\subset \Xi(0)$ the unit ball with respect to this metric\footnote{\label{fn:MetricBalls}That is to say: here $B_1(0)$ is the metric ball with respect to the metric $g$, and the same comment applies to the balls $B_1(0), B_r(p)$ in \cref{cor:14TopInfo}.}.
Assume that being orthogonal to $\Pi(0)$ with respect to the Euclidean metric is equivalent\footnote{\label{fn:FermiChart}In the applications we obtain this orthogonality condition working in a Fermi chart, both here and in \cref{cor:14TopInfo} below.} to being orthogonal to $\Pi(0)$ with respect to $g$.
Given $0<r<1/4$, consider a connected, embedded surface $\ms^2\subset \amb\setminus B_r^{\R^3}(0)$ with $\partial\ms=\ms\cap\partial(\amb\setminus B_r^{\R^3}(0))$ such that
\begin{enumerate} [label={\normalfont(\roman*)}]
\item \label{pb:FB} $\ms$ is free boundary with respect to $\Pi(0)$;
\item \label{pb:cptcomp} $\ms\cap \partial B_r^{\R^3}(0)$ contains a compact component $\Gamma$ and $\ms$ intersects $\partial B_r^{\R^3}(0)$ transversely along $\Gamma$; 
\item \label{pb:curv} for every $x\in\ms$ it holds $\abs{\A_\ms}(x) d_g(x,0) \le 1/4$.
\end{enumerate}
Then $\ms$ is properly embedded in $B_1(0)\setminus B_r^{\R^3}(0)$ and is either a topological disc or a topological annulus.
\end{corollary}
\begin{proof}
Taking $g$ sufficiently close to the Euclidean metric we can assume that $\abs{A_\ms^{\R^3}}(x)\abs{x}\le 1/3$.
Let $f:\ms\to \R$ be given by $f(x) \eqdef \abs{x}^2$. Then, thanks to \cref{lem:GradEst} and \ref{pb:cptcomp}, $f$ has no critical points on $\ms$; in particular it holds
\begin{equation}\label{eq:GradAwayFromZero}
\abs{\grad_\ms^{\R^3}f(x)} \ge \min_{y\in\Gamma}\ \abs{\grad_\ms^{\R^3}f(y)} > 0 \quad \text{for all $x\in\ms$}\point
\end{equation}
Moreover we obtain that $\abs{\grad_\ms^{\R^3}f(x)} \ge 1/2$ for all $x\in\ms\cap (\partial B_1(0)\setminus \Pi(0))$;
hence, taking $g$ possibly closer to the Euclidean metric (independently of $\ms$), we can assume that $\grad_\ms^{\R^3}f(x)$ points strictly out of $\ms$ along $\partial B_1(0)\setminus \Pi(0)$.
Observe also that $\grad_\ms^{\R^3}f$ is parallel to $\Pi(0)$ along $\ms\cap\Pi(0)$ by \ref{pb:FB}.

Thus, defining $\Phi(t,x)$ as the flow of the vector field $\grad_\ms^{\R^3}f$ on $\ms$,
we have that the exit-time map $\ms\to \co 0\infty$ given by $x\mapsto t(x)$ (i.e. $t(x)$ is the supremum of the times $t$ for which $\Phi(t,x)$ is well-defined in $\ms$) is continuous. Furthermore $0 < t(y) <\infty$ for every $y\in\Gamma$ by \eqref{eq:GradAwayFromZero}.
Hence the map $\mathcal F:\Gamma\times\cc 01 \to \ms$ given by $\mathcal F(y,t) \eqdef \Phi(t(y)t,y)$ is a homeomorphism with its image, which must coincide with all of $\ms$ by connectedness.
This concludes the proof, showing that $\ms$ is a properly embedded topological disc if $\Gamma$ is an arc and a properly embedded topological annulus if $\Gamma$ is a circle.
\end{proof}
\begin{remark} \label{rem:14PunctBallVariant}
A variation of the previous proof works in the case when $\ms^2\subset \amb \setminus\{0\}$ is a connected embedded surface satisfying assumptions \ref{pb:FB}, \ref{pb:curv} and such that $\ms\cap (\partial B_1(0)\setminus\Pi(0))$ contains a compact component $\Gamma$. What we obtain is that $\ms$ is properly embedded in $B_1(0)\setminus\{0\}$ and is either a punctured disc (the puncture being in the interior) or a punctured half-disc (the puncture being on the boundary).
Observe that, in this case, a combination of the first estimate of \cref{lem:GradEst} and a flow from $\Gamma$ along $-\grad_\ms^{\R^3}f$ is needed.
\end{remark}

We shall further need the following refinement of the previous statement in order to transfer topological information from a small to a bigger scale. Observe that in the free boundary case a standard Morse-theoretic argument as in \cite[Lemma 3.1]{ChoKetMax17} is not sufficient. Indeed, given (for example) a surface $\ms\subset B_1^{\R^3}(0)\cap\Xi(0)\subset\R^3$, the number of connected components of $\ms\cap\partial B_1^{\R^3}(0)$ can be arbitrarily large even if $\ms\cap (\partial B_1^{\R^3}(0)\cup \Pi(0))$ consists of a single connected component.

\begin{definition} \label{def:StrongEq}
Let $g$ be a metric on $\amb\eqdef \{\abs x\le 2\}\cap \Xi(a)\subset\R^3$ sufficiently close to the Euclidean one, for some $0\ge a\ge-\infty$, and let $B_r(p)\subset \amb$ be the metric ball with respect to $g$ with center $p\in\amb$ and radius $r>0$. Given a surface $\ms\subset \amb$, we say that $\ms\cap (\partial B_r(p)\setminus\Pi(a))$ is \emph{$\eps$-strongly equatorial} for some $\eps>0$ if the following properties hold:
\begin{enumerate} [label={\normalfont(\roman*)}]
    \item each connected component of $\ms$ intersects $\partial B_r(p)\setminus\Pi(a)$;
    \item there exists a plane $\Delta\subset \R^3$ passing through $p$ which is either orthogonal or parallel to $\Pi(a)$ such that each component of $\ms\cap (B_{2r}(p)\setminus B_{r/2}(p))$ is a graph over $\Delta\cap (B_{2r}(p)\setminus B_{r/2}(p))$ with $C^0$ and $C^1$ norm less than $2\eps r$ and $2\eps$ respectively.
\end{enumerate}
\end{definition}

\begin{corollary} \label{cor:14TopInfo}
There exists a universal constant $\mu_0>0$ with the following property.

Let $g$ be a metric on $\{\abs x \le 2\}\cap \Xi(a) \subset \R^3$ sufficiently close to the Euclidean metric, for some $0\ge a \ge -\infty$, and denote by $M^3\eqdef B_1(0)\subset \Xi(a)$ the unit ball with respect to this metric.
Suppose that being orthogonal to $\Pi(a)$ with respect to the Euclidean metric is equivalent to being orthogonal to $\Pi(a)$ with respect to $g$.
Let $\ms^2\subset \amb$ be a compact, connected, embedded surface having free boundary with respect to $\Pi(a)$, with $\partial\ms=\ms\cap \partial\amb$.
Assume that there exist $p\in B_{\mu_0}(0)$  and $0< r\le \mu_0$ so that:
\begin{enumerate} [label={\normalfont(\roman*)}]
    \item $\abs{A_\ms}(x) d_g(x,p) \le \mu_0$ for all $x\in\ms\setminus B_r(p)$; \label{tp:curv}
    \item \label{tp:equat1} $\ms\cap (\partial B_r(p)\setminus\Pi(a))$ is $\mu_0$-strongly equatorial.
\end{enumerate}
Then, if the genus of $\ms\cap B_r(p)$ and the number of connected components of $\ms\cap (\partial B_r(p)\setminus\Pi(a))$ are bounded by $\kappa$, then also $\ms$ has genus and number of connected components of $\ms\cap(\partial B_1(0)\setminus\Pi(a))$ bounded by $\kappa$.
Moreover, if we have also that the number of connected components of $\ms\cap (\Pi(a)\cap B_r(p))$ is bounded by $\kappa$, then the number of connected components of $\ms\cap \Pi(a)$ is bounded by $2\kappa$.
\end{corollary}
\begin{proof}
First observe that, taking $g$ sufficiently close to the Euclidean metric, by \ref{tp:curv} and \ref{tp:equat1} we can assume that $\abs{A_\ms^{\R^3}}(x) \abs{x-p} \le 2\mu_0$ for all $x\in\ms\setminus B_r^{\R^3}(p)$ and that $\ms\cap (\partial B_r^{\R^3}(p)\setminus\Pi(a))$ is $\mu_0$-strongly equatorial.
That being said, and since (as will be clear) the argument we are about to present runs exactly the same in $B_1(0)$ and $B_1^{\R^3}(0)$, for notational convenience we will work with respect to the Euclidean metric (in particular $B_r(p)$, $B_1(0)$ will be balls in $\Xi(a)\subset \R^3$, the Euclidean space). 

By \cref{lem:GradEst}, denoting by $f:\ms\to\R$ the function $f(x)\eqdef\abs{x-p}^2$, for all $x\in\ms\setminus B_r(p)$ we have
\begin{equation*}
\abs{\grad_\ms f(x)} \ge 2(1-2\mu_0)(\abs{x-p}-r) + \min_{y\in\ms\cap \partial B_r(p)}\  \abs{\grad_\ms f(y)} \ge 2(1-2\mu_0)\abs{x-p} > 0\comma
\end{equation*}
where we have used that $\grad_\ms f(x)$ is the orthogonal projection of the vector $2(x-p)$ on $T_x\ms$ and thus we can choose $\mu_0$ small enough that $\abs{\grad_\ms f(y)} \ge 2(1-2\mu_0)r$ for all $y\in\ms\cap (\partial B_r(p)\setminus\Pi(a))$, by definition of $\mu_0$-strongly equatorial.
Hence, we have that
\begin{equation} \label{eq:AlmostTangent}
\begin{split}
\abs{\grad_\ms f(x)-2(x-p)} &= \sqrt{4\abs{x-p}^2 - \abs{\grad_\ms f(x)}^2} \le \sqrt{4\abs{x-p}^2 - 4(1-2\mu_0)^2\abs{x-p}^2} \\
&= 2\abs{x-p}\sqrt{4\mu_0-4\mu_0^2} \le 4 \abs{x-p}\sqrt{\mu_0} \le 8 \sqrt{\mu_0} 
\end{split}
\end{equation}
for all $x\in \ms\setminus B_r(p)$.

Now let us denote by $\Phi(t,x)$ the flow of the vector field $\grad_\ms f$ on $\ms$.
Observe that $\grad_\ms f$ points (strictly) towards $\ms\setminus B_r (p)$ along $\partial B_r (p)\setminus\Pi(a)$ and out of $\ms$ along $\partial B_1(0)\setminus\Pi(a)$ (as long as $\mu_0$ is sufficiently small). Moreover, $\grad_\ms f$ points (strictly) out of $\ms$ along $\Pi(a)\setminus \partial B_r (p)$ if $p\not\in \Pi(a)$ and is parallel to $\Pi(a)$ if $p\in\Pi(a)$ thanks to the free boundary property.
Therefore we can argue similarly to \cref{cor:14PuncturedBall} (keeping in mind that, this time, the exit-time is zero in particular for points in $\ms\cap \partial (\Pi(a)\cap B_r (p))$ to obtain that $\ms$ has genus and number of boundary components in $\partial B_1(0)\setminus B_r (p)$ bounded by $\kappa$.
We only need to show that the number of connected components of $\ms\cap(\partial B_1(0)\setminus \Pi(a))$ is also bounded by $\kappa$.

Since we can work separately on each connected component of $\ms\setminus B_r (p)$ and different connected components correspond to different connected components of $\ms\cap (\partial B_r (p)\setminus\Pi(a))$, for the sake of simplicity we can assume that $\Gamma\eqdef \ms\cap (\partial B_r (p)\setminus\Pi(a))$ consists of only one compact curve (i.e. an $S^1$ or an arc).
Then we want to prove that $\ms\cap (\partial B_1(0)\setminus \Pi(a))$ also consists of a single connected component.

If $\ms\cap (\partial B_1(0)\setminus \Pi(a))$ contains a closed curve (i.e. an $S^1$), then this must be the only connected component of $\ms\cap(\partial B_1(0)\setminus B_r (p))$ and thus we obtain what we want.
Otherwise all components of $\ms\cap (\partial B_1(0)\setminus \Pi(a))$ are arcs with endpoints in $\partial(\Pi(a)\cap B_1(0))$.
Consider one of these arcs and parametrize it with a unit-speed curve $\gamma:\cc 0l\to\partial B_1(0)\setminus \Pi(a)$ in such a way that $\gamma(0),\gamma(l)\in \partial(\Pi(a)\cap B_1(0))$.

Denote by $\beta:\cc 0l\to S^2$ a choice of the unit normal vector field along $\gamma$ orthogonal to $\gamma'$ in $\partial B_1(0)$. Then, thanks to \eqref{eq:AlmostTangent} and for $\mu_0$ sufficiently small, we can assume that $\scal{\beta(t)}{\nu(\gamma(t))}\ge 1/2$ for all $t\in\cc 0l$, where $\nu$ is a choice unit normal to $\ms$.
Therefore, we have that
\[
\begin{split}
\abs{\cov_{\gamma'}^{\partial B_1(0)}{\gamma'}} &= \abs{\scal{\cov_{\gamma'}^{\partial B_1(0)} \gamma'}{\beta}} = \scal{\beta}{\nu(\gamma)}^{-1} \abs{\scal{\cov_{\gamma'}^{\partial B_1(0)} \gamma'}{\nu(\gamma)}} \\
&\le 2(\abs{\scal{\Diff_{\gamma'}\gamma'}{\nu(\gamma)}} + \abs{\scal{\cov_{\gamma'}^{\partial B_1(0)}\gamma' - \Diff_{\gamma'}\gamma'}{\nu(\gamma)}})\le 2(\abs {A_{\ms}} + \abs{\scal{\gamma}{\nu(\gamma)}}) \point
\end{split}
\]
Hence note that we can choose $\mu_0$ small enough that $\abs{\cov_{\gamma'}^{\partial B_1(0)}{\gamma'}}\le \eps_1$, for some constant $\eps_1>0$ to be chosen later.

Therefore, since $\gamma'(0)$ is almost (depending on $\mu_0$) orthogonal to $\partial(\Pi(a)\cap B_1(0))$ by the free boundary condition, $\gamma$ remains close to the geodesic in $\partial B_1(0)\setminus\Pi(a)$ connecting $\gamma(0)$ and the north pole of $\partial B_1(0)$, which is the intersection between the normal to $\Pi(a)$ passing through the origin and $\partial B_1(0)\setminus\Pi(a)$. In particular, given $\eps_2>0$, we can choose $\eps_1$ (and $\mu_0$) in such a way that the $\eps_2$-neighborhood of the north pole of $\partial B_1(0)$ contains $\gamma(t_0)$ for some $t_0\in\cc 0l$.

Now assume by contradiction that $\ms\cap (\partial B_1(0)\setminus\Pi(a))$ consists of two or more connected components. Performing the above argument for two of them, we find two points $y,z\in\ms\cap (\partial B_1(0)\setminus\Pi(a))$ in two different connected components of $\ms\cap (\partial B_1(0)\setminus\Pi(a))$ that are both contained in an $\eps_3$-neighborhood of the north pole of $\partial B_1(0)$.

Consider the points $y',z'\in \ms\cap(\partial B_r (p)\setminus\Pi(a))$ such that $\Phi(t(y'),y') = y$ and $\Phi(t(z'),z') = z$, where $t(x)$ is the exit-time of $x\in\ms\setminus B_r(p)$ as in \cref{cor:14PuncturedBall}. By \eqref{eq:AlmostTangent}, given any constant $\eps_3>0$, we can choose $\mu_0$ sufficiently small such that $\abs{\Phi(t,x) - \hat\Phi(t,x)}\le \eps_3$ for all $x\in\ms$ and $t\le t(x)$, where $\hat\Phi$ is the flow of the vector field $2(x-p)$ in $\R^3$.
Then we have that
\[
\abs{\hat \Phi(t(y'),y') - \hat\Phi(t(z'),z')} \le 2(\eps_2+\eps_3).
\]
Now, choose $\eps_2,\eps_3$ so that $\abs{y'-z'} \le r/20$.
Hence, since $\ms\cap(\partial B_r (p)\setminus\Pi(a))$ consists of a single connected component and it is $\mu_0$-strongly equatorial, there exists an arc $\Gamma'\subset \ms\cap(\partial B_r (p)\setminus\Pi(a))$ of length less than $r/10$ that connects $y'$ and $z'$.
Therefore $\Gamma'' \eqdef \{ \Phi(t(x),x) \st x\in\Gamma'\}$ is a connected subset of $\ms\cap (\partial B_1(0)\setminus\Pi(a))$ containing $y$ and $z$, and contained in a $(1/4)$-neighborhood of the north pole of $\partial B_1(0)$. However, this contradicts the fact that $y$ and $z$ are contained in two different connected components of $\ms\cap (\partial B_1(0)\setminus \Pi(a))$ and thus proves what we wanted. Finally, the last assertion of the statement is easily verified given all the previous information.
\end{proof}

\section{Multiplicity one convergence} \label{sec:Mult1Conv}

In this section we prove that convergence to a surface with multiplicity one well-behaves in presence of isolated singularities, which essentially follows from Allard's regularity theory (see \cite{All72}).

\begin{lemma} \label{lem:MultOneConvExtends}
Let $\ms_j^2\subset\amb^3$ be a sequence of connected free boundary minimal surfaces in a three-dimensional complete Riemannian manifold $\amb$. Moreover let $\ms^2\subset \amb$ be an embedded free boundary minimal surface in $\amb$. Assume that the sequence $\ms_j$ converges locally smoothly to $\ms$, with multiplicity one away from a finite set of points\footnote{Here we mean that for every $x\in\ms$ there exists a neighborhood $U$ of $x$ such that $\ms_j\cap U$ converges graphically smoothly to $\ms\cap U$ with multiplicity one.} $\set$.
Then $\ms_j$ converges locally smoothly to $\ms$ everywhere.
\end{lemma}
\begin{proof}
Let $p$ be a point in $\ms\cap\set$ and let ${r_0}>0$ be sufficiently small such that $B_{4{r_0}}(p)$ does not contain any other point of $\set$ apart from $p$. Fix $\eps>0$ and take ${r_0}$ possibly smaller in such a way that 
\[
\frac{\Haus^2(\ms\cap (B_{2{r_0}}(p)\setminus B_{r_0}(p)))}{\omega_2(4{r_0}^2-{r_0}^2)} < 
\begin{cases}
1+\eps/4 & \text{if $p\in \ms\setminus\partial\ms$}\\
\frac 12(1+\eps/4) & \text{if $p\in\partial\ms$} \point
\end{cases}
\]
Since the convergence of $\ms_j$ to $\ms$ is smooth and graphical in $B_{2{r_0}}(p)\setminus B_{r_0}(p)$, then we can assume that the same estimate holds for every $j$ sufficiently large substituting $\eps/4$ with $\eps/2$. By the extended monotonicity formula (cf. \cite{GuaLiZho17}) this implies that
\[
\frac{\Haus^2(\ms_j\cap B_{r_0}(p))}{\omega_2{r_0}^2} < 
\begin{cases}
1+\eps/2 & \text{if $p\in \ms\setminus\partial\ms$}\\
\frac 12(1+\eps/2) & \text{if $p\in\partial\ms$} \point
\end{cases}
\]
Thus, again by the same monotonicity formula, taking $r_0>0$ possibly smaller, we have
\[
\frac{\Haus^2(\ms_j\cap B_r(p))}{\omega_2r^2} < 
\begin{cases}
1+\eps & \text{if $p\in \ms\setminus\partial\ms$}\\
\frac 12(1+\eps) & \text{if $p\in\partial\ms$}\comma
\end{cases}
\]
for all $r<r_0$ and $j$ sufficiently large.
This concludes the proof by virtue of Theorem 17 in \cite{AmbCarSha18-Compactness}.
\end{proof}

\bibliography{biblio}

\begin{bibdiv}
\begin{biblist}

\bib{All72}{article}{
      author={Allard, William~K.},
       title={On the first variation of a varifold},
        date={1972},
     journal={Ann. of Math. (2)},
      volume={95},
       pages={417\ndash 491},
}

\bib{AmbCarMas18}{book}{
      author={Ambrosio, Luigi},
      author={Carlotto, Alessandro},
      author={Massaccesi, Annalisa},
       title={Lectures on elliptic partial differential equations},
      series={Appunti. Scuola Normale Superiore di Pisa (Nuova Serie) [Lecture
  Notes. Scuola Normale Superiore di Pisa (New Series)]},
   publisher={Edizioni della Normale, Pisa},
        date={2018},
      volume={18},
}

\bib{AmbBuzCarSha18}{article}{
      author={Ambrozio, Lucas},
      author={Buzano, Reto},
      author={Carlotto, Alessandro},
      author={Sharp, Ben},
       title={Bubbling analysis and geometric convergence results for free
  boundary minimal surfaces},
        date={2019},
     journal={Journal de l'École Polytechnique - Mathématiques},
      volume={6},
       pages={621\ndash 664},
}

\bib{AmbCarSha18-Compactness}{article}{
      author={Ambrozio, Lucas},
      author={Carlotto, Alessandro},
      author={Sharp, Ben},
       title={Compactness analysis for free boundary minimal hypersurfaces},
        date={2018},
     journal={Calc. Var. Partial Differential Equations},
      volume={57},
      number={1},
       pages={1\ndash 39},
}

\bib{Bro86}{inproceedings}{
      author={Brothers, J.~E.},
       title={Some open problems in geometric measure theory and its
  applications suggested by participants of the 1984 {AMS} summer institute},
        date={1986},
   booktitle={Proc. {S}ympos. {P}ure {M}ath.},
   publisher={Amer. Math. Soc., Providence, RI},
       pages={441\ndash 464},
}

\bib{Cal67}{article}{
      author={Calabi, Eugenio},
       title={Minimal immersions of surfaces in {E}uclidean spheres},
        date={1967},
     journal={J. Differential Geom.},
      volume={1},
       pages={111\ndash 125},
}

\bib{Car15}{article}{
      author={Carlotto, Alessandro},
       title={Generic finiteness of minimal surfaces with bounded {M}orse
  index},
        date={2017},
     journal={Ann. Sc. Norm. Super. Pisa Cl. Sci. (5)},
      volume={17},
      number={3},
       pages={1153\ndash 1171},
}

\bib{CarChoEic16}{article}{
      author={Carlotto, Alessandro},
      author={Chodosh, Otis},
      author={Eichmair, Michael},
       title={Effective versions of the positive mass theorem},
        date={2016},
     journal={Invent. Math.},
      volume={206},
      number={3},
       pages={975\ndash 1016},
}

\bib{CarLi19}{article}{
      author={Carlotto, Alessandro},
      author={Li, Chao},
       title={Constrained deformations of positive scalar curvature metrics},
     journal={preprint (arXiv:1903.11772)},
}

\bib{ChoKetMax17}{article}{
      author={Chodosh, Otis},
      author={Ketover, Daniel},
      author={Maximo, Davi},
       title={Minimal hypersurfaces with bounded index},
        date={2017},
     journal={Invent. Math.},
      volume={209},
      number={3},
       pages={617\ndash 664},
}

\bib{ChoMax16}{article}{
      author={Chodosh, Otis},
      author={Maximo, Davi},
       title={On the topology and index of minimal surfaces},
        date={2016},
     journal={J. Differential Geom.},
      volume={104},
      number={3},
       pages={399\ndash 418},
}

\bib{ChoMax18}{article}{
      author={Chodosh, Otis},
      author={Maximo, Davi},
       title={On the topology and index of minimal surfaces {II}},
        date={arXiv:1808.06572},
     journal={preprint},
}

\bib{ChoSch85}{article}{
      author={Choi, Hyeong~In},
      author={Schoen, Richard},
       title={The space of minimal embeddings of a surface into a
  three-dimensional manifold of positive {R}icci curvature},
        date={1985},
     journal={Invent. Math.},
      volume={81},
      number={3},
       pages={387\ndash 394},
}

\bib{ColDeL05}{article}{
      author={Colding, Tobias},
      author={De~Lellis, Camillo},
       title={Singular limit laminations, {M}orse index, and positive scalar
  curvature},
        date={2005},
     journal={Topology},
      volume={44},
      number={1},
       pages={25\ndash 45},
}

\bib{ColMin04}{article}{
      author={Colding, Tobias~H.},
      author={Minicozzi~II, William~P.},
       title={The space of embedded minimal surfaces of fixed genus in a
  3-manifold. {IV}. {L}ocally simply connected},
        date={2004},
     journal={Ann. of Math. (2)},
      volume={160},
      number={2},
       pages={573\ndash 615},
}

\bib{ColMin11}{book}{
      author={Colding, Tobias~H.},
      author={Minicozzi~II, William~P.},
       title={A course in minimal surfaces},
      series={Graduate Studies in Mathematics},
   publisher={Amer. Math. Soc.},
     address={Providence, RI},
        date={2011},
      volume={121},
}

\bib{EjiMic08}{article}{
      author={Ejiri, N.},
      author={Micallef, M.},
       title={Comparison between second variation of area and second variation
  of energy of a minimal surface},
        date={2008},
     journal={Adv. Calc. Var.},
      volume={1},
      number={3},
       pages={223\ndash 239},
}

\bib{Fis85}{article}{
      author={Fischer-Colbrie, Doris},
       title={On complete minimal surfaces with finite {M}orse index in
  three-manifolds},
        date={1985},
     journal={Invent. Math.},
      volume={82},
      number={1},
       pages={121\ndash 132},
}

\bib{FisSch80}{article}{
      author={Fischer-Colbrie, Doris},
      author={Schoen, Richard},
       title={The structure of complete stable minimal surfaces in 3-manifolds
  of nonnegative scalar curvature},
        date={1980},
     journal={Comm. Pure Appl. Math.},
      volume={33},
      number={2},
       pages={199\ndash 211},
}

\bib{FraLi14}{article}{
      author={Fraser, Ailana},
      author={Li, Martin},
       title={Compactness of the space of embedded minimal surfaces with free
  boundary in three-manifolds with nonnegative {R}icci curvature and convex
  boundary},
        date={2014},
     journal={J. Differential Geom.},
      volume={96},
      number={2},
       pages={183\ndash 200},
}

\bib{FraSch11}{article}{
      author={Fraser, Ailana},
      author={Schoen, Richard},
       title={The first {S}teklov eigenvalue, conformal geometry, and minimal
  surfaces},
        date={2011},
     journal={Adv. Math.},
      volume={226},
      number={5},
       pages={4011\ndash 4030},
}

\bib{FraSch13}{incollection}{
      author={Fraser, Ailana},
      author={Schoen, Richard},
       title={Minimal surfaces and eigenvalue problems},
        date={2013},
   booktitle={{G}eometric analysis, mathematical relativity, and nonlinear
  partial differential equations, 105--121, {C}ontemp. {M}ath., 599, {A}mer.
  {M}ath. {S}oc., {P}rovidence, {RI}},
      series={Contemp. Math.},
      volume={599},
   publisher={Amer. Math. Soc., Providence, RI},
}

\bib{FraSch16}{article}{
      author={Fraser, Ailana},
      author={Schoen, Richard},
       title={Sharp eigenvalue bounds and minimal surfaces in the ball},
        date={2016},
     journal={Invent. Math.},
      volume={203},
      number={3},
       pages={823\ndash 890},
}

\bib{GroLaw80a}{article}{
      author={Gromov, M.},
      author={Lawson, H.~B.},
       title={Spin and scalar curvature in the presence of a fundamental group.
  {I}},
        date={1980},
     journal={Ann. of Math. (2)},
      volume={111},
      number={2},
       pages={209\ndash 230},
}

\bib{GuaLiWanZho19}{article}{
      author={Guang, Qiang},
      author={Li, Martin},
      author={Wang, Zhichao},
      author={Zhou, Xin},
       title={Min-max theory for free boundary minimal hypersurfaces {II} -
  {G}eneral {M}orse index bounds and applications},
     journal={preprint (arXiv:1907.12064)},
}

\bib{GuaLiZho17}{article}{
      author={Guang, Qiang},
      author={Li, Martin},
      author={Zhou, Xin},
       title={Curvature estimates for stable free boundary minimal
  hypersurfaces},
     journal={J. Reine Angew. Math. (\emph{to appear})},
}

\bib{GuaWanZho19}{article}{
      author={Guang, Qiang},
      author={Wang, Zhichao},
      author={Zhou, Xin},
       title={Compactness and generic finiteness for free boundary minimal
  hypersurfaces ({I})},
     journal={preprint (arXiv:1803.01509)},
}

\bib{GulLaw86}{incollection}{
      author={Gulliver, Robert},
      author={Lawson, H.~Blaine},
       title={The structure of stable minimal hypersurfaces near a
  singularity},
        date={1986},
   booktitle={Geometric measure theory and the calculus of variations
  ({A}rcata, {C}alif., 1984), 213--237, {P}roc. {S}ympos. {P}ure {M}ath., 44,
  {A}mer. {M}ath. {S}oc., {P}rovidence, {RI}},
      series={Proc. Sympos. Pure Math},
      volume={44},
   publisher={Amer. Math. Soc., Providence, RI},
}

\bib{Har64}{article}{
      author={Hartman, Philip},
       title={Geodesic parallel coordinates in the large},
        date={1964},
     journal={Amer. J. Math.},
      volume={86},
       pages={705\ndash 727},
}

\bib{HasNorRub03}{article}{
      author={Hass, J.},
      author={Norbury, P.},
      author={Rubinstein, J.~H.},
       title={Minimal spheres of arbitrarily high {M}orse index},
        date={2003},
     journal={Comm. Anal. Geom.},
      volume={11},
      number={3},
       pages={425\ndash 439},
}

\bib{Hsi83}{article}{
      author={Hsiang, W.-Y.},
       title={Minimal cones and the spherical {B}ernstein problem. {I}},
        date={1983},
     journal={Ann. of Math. (2)},
      volume={118},
      number={1},
       pages={61\ndash 73},
}

\bib{IriMarNev18}{article}{
      author={Irie, K.},
      author={Marques, F.},
      author={Neves, A.},
       title={Density of minimal hypersurfaces for generic metrics},
        date={2018},
     journal={Ann. of Math. (2)},
      volume={187},
      number={3},
       pages={963\ndash 972},
}

\bib{KapLi17}{article}{
      author={Kapouleas, Nikolaos},
      author={Li, Martin},
       title={Free boundary minimal surfaces in the unit three-ball via
  desingularization of the critical catenoid and the equatorial disk},
     journal={preprint (arXiv:1709.08556)},
}

\bib{KapWiy17}{article}{
      author={Kapouleas, Nikolaos},
      author={Wiygul, David},
       title={Free-boundary minimal surfaces with connected boundary in the
  $3$-ball by tripling the equatorial disc},
     journal={preprint (arXiv:1711.00818)},
}

\bib{Ket16equiv}{article}{
      author={Ketover, Dan},
       title={Equivariant min-max theory},
     journal={preprint (arXiv:1612.08692)},
}

\bib{Ket16fb}{article}{
      author={Ketover, Dan},
       title={Free boundary minimal surfaces of unbounded genus},
     journal={preprint (arXiv:arXiv:1612.08691)},
}

\bib{Li19}{article}{
      author={Li, Martin},
       title={Free boundary minimal surfaces in the unit ball: recent advances
  and open questions},
     journal={preprint (arXiv:1907.05053)},
}

\bib{LiZho16}{article}{
      author={Li, Martin},
      author={Zhou, Xin},
       title={Min-max theory for free boundary minimal hypersurfaces {I} -
  regularity theory},
     journal={J. Differential Geom. (\emph{to appear})},
}

\bib{Lim17}{article}{
      author={Lima, Vanderson},
       title={Bounds for the {M}orse index of free boundary minimal surfaces},
     journal={preprint (arXiv:1710.10971)},
}

\bib{MeePerRos16}{article}{
      author={Meeks, William~H.},
      author={P\'erez, Joaqu\'in},
      author={Ros, Antonio},
       title={Local removable singularity theorems for minimal laminations},
        date={2016},
     journal={J. Differential Geom.},
      volume={103},
      number={2},
       pages={319\ndash 362},
}

\bib{Mia02}{article}{
      author={Miao, Pengzi},
       title={Positive mass theorem on manifolds admitting corners along a
  hypersurface},
        date={2002},
     journal={Adv. Theor. Math. Phys.},
      volume={6},
      number={6},
       pages={1163\ndash 1182},
}

\bib{Sch83_estimates}{incollection}{
      author={Schoen, Richard},
       title={Estimates for stable minimal surfaces in three dimensional
  manifolds},
        date={1983},
   booktitle={Seminar on minimal submanifolds, 111--126, {A}nn. of {M}ath.
  {S}tud., 103, {P}rinceton {U}niv. {P}ress, {P}rinceton, {NJ}},
      series={Ann. of Math. Stud.},
      volume={103},
   publisher={Princeton Univ. Press, Princeton, NJ},
}

\bib{Sch83_uniqueness}{article}{
      author={Schoen, Richard},
       title={Uniqueness, symmetry, and embeddedness of minimal surfaces},
        date={1983},
     journal={J. Differential Geom.},
      volume={18},
      number={4},
       pages={791\ndash 809},
}

\bib{SchSim81}{article}{
      author={Schoen, Richard},
      author={Simon, Leon},
       title={Regularity of stable minimal hypersurfaces},
        date={1981},
     journal={Comm. Pure Appl. Math.},
      volume={34},
      number={6},
       pages={741\ndash 797},
}

\bib{SchYau83}{article}{
      author={Schoen, Richard},
      author={Yau, Shing-Tung},
       title={The existence of a black hole due to condensation of matter},
        date={1983},
     journal={Comm. Math. Phys.},
      volume={90},
      number={4},
       pages={575\ndash 579},
}

\bib{SchYau17}{article}{
      author={Schoen, Richard},
      author={Yau, Shing-Tung},
       title={Positive {S}calar {C}urvature and {M}inimal {H}ypersurface
  {S}ingularities},
        date={arXiv:1704.05490},
     journal={preprint},
}

\bib{ShiShiTan03}{book}{
      author={Shiohama, Katsuhiro},
      author={Shioya, Takashi},
      author={Tanaka, Minoru},
       title={The geometry of total curvature on complete open surfaces},
      series={Cambridge Tracts in Mathematics},
   publisher={Cambridge University Press, Cambridge},
        date={2003},
      volume={159},
}

\bib{Vol15}{thesis}{
      author={Volkmann, Alexander},
       title={Free boundary problems governed by mean curvature},
        type={Ph.D. Thesis},
        date={2015},
}

\bib{Whi86}{article}{
      author={White, Brian},
       title={Complete surfaces of finite total curvature},
        date={1987},
     journal={J. Differential Geom.},
      volume={26},
      number={2},
       pages={315\ndash 326},
}

\bib{Whi87}{article}{
      author={White, Brian},
       title={Curvature estimates and compactness theorems in 3-manifolds for
  surfaces that are stationary for parametric elliptic functionals},
        date={1987},
     journal={Invent. Math.},
      volume={88},
      number={2},
       pages={243\ndash 256},
}

\bib{Whi09}{article}{
      author={White, Brian},
       title={Which ambient spaces admit isoperimetric inequalities for
  submanifolds?},
        date={2009},
     journal={J. Differential Geom.},
      volume={83},
      number={1},
       pages={213\ndash 228},
}

\bib{Whi16}{incollection}{
      author={White, Brian},
       title={Introduction to minimal surface theory},
        date={2016},
   booktitle={Geometric analysis, 387--438, {IAS}/{P}ark {C}ity {M}ath. {S}er.,
  22, {A}mer. {M}ath. {S}oc., {P}rovidence, {RI}},
      volume={22},
   publisher={Amer. Math. Soc., Providence, RI},
}

\end{biblist}
\end{bibdiv}
\end{spacing}

\end{document}